\documentclass[a4paper,12pt,reqno]{amsart}

\newif\ifpreview
\previewfalse
\usepackage{xcolor}

\allowdisplaybreaks

\usepackage[margin=2.6cm]{geometry}

\usepackage{adjustbox}

\usepackage{enumerate}
\usepackage{amssymb, amsmath, bm}
\usepackage{mathrsfs}
\usepackage{amscd}
\usepackage[active]{srcltx}
\usepackage{verbatim}
\usepackage[colorlinks,linkcolor={black},citecolor={black},urlcolor={black}]{hyperref}

\usepackage[latin1]{inputenc}
\usepackage{tikz}
\usetikzlibrary{trees}

\usepackage{mhfig}
\usepackage{mathrsfs}
\def\1{{\mhpastefig{root}}}
\def\2{{\mhpastefig[2/3]{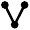}}}
\def\3{{\!\mhpastefig[1/2]{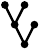}}}
\def\9{{\mhpastefig[1/2]{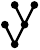}}}
\def\4{{\mhpastefig[1/2]{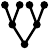}}}
\def\5{{\!\mhpastefig[1/2]{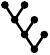}}}
\def\6{{\!\mhpastefig[1/2]{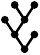}}}
\def\8{{\mhpastefig[1/2]{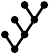}}}
\def\7{{\mhpastefig[1/2]{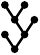}}}
\def\10{{\mhpastefig[1/2]{tree}}}

\renewcommand{\aa}{{\boldsymbol\alpha}}

\theoremstyle{plain}
\newtheorem{theorem}{Theorem}[section]
\newtheorem{corollary}[theorem]{Corollary}
\newtheorem{lemma}[theorem]{Lemma}
\newtheorem{proposition}[theorem]{Proposition}

\newtheorem{definition}[theorem]{Definition}
\newtheorem{assumption}[theorem]{Assumption}

\newtheorem*{definition*}{Definition}

\theoremstyle{remark}
\newtheorem{remark}[theorem]{Remark}
\newtheorem{example}[theorem]{Example}

\newtheorem*{claim*}{Claim}
\newtheorem*{remark*}{Remark}
\newtheorem*{example*}{Example}
\newtheorem*{notation*}{Notation}

\numberwithin{equation}{section}

\def\N{{\mathbb N}}

\def\R{{\mathbb R}}
\def\C{{\mathbb C}}

\def\M{{\mathbb M}}
\def\H{{\mathbb H}}

 \newcommand{\Dens}{{\mathfrak P}}

\newcommand{\one}{{{\bf 1}}}

\newcommand{\sK}{\mathscr{K}}

\newcommand{\eps}{\varepsilon}
\renewcommand{\phi}{\varphi}

\newcommand{\Cln}{\mathfrak{C}^n}

\newcommand{\dd}{\; \mathrm{d}}

\newcommand{\cC}{\mathcal{C}}
\newcommand{\cone}{\cP}

\newcommand{\ot}{\otimes}

\DeclareMathOperator{\Tr}{Tr}

\DeclareMathOperator{\Ran}{\mathsf{Ran}}

\DeclareMathOperator{\Ker}{\mathsf{Ker}}
\DeclareMathOperator{\lin}{lin}

\newcommand{\ip}[1]{\langle {#1}\rangle}

\newcommand{\bip}[1]{\big\langle {#1}\big\rangle}

\newcommand{\abs}[1]{\vert {#1}\vert}

\DeclareMathOperator{\Ric}{Ric}

\DeclareMathOperator{\Hess}{Hess}

\DeclareMathOperator{\spec}{sp}
\DeclareMathOperator{\Ent}{Ent}
\newcommand{\rmD}{\mathrm{D}}
\newcommand{\KMS}[1]{\langle {#1}\rangle_{L_{\rm KMS}^{2}(\sigma)}}
\newcommand{\KMSj}[1]{\langle {#1}\rangle_{L_{{\rm KMS}, j}^{2}(\sigma)}}

\def\tand{\quad{\rm and}\quad}

\newcommand{\ddt}{\frac{\mathrm{d}}{\mathrm{d}t}}
\newcommand{\ddtr}{\frac{\mathrm{d}^+}{\mathrm{d}t}}

\newcommand{\cQ}{\mathcal{Q}}
\newcommand{\cN}{\mathcal{N}}
\newcommand{\cM}{\mathscr{M}}
\newcommand{\cH}{\mathcal{H}}
\newcommand{\cB}{\mathcal{B}}
\newcommand{\cD}{\mathcal{D}}
\newcommand{\cW}{\mathscr{W}}
\newcommand{\cI}{\mathcal{I}}
\newcommand{\cL}{\mathscr{L}}

\newcommand{\Proj}{\mathsf{P}}
\newcommand{\sH}{\mathscr{H}}

\newcommand{\cG}{\mathscr{G}}
\newcommand{\cK}{\mathscr{K}}
\newcommand{\cJ}{\mathcal{J}}
\newcommand{\cX}{\mathcal{X}}
\newcommand{\cT}{\mathcal{T}}

\newcommand{\cA}{\mathcal{A}}
\newcommand{\sA}{\mathscr{A}}

\newcommand{\cU}{\mathcal{U}}

\newcommand{\sfL}{\mathsf{L}}
\newcommand{\sfR}{\mathsf{R}}
\newcommand{\sE}{\mathscr{E}}

\newcommand{\bbA}{\mathbf{A}}
\newcommand{\bbB}{\mathbf{B}}
\newcommand{\bbC}{\mathbf{C}}

\DeclareMathOperator{\dive}{div}

\newcommand{\cF}{\mathcal{F}}

\newcommand{\cP}{\mathscr{P}}
\newcommand{\PX}{\cP(\cX)}

\newcommand{\hrho}{\widehat\rho}

\newcommand{\hrhop}{\widehat{\rho_{j}}}

\newcommand{\lrho}{\check{\rho}}

\renewcommand{\tilde}{\widetilde}

\newcommand{\BB}{\mathbf{B}}

\newcommand{\sigmax}{\left[\begin{array}{cc}0 & 1 \\1 & 0\end{array}\right]}
\newcommand{\sigmay}{\left[\begin{array}{cc}0 & -i \\i & 0\end{array}\right]}
\newcommand{\sigmaz}{\left[\begin{array}{cc}1 & 0 \\0 & -1\end{array}\right]}

\DeclareMathOperator{\HcWI}{H{\cW}I}
\DeclareMathOperator{\mLSI}{MLSI}
\DeclareMathOperator{\mTal}{T_\cW}
\DeclareMathOperator{\Poinc}{P}
\newcommand{\ddtrz}{\left.\frac{\mathrm{d}^+}{\mathrm{d}t}\right\vert_{t=0}}

\newcommand{\treelr}[3]
{
\begin{adjustbox}{trim=0.1cm 0.2cm 0.1cm 0.9cm}
\begin{tikzpicture}
\tikzstyle{level 1}=[level distance=0.25cm, sibling distance=1.3cm]
\tikzstyle{level 2}=[level distance=0.25cm, sibling distance=1.3cm]
  \node {\tiny{$#1$}} [grow'=up] 
    child {node {\tiny{$#2$}}}
    child { node {\tiny{$#3$}}
    }; 
\end{tikzpicture}
\end{adjustbox}
}

\newcommand{\widetreelr}[3]
{
\begin{adjustbox}{trim=0.1cm 0.25cm 0.1cm 0.0cm}
\begin{tikzpicture}
\tikzstyle{level 1}=[level distance=0.25cm, sibling distance=1.7cm]
\tikzstyle{level 2}=[level distance=0.25cm, sibling distance=1.7cm]
  \node {\tiny{$#1$}} [grow'=up] 
    child {node {\tiny{$#2$}}}
    child { node {\tiny{$#3$}}
    }; 
\end{tikzpicture}
\end{adjustbox}
}

\newcommand{\treellr}[5]
{
\begin{adjustbox}{trim=0.0cm 0.3cm 0.0cm 0.0cm}
\begin{tikzpicture}
\tikzstyle{level 1}=[level distance=0.25cm, sibling distance=1.3cm]
\tikzstyle{level 2}=[level distance=0.25cm, sibling distance=1.3cm]
  \node {\tiny{$#1$}} [grow'=up] 
    child {node {\tiny{$#3$}} 
      child {node {\tiny{$#4$}}} 
      child {node {\tiny{$#5$}}}} 
    child { node {\tiny{$#2$}}
    }; 
\end{tikzpicture}
\end{adjustbox}
}

\newcommand{\widetreellr}[5]
{
\begin{adjustbox}{trim=0.1cm 0.37cm 0.1cm 0.0cm}
\begin{tikzpicture}
\tikzstyle{level 1}=[level distance=0.25cm, sibling distance=1.7cm]
\tikzstyle{level 2}=[level distance=0.25cm, sibling distance=1.7cm]
  \node {\tiny{$#1$}} [grow'=up] 
    child {node {\tiny{$#3$}} 
      child {node {\tiny{$#4$}}} 
      child {node {\tiny{$#5$}}}} 
    child { node {\tiny{$#2$}}
    }; 
\end{tikzpicture}
\end{adjustbox}
}

\newcommand{\treelllr}[7]
{
\begin{adjustbox}{trim=0.0cm 0.5cm 0.0cm 0.0cm}
\begin{tikzpicture}
\tikzstyle{level 1}=[level distance=0.25cm, sibling distance=1.3cm]
\tikzstyle{level 2}=[level distance=0.25cm, sibling distance=1.3cm]
  \node {\tiny{$#1$}} [grow'=up] 
    child {node {\tiny{$#2$}} 
		      child { node {\tiny{$#3$}} 
		      			      child {node {\tiny{$#4$}}} 
						      child {node {\tiny{$#5$}}} 
		      		} 
		      child { node {\tiny{$#6$}} } 
      	  } 
    child { node {\tiny{$#7$}} }; 
\end{tikzpicture}
\end{adjustbox}
}

\newcommand{\treelrr}[5]
{
\begin{adjustbox}{trim=0.0cm 0.3cm 0.0cm 0.0cm}
\begin{tikzpicture}
\tikzstyle{level 1}=[level distance=0.25cm, sibling distance=1.3cm]
\tikzstyle{level 2}=[level distance=0.25cm, sibling distance=1.3cm]
  \node {\tiny{$#1$}} [grow'=up] 
    child {node {\tiny{$#2$}}} 
    child {node {\tiny{$#3$}} 
      child {node {\tiny{$#4$}}} 
      child {node {\tiny{$#5$}}} 
    }; 
\end{tikzpicture}
\end{adjustbox}
}

\newcommand{\widetreelrr}[5]
{
\begin{adjustbox}{trim=0.1cm 0.37cm 0.1cm 0.0cm}
\begin{tikzpicture}
\tikzstyle{level 1}=[level distance=0.25cm, sibling distance=1.7cm]
\tikzstyle{level 2}=[level distance=0.25cm, sibling distance=1.7cm]
  \node {\tiny{$#1$}} [grow'=up] 
    child {node {\tiny{$#2$}}} 
    child {node {\tiny{$#3$}} 
      child {node {\tiny{$#4$}}} 
      child {node {\tiny{$#5$}}} 
    }; 
\end{tikzpicture}
\end{adjustbox}
}

\newcommand{\err}{{r}}

\definecolor{jan}{rgb}{0.0,0.3,0.8}

\newcommand{\phys}[2]{\,|#1\rangle\langle#2|\,}

\begin{document}

      \ifpreview {
          \definecolor{personal}{rgb}{0.95, 0.92, 0.88}
                  \pagecolor{personal}
         }
      \fi

\title
[Quantum transport metrics]
{Non-commutative calculus, optimal transport and functional inequalities in dissipative quantum systems}
\author{Eric A. Carlen}
\address{Department of Mathematics\\ 
Hill Center\\
Rutgers University\\
110 Frelinghuysen Road\\
Piscataway\\
NJ 08854-8019\\
USA}
\email{carlen@math.rutgers.edu}

\author{Jan Maas}
\address{
Institute of Science and Technology Austria (IST Austria)\\
Am Campus 1\\ 
3400 \newline Klos\-ter\-neu\-burg\\ 
Austria}
\email{jan.maas@ist.ac.at}


 \begin{abstract}
We study dynamical optimal transport metrics between density matrices associated to symmetric Dirichlet forms on finite-dimensional $C^*$-algebras. 
Our setting covers arbitrary skew-derivations and it provides a unified framework that simultaneously generalizes recently constructed transport metrics for Markov chains, Lindblad equations, and the Fermi Ornstein--Uhlenbeck semigroup. 
We develop a non-nommutative differential calculus that allows us to obtain non-commutative Ricci curvature bounds, logarithmic Sobolev inequalities, transport-entropy inequalities, and spectral gap estimates.
 \end{abstract} 


\maketitle
 
\setcounter{tocdepth}{1}
{\footnotesize\renewcommand{\contentsname}{}\vspace*{-2em} \tableofcontents}

\section{Introduction}
\label{sec:intro}

In the context of diffusion semigroups, a great deal of recent progress has been made based on two different gradient flow interpretations of the heat flow, namely as
\begin{enumerate}
\item the gradient flow of the Dirichlet energy in $L^2$;
\item the gradient flow of the Boltzmann entropy in the space of probability measures endowed with the $2$-Kantorovich metric.
\end{enumerate}
In this paper we study the analogs of (1) and (2) for non-commutative probability, in the setting von Neumann algebras, and we establish the equivalence  of (1) and (2)  in this setting. 
This naturally involves the construction of non-commutative analogs of the $2$-Kantorovich metric, a topic that was investigated in our earlier papers \cite{CM12,2016-Carlen-Maas} and in the independent work \cite{Mie12,MiMi17}. 
Recently the subject received the attention of a number of authors; see \cite{CGGT17,CGT16} for noncommutative transport metrics, \cite{RD17a,RD17,GJLR18} for functional inequalities, and \cite{Horn18,Wirt18} for results in infinite dimensions. 
We refer to \cite{GMP16,BrVo18} for different non-commutative variants of the $2$-Kantorovich metric in other contexts.

Our focus in this paper is on developing the relations between (1) and (2) in the non-commutative setting with the aim of proving functional inequalities relevant to the study of the rate of approach to equilibrium for quantum Markov semigroups, in close analogy with what has been accomplished along these lines in the classical setting in recent years.  

In order not to obscure the main ideas we shall work in a finite-dimensional setting and postpone the infinite-dimensional extension to a future work. The finite-dimensional case is of direct interest in quantum information theory, and the essential aspects of our new results are interesting even in this setting where they can be explained to a wider audience that is not thoroughly familiar with the Tomita--Takesaki theory. We now briefly describe the content of the paper. Any unfamiliar terminology is explained in the next subsection, but hopefully many readers will not need to look ahead.

The central object of study in this paper is a quantum Markov semigroup (QMS)  $(\cP_t)_{t>0}$ on $\cA$, a finite-dimensional $C^*$-algebra containing the identity $\one$. That is, for each $t$, $\cP_t\one = \one$ and $\cP_t$ is completely positive. The generators ${\cL}$  of such semigroups have been characterized in \cite{Lin76,GKS76}.

We are concerned with the case in which there is a unique faithful  invariant state $\sigma$  for the dual semigroup; i.e., $\cP_t^\dagger \sigma = \sigma$ for all $t$. The paper \cite{SpLe77} is an excellent source for the physical context and makes it clear that assuming that the invariant state $\sigma$ is tracial, which we do not do, would preclude a great many physical applications. 
Let $\Dens_+$ denote the space of faithful states. We would like to know, for instance, when there is a Riemannian metric $g$  on $\Dens_+$ such that the flow  on $\Dens_+$ given by the dual semigroup $(\cP_t^\dagger)_{t>0}$ is the gradient flow driven by the relative entropy functional
$\Ent_\sigma(\rho) = \Tr[\rho(\log \rho - \log \sigma)]$ with respect to the Riemannian metric.  In \cite{2016-Carlen-Maas,MiMi17}, it is shown that  when each $\cP_t$ is self-adjoint with respect to the Gelfand-Naimark-Segal (GNS) inner product induced on $\cA$ by $\sigma$,
this is the case. We constructed the metric using ideas from optimal mass transport, and showed that, 
as in the classical case, the framework provided an efficient means for proving functional inequalities. 
This has been taken up and further developed by other authors, in particular Rouz\'e and Datta \cite{RD17a,RD17}.
As in the classical case, Ricci curvature bounds are essential for the framework to be used to obtain functional 
inequalities. 
As shown in \cite{2016-Carlen-Maas,RD17}, once one has Ricci curvature bounds, a host of functional 
inequalities follow. A central problem then is to prove such bounds. A main contribution of the present paper is a
flexible framework for doing this. It turns out that there are many ways to write a given QMS generator 
$\cL$ (that is self-adjoint in the GNS sense)  in ``divergence form'' for non-commutative derivatives. Each of the different ways of doing this can be associated to a Riemannian metric on $\Dens_+$. Different ways of writing  $\cL$ 
in divergence form may have advantages over others, for example in proving Ricci curvature bounds. 
Hence it is important to have as much flexibility here as possible.  We shall use this flexibility to give new 
examples in which we can obtain sharp Ricci  curvature  bounds. The machinery is useful for other 
functionals and other flows; the methods of this paper are not by any means restricted to gradient flow for 
relative entropy, despite our focus on this example here in the introduction. 

An interesting problem remains:
For each way of writing $\cL$ in divergence form, we have a Riemannian metric. The formulas 
are different, but in principle, all of the metrics might be the same. That is, they might all be determined by 
$\cL$, and not the particular way of writing in divergence form, even though doing this one way or 
another may facilitate certain computations.  

The problem of writing QMS as gradient flow for the relative entropy was also taken up independently by Mittnenzweig and Mielke \cite{MiMi17}, and although their framework is somewhat different, their approach also works in the case that each $\cP_t$ is self-adjoint with respect to the GNS inner product induced on $\cA$ by $\sigma$. 
Here, we shall show that if $(\cP_t)_{t\geq 0}$ can be written as gradient flow for $\Ent_\sigma$ with respect to some continuously differentiable Riemannian metric, then each $\cP_t$ is necessarily self-adjoint with respect to another inner product associated to $\sigma$, the Boguliobov-Kubo-Mori (BKM) inner product. 
As we show, the class of QMS with this self-adjointness property is strictly larger than the class of QMS with the GNS self-adjointness property. Thus, there is at present an interesting gap between the known necessary condition for the construction of the Riemannian metric, and the known sufficient condition. 
Of course, in the classical setting, the two notions of self-adjointness coincide, and one has a pleasing characterization of reversible Markov chains in terms of gradient flow \cite{D15}.

\subsection{Notation}

Let $\cA$ be finite-dimensional $C^*$-algebra containing the identity $\one$. In the finite-dimensional setting, all topologies one might impose on $\cA$ are equivalent, and $\cA$ is also a von Neumann algebra. In particular, it is generated by the projections it contains. We may regard any such algebra as a $*$-subalgebra of $\M_n(\C)$, the set of all complex $n \times n$ matrices. 
Let $\cA_h$ be the subset of hermitian elements in $\cA$, and let $\cA_+ \subseteq \cA$ denote the class of elements that are positive definite (i.e., $\spec(A) \subseteq (0,\infty)$ for $A \in \cA_+$. For $\cA = \M_n(\C)$ we write $\cA_+ = \M_n^+(\C)$.

Throughout this section we fix a positive linear functional $\tau$ on $\cA$ that is \emph{tracial} (i.e., $\tau[AB] = \tau[BA]$ for all $A, B \in \cA$) and \emph{faithful} (i.e., $A=0$ whenever $\tau[A^*A] =0$). Under these assumptions, $\tau$ induces a scalar product on $\cA$ given by $\ip{A,B}_{L^2(\cA,\tau)} = \tau[A^* B]$ for $A, B \in \cA$.
In our applications, $\tau$ will often be the usual trace $\Tr$ on $\M_n(\C)$ in which case the scalar product is the Hilbert--Schmidt scalar product, but it will be useful to include different situations, e.g., the trace induced by a non-uniform probability measure on a finite set.

A \emph{state} on $\cA$ is a positive linear functional $\phi$ on $\cA$ such that $\varphi(\one) =1$. 
If $\varphi$ is a state, there is a uniquely determined $\sigma\in \cA$ such that 
	$\varphi(A) = \tau[\sigma A]$ 
for all $A \in \cA$. 
Note that $\sigma$ is a \emph{density matrix}; i.e., it is positive semidefinite and $\tau[\sigma] =1$. 
Let $\Dens(\cA)$ denote the set of density matrices.
We write $\Dens_+(\cA) = \{ \rho \in 
\Dens(\cA) : \rho \text{ is positive definite} \}$.  We will simply write $\Dens = \Dens(\cA)$ and $\Dens_+ 
= \Dens_+(\cA)$ if the algebra $\cA$ is clear from the context.

We always use $\dagger$ to denote the adjoint of a linear transformation on $\cA$ with respect to the scalar product $\ip{\cdot, \cdot}_{L^2(\cA,\tau)}$. If $\sK$ is such a linear transformation,
\begin{equation}\label{dagger}
\langle A, \sK B\rangle_{L^2(\cA,\tau)} = \langle \sK^\dagger A  ,   B\rangle_{L^2(\cA,\tau)}  \ .
\end{equation}

Though we suppose no familiarity with the Tomita--Takesaki Theory of standard forms of von Neumann algebras, we will make use of the so-called modular and relative modular operators that arise there. In our setting, these operators have a simple direct definition:

\begin{definition}[The relative modular operator] \label{modular} Let $\sigma,\rho\in \Dens_+$. The corresponding  {\em relative modular operator} $\Delta_{\sigma,\rho}$ is the linear transformation on $\cA$ defined by 
\begin{equation}\label{modef}
\Delta_{\sigma,\rho}(A) = \sigma A \rho^{-1}\ .
\end{equation}
The {\em modular operator}  corresponding to $\sigma$, $\Delta_\sigma$, is  defined by $\Delta_{\sigma} := 
\Delta_{\sigma,\sigma}$.
\end{definition}

Since $\langle B, \Delta_{\sigma,\rho} A\rangle_{L^2(\cA,\tau)} = \tau[(\sigma^{1/2}B\rho^{-1/2})^*(\sigma^{1/2}A\rho^{-1/2})]$ for all $A,B\in \cA$, the operator $\Delta_{\sigma,\rho}$ is positive definite on $L^2(\cA,\tau)$. 
In case that $\tau$ is the restriction of the usual trace $\Tr$ to $\cA \subseteq \M_n(\C)$, the operators $\sigma$ and $\rho$ are also positive density matrices in $\M_n(\C)$, and the same computations are valid for all $A,B\in \M_n(\C)$. We may regard $\Delta_\sigma$ as an operator on $\M_n(\C)$, equipped with the Hilbert--Schmidt inner product, and then, so extended, it is still positive definite. 

We are interested in evolution equations on $\Dens_+(\cA)$ that correspond to  forward Kolmogorov equations for ergodic Markov processes satisfying a detailed balance condition, or in other words a reversibility condition, with respect to their unique invariant probability measure. 
Before presenting our results, we introduce the class of  quantum Markov semigroups satisfying a detailed balance condition that are the focus of our investigation.

\section{Quantum Markov semigroups with detailed balance}
\label{sec:QMS}

Let $\cA \subseteq B(\sH)$ be a $C^*$-algebra of operators acting on a finite-dimensional Hilbert space $\sH$. 
Let $\tau$ be a tracial and faithful positive linear functional on $\cA$.
A \emph{quantum Markov semigroup} on $\cA$ is a $C_0$-semigroup of operators $(\cP_t)_{t \geq 0}$ acting on $\cA$, satisfying
\begin{enumerate}
	\item \label{item:Markov1} $\cP_t \one = \one$;
	\item \label{item:Markov2} $\cP_t$ is \emph{completely positive}, i.e., $\cP_t \otimes I_{\M_n(\C)}$ is a positivity preserving operator on $\cA \otimes \M_n(\C)$ for all $n \in \N$.
\end{enumerate}
Note that \eqref{item:Markov2} implies that $\cP_{t}$ is \emph{real}, i.e., $ (\cP_t A)^* = \cP_t A^*$ for all $A \in \cA$.
Let $\cP_t^\dagger$ be the Hilbert--Schmidt adjoint of $\cP_t$ satisfying $\tau[A^* \cP_t^\dagger B] = \tau[(\cP_t A)^*B]$ for all $A, B \in \cA$. It follows that $\cP_t^\dagger$  is trace-preserving and completely positive.

It is well known \cite{GKS76,Lin76} that the generator $\cL$ of the semigroup $\cP_t = e^{t \cL}$ can be written in \emph{Lindblad form}
\begin{align}\label{eq:Lindblad-form}
  \cL A & =  i [\widetilde{H},A] + \sum_{j \in \cJ}  V_j^* [A, V_j] +  [V_j^* , A] V_j \ , \\
  \cL^\dagger \rho & = - i [\widetilde{H},\rho] + \sum_{j \in \cJ}  [V_j, \rho V_j^*] +  [V_j \rho, V_j^*] \ ,
 \end{align}
where $\cJ$ is a finite index set, $V_j \in B(\sH)$ (not necessarily belonging to $\cA$) for all $j \in \cJ$, and the Hamiltonian $\widetilde{H} \in B(\sH)$ is self-adjoint.

\subsection{Detailed balance}
\label{sec:detailed-balance}
The starting point of our investigations is the assumption that $(\cP_t)_{t \geq 0}$ satisfies the condition of \emph{detailed balance}.

In the commutative setting, if $P = (P_{ij})$ is the transition matrix of a Markov chain on $\{1,\dots, n\}$ with invariant probability vector $\sigma$, we say that detailed balance holds if $\sigma_i P_{ij} = \sigma_j P_{ji}$ for all $i,j$. 
An analytic way to formulate this condition is that $P$ is self-adjoint with respect to the weighted inner product on $\C^n$ given by 
$\ip{f,g}_\sigma = \sum_{j=1}^n \sigma_j \overline{f_j}g_j$.

In the quantum setting, with a reference density matrix $\sigma$ that is not a multiple of the identity, there are many candidates for such a weighted inner product. E.g., given $\sigma\in \Dens_+$, and $s\in [0,1]$ one can define an inner product on $\cA$ by
\begin{equation}\label{sinner}
	\ip{X,Y}_s = \tau[X^* \sigma^{s} Y\sigma^{1-s}]\ .
\end{equation}
Note that by cyclicity of the trace, $\ip{X,X}_s = \tau[|\sigma^{s/2}X \sigma^{(1-s)/2}|^2] \geq 0$,
so that $\ip{\cdot, \cdot}_s$ is indeed a positive definite sesquilinear form.
The inner products for $s=0$ and $s=\frac12$ will come up frequently in what follows, and they have their own names:
$\ip{\cdot,\cdot}_0$ is the Gelfand--Naimark--Segal inner product, denoted $\ip{\cdot, \cdot}_{L^{2}_{\rm GNS}(\sigma)}$, and $\ip{\cdot, \cdot}_{1/2}$ is the Kubo--Martin--Schwinger inner product, denoted $\ip{\cdot, \cdot}_{L^2_{\rm KMS}(\sigma)}$. 
We shall write $\cA = L^2_{\rm GNS}(\cA, \sigma)$ (resp. $\cA = L^2_{\rm KMS}(\cA, \sigma)$) if we want to stress this Hilbert space structure.

Suppose, for some $s\in [0,1]$, that $\cP_t$ is self-adjoint with respect to the $\langle \cdot,\cdot\rangle_s$ inner product. Then, for all $A \in \cA$,
\begin{multline*} 
	\tau[(\cP_t^\dagger \sigma) A] 
		= \tau[\sigma \cP_t A ] 
		= \tau[\sigma^{1-s} \one \sigma^s \cP_t A ] 
		=\langle \one, \cP_t A\rangle_s 
	\\
		= \langle \cP_t\one,  A\rangle_s
	    =  \langle \one,  A\rangle_s
		= \tau[\sigma A ] \ .
\end{multline*}
Hence for each of these inner products, self-adjointness of $\cP_t$ implies that $\sigma$ is invariant under $\cP_t^\dagger$. 

The following lemma of Alicki \cite{Alic76} relates some of the possible definitions of detailed balance; a proof may be found in \cite{2016-Carlen-Maas}.

\begin{lemma}\label{Allem} Let $\sK$ be a real linear transformation on $\cA$. If $\sK$ is self-adjoint with respect to the $\langle \cdot, \cdot\rangle_s$ inner product for some $s\in [0,1/2)\cup(1/2,1]$, then $\sK$ commutes with $\Delta_\sigma$, and 
$\sK$ is self-adjoint with respect to $\langle \cdot, \cdot\rangle_s$ for all $s\in [0,1]$, including $s=1/2$. 
\end{lemma}

As we have remarked, for a QMS $(\cP_t)_{t\geq 0}$, each $\cP_t$ is real, and so $\cP_t$ is self-adjoint with respect to the GNS inner product if and only if it is self-adjoint with respect to the $\langle \cdot, \cdot\rangle_s$ inner product
for all $s\in [0,1]$. However, if each $\cP_t$ is self-adjoint with respect to the KMS inner product, then it need not be self-adjoint with respect to the GNS inner product: 
There exist QMS for which each $\cP_t$ is self-adjoint with respect to the KMS inner product, but for which $\cP_t$ does not commute with $\Delta_\sigma$, and therefore cannot be self-adjoint with respect to the GNS inner product. A simple example is provided in appendix B of \cite{2016-Carlen-Maas}. The generators of QMS such that $\cP_t$ is self-adjoint with respect to the KMS inner product have been investigated by Fagnola and Umanita \cite{FU07}. However, there is a third notion of detailed balance that is natural in the present context, namely the requirement that each $\cP_t$ be self-adjoint with respect to the Boguliobov--Kubo--Mori inner product:

\begin{definition}[BKM inner product]\label{bkmdef}  The {\em BKM inner product} is defined by
\begin{equation}\label{bkmip}
\ip{A, B}_{L^2_{\rm BKM}(\sigma)} = \int_0^1 \ip{A,B}_s \dd s\ .
\end{equation}
\end{definition}

By what we have remarked above, if each $\cP_t$ is self-adjoint with respect  to the GNS inner product, then each $\cP_t$ is self-adjoint with respect to the BKM inner product. 
However, as will be discussed at the end of this section, the converse is not in general true. 
The relevance of the BKM version of detailed balance is due to the following result that we show in Theorem \ref{necBKM}: If the forward Kolmogorov equation for an ergodic QMS $(\cP_t)_{t\geq 0}$ with invariant state $\sigma\in \Dens_+$ is gradient flow for the quantum relative entropy $\Ent_{\sigma}(\rho) := \tau[\rho ( \log \rho - \log \sigma) ]$ with respect to some continuously differentiable Riemannian metric on $\Dens_+$, then each $\cP_t$ is self-adjoint with respect to the BKM inner product.   
The BKM inner product is closely connected to the relative entropy functional, and for this reason it appears in some of the functional inequalities that we consider in Section~\ref{sec:functional}.

On the other hand, only when each $\cP_t$  is self-adjoint with respect to the GNS inner product do we have a construction of such a Riemannian metric. 
The same is true for other constructions of Riemannian metrics on $\Dens_+$ for which QMS become gradient flow for  $\Ent_{\sigma}(\rho)$, in particular see \cite{MiMi17}.
Since most of this paper is concerned with our construction and its consequences, we make the following definition:

\begin{definition}[Detailed balance]\label{detbaldef}
Let $\sigma \in \cA$ be non-negative. 
We say that a quantum Markov semigroup $(\cP_t)_{t\geq 0}$ satisfies the {\em detailed balance condition} with respect to $\sigma$ if for each $t>0$, $\cP_t$ is self-adjoint with respect to the GNS inner product on $\cA$ induced by $\sigma$, i.e.,
	\begin{align*}
		\tau[\sigma A^* \cP_t B ] 
		= \tau[\sigma (\cP_t A)^* B ] 
		\quad \text{ for all } A, B \in \cA \ .
	\end{align*}
We shall write that $(\cP_t)_t$ satisfies $\sigma$-DBC for brevity.
\end{definition} 

The following result gives the general form of the generator of quantum Markov semigroups on $B(\sH)$ satisfying detailed balance. 
This result is due to Alicki \cite[Theorem 3]{Alic76}; see \cite{2016-Carlen-Maas} for a detailed proof.

\begin{theorem}[Structure of Lindblad operators with detailed balance] \label{thm:structure}
Let $\cP_t = e^{t \cL}$ be a quantum Markov semigroup on $B(\sH)$ satisfying detailed balance with respect to $\sigma \in \Dens_{+}$.
Then the generator $\cL$ and its adjoint $\cL^\dagger$ have the form
\begin{align}\label{eq:L-general}
	\cL & = \sum_{j \in \cJ}  e^{-\omega_j/2} \cL_j \ , \qquad
	\cL_{j}(A) =  V_j^* [A, V_j] +  [V_j^* , A] V_j \ , \\
\label{eq:L-dagger-general}
	\cL^\dagger & = \sum_{j \in \cJ}  e^{-\omega_j/2} \cL_{j}^\dagger \ , \qquad 
		\cL^\dagger_{j}(\rho) = 
		  [V_j, \rho V_j^*] + [V_j \rho, V_j^* ] \ ,
\end{align}
where $\cJ$ is a finite index set, the operators $V_j \in B(\sH)$ satisfy $\{ V_{j} \}_{j \in \cJ} = \{ V_{j}^{*} \}_{j \in \cJ}$, and $\omega_{j} \in \R$ satisfies
\begin{align}
\label{eq:log-relation} 
\Delta_\sigma V_j & = e^{-\omega_j} V_j \quad \text{for all } j \in \cJ \ .
\end{align}
\end{theorem}

For $j \in \cJ$, let $j^{*} \in \cJ$ be an index such that $V_{j^{*}} = V_{j}^{*}$. It follows from \eqref{eq:log-relation} that 
\begin{align*}
\omega_{j^{*}} = - \omega_{j} \ .
\end{align*}
Moreover, if we define $H = -\log \sigma$, \eqref{eq:log-relation} is equivalent to the commutator identity 
$
[V_j,H]  = - \omega_{j}V_j
$.
Furthermore, in our finite-dimensional context, the identity 
\begin{align}\label{eq:log-relation-sg}
\Delta_\sigma^t V_j  = e^{- \omega_j t} V_j
\end{align}
is valid for some $t\neq 0$ in $\R$ if and only if it is valid for all $t \in \C$.

\subsection{Gradient flow structure for the non-commutative Dirichlet energy}

Let $(\cP_t)_{t\geq 0}$ be a quantum Markov semigroup satisfying detailed balance with respect to $\sigma \in \Dens_+(\cA)$. 
Let $\cL$ be the generator, so that for each $t>0$, $\cP_t = e^{t\cL}$.  As explained in the discussion leading up to
Definition~\ref{detbaldef}, for each $t$,  $\cP_t$ is self-adjoint with respect to {\em both} the GNS  and the KMS inner products induced by $\sigma$. Therefore, we may define a {\em Dirichlet form} $\sE$ on $\cA$ by 
\begin{equation}\label{dirform}
\sE(A,A) = \lim_{t\downarrow 0}\frac{1}{t} \langle A, (I - \cP_t )A\rangle
\end{equation}
where the inner product is either the GNS or the KMS inner product.  Then, either way, 
the \emph{Kolmogorov backward equation} $\partial_t A = \cL A$ is a gradient flow equation for the energy 
$\sE(A,A)$ with respect to the chosen $L^2$ metric. 

The class of bilinear forms $\sE$ defined in terms of a self-adjoint QMS $(\cP_t)_{t\geq 0}$  through \eqref{dirform} 
is, by definition, the class of {\em conservative completely Dirichlet forms} on $\cA$ in the specified inner product. The abstract Beurling--Deny Theorem, discussed in the next section, provides an intrinsic characterization of such bilinear forms. 

Although Definition~\ref{detbaldef} might seem to suggest that the natural choice of the $L^2$ metric is the one given by
the GNS inner product, we shall show that in some sense it is the KMS inner product that is more natural: 
The Dirichlet form defined by \eqref{dirform} using the KMS inner product induced by $\sigma$ can be expressed in terms of a ``squared gradient'', and the associated non-commutative differential calculus will turn out to be very useful for investigating properties of the flow specified by $\partial_t A = \cL A$. 
A somewhat different construction leading to the representation of Dirichlet forms with respect to the KMS metric in terms of derivations has 
been given by Cipriani and Sauvageot \cite{CipSau03}. 
Our ``derivatives'' are not always derivations, and this more general 
structure is suited to applications. 
Indeed, one of the first non-commutative Dirichlet forms to be investigated in mathematical physics, the {\em Clifford Dirichlet form} of Gross, is most naturally expressed in terms of a sum of squares of {\em skew derivations}. 
The flexibility of our framework will be essential to our later applications. 
In this part of the introduction, we present only some of the key computations in a simple setting involving derivations to explain the roles of the KMS inner product. 
Our more general framework will be presented in Section 4.

Consider a Lindblad generator $\cL$ given as in Theorem~\ref{thm:structure}. 
To bring out the analogy with classical Kolmogorov backward diffusion equations of the form
\begin{equation}\label{kolfor}
\frac{\partial}{\partial t}f(x,t) = \Delta f(x,t) + (\nabla \log \sigma(x))\cdot \nabla f(x,t)\ ,
\end{equation}
where $\sigma$ is a smooth, strictly positive probability density on $\R^n$, 
we define the following {\em partial derivative operators on $\cA$}:
\begin{align}\label{simpleder}
\partial_j A  = [V_j, A] \ ,
\end{align}
$j\in \cJ$. 
Note that $\partial_{j}^\dagger = \partial_{j^{*}}$, where we recall that $j^{*}$ denotes an index such that $V_{j^*} = V_j^*$. 
An easy computation shows that the adjoint of $\partial_{j}$ with respect to $\KMS{\cdot,\cdot}$ is given by 
\begin{align}\label{eq:partial-KMS-adjoint}
	\partial_{j,\sigma}^\dagger A = 
	  \sigma^{-1/2} \partial_{j}^\dagger \big(  \sigma^{1/2} A \sigma^{1/2}  \big) \sigma^{-1/2} \ .
\end{align}

\begin{proposition}[Divergence form representation of $\cL$]\label{prop:Dirichlet-KMS}
For all $A \in \cA$ we have 
\begin{align*}
 \cL A = - \sum_{j \in \cJ} \partial_{j,\sigma}^\dagger \partial_{j} A \ .
\end{align*}
\end{proposition}

\begin{proof}
Using \eqref{eq:partial-KMS-adjoint} and \eqref{eq:log-relation-sg} we obtain
\begin{align*}
	\sum_{j \in \cJ}  \partial_{j,\sigma}^\dagger \partial_{j} A 
	&  = \sum_{j \in \cJ} \partial_{j,\sigma}^\dagger (V_j A - A V_j )
	\\& = \sum_{j \in \cJ}
			\sigma^{-1/2} \partial_{j}^\dagger \big(  \sigma^{1/2}  (V_j A - A V_j )
 \sigma^{1/2}  \big) \sigma^{-1/2}
	\\& = \sum_{j \in \cJ}
			\sigma^{-1/2} 
			\Big( V_j^* \sigma^{1/2}  (V_j A - A V_j ) \sigma^{1/2} 
				-  \sigma^{1/2}  (V_j A - A V_j ) \sigma^{1/2} V_j^* \Big)
			 \sigma^{-1/2}
	\\& = \sum_{j \in \cJ} \Big(
 e^{-\omega_j/2} V_j^*  (V_j A - A V_j )
					- e^{\omega_j/2} (V_j A - A V_j ) V_j^* \Big)
	\\& = \sum_{j \in \cJ} \Big(
 e^{-\omega_j/2} V_j^*  (V_j A - A V_j )
					- e^{-\omega_j/2} (V_j^* A - A V_j^* ) V_j \Big)
	\\& = -\sum_{j \in \cJ} 
 e^{-\omega_j/2} \cL_j(A) = -\cL A\ ,
\end{align*}
as desired. 
\end{proof}

Proposition \ref{prop:Dirichlet-KMS} can be stated equivalently as an integration by parts identity
\begin{align}\label{eq:IBP}
 \sum_{j \in \cJ} \KMS{\partial_j A, \partial_j B} = - \KMS{A,\cL B} \quad \text{ for } A, B \in \cA \ .
\end{align}
 
It is now immediate that the backward  equation $\partial_t A = \cL A$ with $\cL$ given by
\eqref{eq:Lindblad-form}, is the gradient flow equation for the energy 
$\sE(A,A)$ with respect to the KMS inner product induced by $\sigma$.  What makes this particular gradient flow representation especially useful is that the Dirichlet form $\sE$ is written, in  \eqref{eq:IBP}, as the expectation of a squared gradient.
That is, the gradient flow structure given here is analogous to the gradient flow formulation for the Kolmogorov backward  equation \eqref{kolfor} for the Dirichlet energy $\cD_{class}(f) = \frac12\int_{\R^n} |\nabla f(x)|^2 \sigma(x)\dd x$.
This would not be the case if we had considered the Dirichlet form based on the GNS inner product: 
We would have a gradient flow structure, but the Dirichlet form would not be the expectation of a squared gradient in any meaningful sense;  see however, Proposition \ref{prop:D-detailed-balance} below for a related representation.

In the next section we show how the non-commutative differential calculus associated to the Dirichlet from $\sE$ allows us to write the corresponding {\em forward equation} as gradient flow for the relative entropy with respect to a Riemannian metric constructed in terms of this differential calculus.

\subsection{A gradient flow structure for the quantum relative entropy}\label{umegak}

Consider the \emph{ quantum relative entropy functionals} $\Ent_{\sigma}: \Dens_{+} \to \R$ given by
\begin{align*}
	\Ent_{\sigma}(\rho)  := \tau[\rho ( \log \rho - \log \sigma) ] \ .
\end{align*}

Our goal is to sketch a proof of one of the results of \cite{2016-Carlen-Maas,MiMi17}, namely that the quantum master equation $\partial_t \rho = \cL^\dagger \rho$, which is a Kolmogorov {\em forward equation},  can be formulated as the gradient flow equation for $\Ent_{\sigma}$ with respect to a suitable Riemannian metric on $\Dens_{+}$.
The construction of the Riemannian metric will make use of the ``quantum directional derivatives'' $\partial_j$ introduced in the last subsection. 

Since $\Dens_{+}$ is a relatively open subset of the $\R$-affine subspace $\{ A \in \cA_{h} : \tau[A] = 1 \}$, we may identify, at each point in $\rho \in \Dens_{+}$, its tangent space $T_{\rho} \Dens_{+}$ with $\cA_{0} := \{ A \in \cA_{h} : \tau[A] = 0 \}$. The cotangent space $T_{\rho}^\dagger \Dens_{+}$ may also be identified with $\cA_{0}$ through the duality pairing $\ip{A,B} = \tau[A B]$ for $A, B \in \cA_{0}$.

Let $(g_{\rho})_{\rho \in \Dens_{+}}$ be a Riemannian metric on $\Dens_{+}$, i.e., a collection of positive definite bilinear forms $g_{\rho} : T_{\rho}\Dens_{+} \times T_{\rho}\Dens_{+} \to \R$ depending smoothly on $\rho \in \Dens_{+}$. 
Consider the associated operator $\cG_{\rho} :  T_{\rho}\Dens_{+} \to  T_{\rho}^{\dagger}\Dens_{+}$ defined by $\ip{A, \cG_{\rho} B} = g_{\rho}(A, B)$ for $A, B \in T_{\rho} \Dens_{+}$. 
Clearly, $\cG_{\rho}$ is invertible and self-adjoint with respect to the Hilbert--Schmidt inner product on $\cA_{0}$.
Define $\cK_{\rho} :  T_{\rho}^{\dagger}\Dens_{+} \to T_{\rho} \Dens_{+}$ by $\cK_{\rho} = (\cG_{\rho})^{-1}$, so that 
\begin{align}\label{eq:metric-K}
g_{\rho}(A,B) = \ip{ A, \cK_{\rho}^{-1} B} \ .
\end{align}
In many situations of interest it is convenient to define the metric $g_{\rho}$ by specifying the operator $\cK_{\rho}$. 
In such cases, there is often no explicit formula available for $\cG_{\rho}$ and $g_{\rho}$.

For a smooth functional $\cF : \Dens_{+} \to \R$ and $\rho \in \Dens_{+}$, its \emph{differential} $\rmD\cF(\rho) \in T_{\rho}^{\dagger}\Dens_{+}$ is defined by $\lim_{\eps \to 0} \eps^{-1}(\cF(\rho + \eps A) - \cF(\rho) ) = \ip{A,\rmD \cF(\rho)}$ for $A \in T_{\rho}\Dens_{+}$ (independently of the Riemannian metric $g_{\rho}$).
Its \emph{gradient} $\nabla_g \cF(\rho) \in T_{\rho}\Dens_{+}$ depends on the Riemannian metric through the duality formula $g_{\rho}(A, \nabla_g \cF(\rho)) = \ip{A, \rmD \cF(\rho)}$ for $A \in T_{\rho}\Dens_{+}$. It follows that $\cG_{\rho}\nabla_g \cF(\rho) = \rmD \cF(\rho)$, or equivalently
\begin{align*}
	\nabla_g \cF(\rho) = \cK_{\rho} \rmD \cF(\rho) \ .
\end{align*}
The gradient flow equation $\partial_{t}\rho = - \nabla_g \cF(\rho)$ takes the form 
\begin{align*}
\partial_{t}\rho = - \cK_{\rho} \rmD \cF(\rho) \ .
\end{align*}

Let us now focus on the relative entropy functional $\Ent_\sigma$ for some $\sigma\in \Dens_+$, and note that its differential is given by 
\begin{align}\label{eq:diff-entropy}
\rmD \Ent_{\sigma}(\rho) = \log \rho - \log \sigma \ .
\end{align}
Consider a generator $\cL^\dagger$ written in the form \eqref{eq:L-dagger-general}, i.e., 
\begin{align*}
	\cL^\dagger & = \sum_{j \in \cJ}  e^{-\omega_j/2} \cL_{j}^\dagger \ , \qquad 
		\cL^\dagger_{j}(\rho) = 
		  [V_j, \rho V_j^*] + [V_j \rho, V_j^* ] \ ,
\end{align*}
where $\{V_j\}_{j\in \cJ}$ is a finite set of eigenvectors of $\Delta_\sigma$ such that $\{V_j^*\}_{j\in \cJ} = \{V_j\}_{j\in \cJ}$, and where $\Delta_\sigma V_j = e^{-\omega_j}V_j$ for some $\omega_j \in \R$. 
As before, we use the notation $\partial_j A := [V_j, A]$.

For $\rho \in \Dens$ we define $\hrho_j\in \cA \otimes \cA$ by
\begin{align*}
   \hrhop = \int_0^1 \big(e^{\omega_j /2} \rho\big)^{1-s}
   			 \otimes  \big(e^{-\omega_j /2} \rho\big)^{s}\dd s  \ .
\end{align*}
We shall frequently make use of the \emph{contraction operator} $\# : (\cA \ot \cA) \times \cA \to \cA$ defined by
 \begin{align}\label{eq:contr}
  (A \ot B) \# C &:= ACB
   \end{align}
and linear extension. 
A crucial step towards obtaining the gradient flow structure is the following chain rule for the commutators $\partial_j$, which involves the differential of the entropy.

\begin{lemma}[Chain rule for the logarithm]
\label{lem:chain-log}
For all $\rho \in \Dens_+$ and $j \in \cJ$ we have
	\begin{align}\label{eq:chain-rule-log-LB}
		 e^{-\omega_j/2}V_j \rho  - e^{\omega_j/2}\rho V_j
		 &  =    \hrhop \# \partial_{j}(\log \rho - \log \sigma) \ .
	\end{align}
\end{lemma}

\begin{proof}
	Using \eqref{eq:log-relation} we infer that	
	\begin{align*}
		\partial_{j}( \log \rho - \log \sigma  ) 
		   = V_j \log (e^{-\omega_j/2} \rho) - \log (e^{\omega_j/2} \rho) V_j \ .
	\end{align*}
	Consider the spectral decomposition $\rho = \sum_\ell \lambda_\ell E_\ell$, where $\lambda_\ell > 0$ for all $i$, and $\{E_\ell\}_\ell$ are the spectral projections, so that  $E_\ell E_m = \delta_{\ell m} E_\ell$ and $\sum_\ell E_\ell = \one$. Observe that
	\begin{align*}
		\hrho_{j} = \sum_{\ell,m} \Lambda(e^{\omega_j/2} \lambda_\ell, e^{-\omega_j/2} \lambda_m) E_\ell \ot E_m \ , 
	\end{align*}
	where $\Lambda(\xi,\eta) = \int_0^1 \xi^{1-s} \eta^s \dd s = \frac{\xi - \eta}{\log \xi - \log \eta}$ denotes the logarithmic mean of $\xi$ and $\eta$.
	Thus, 
	\begin{align*}
		&\hrhop \# \big( \partial_{j}( \log \rho - \log \sigma  )  \big)
		\\& = \sum_{\ell, m, p}  \Lambda(e^{\omega_j/2} \lambda_\ell, e^{-\omega_j/2} \lambda_m) E_\ell 
		 \Big( \log (e^{-\omega_j/2} \lambda_p)  V_j E_p - \log (e^{\omega_j/2} \lambda_p) E_p V_j \Big) E_m
		 \\& = \sum_{\ell, m}  \Lambda(e^{\omega_j/2} \lambda_\ell, e^{-\omega_j/2} \lambda_m) 
		 \Big( \log (e^{-\omega_j/2} \lambda_m) - \log (e^{\omega_j/2} \lambda_\ell)  \Big)  E_\ell V_j E_m
		 \\& = \sum_{\ell, m}   
		 \big( e^{-\omega_j/2} \lambda_m - e^{\omega_j/2} \lambda_\ell  \big)  E_\ell V_j E_m
		 \\& =  e^{-\omega_j/2} V_j \rho - e^{\omega_j/2} \rho V_j\ ,
	\end{align*}
which proves \eqref{eq:chain-rule-log-LB}.
\end{proof}

For $\rho \in \Dens_+$ we define the operator $\cK_\rho : \cA \to \cA$ by
\begin{align}\label{eq:K-rho}
	\cK_\rho A := \sum_{j \in \cJ} \partial_{j}^\dagger \big( \hrhop \#  \partial_{j} A \big) \ .
\end{align}
Since $\Tr(A^* \cK_\rho B) = \overline{\Tr(B^* \cK_\rho A)}$ for $A, B \in \cA$, it follows that $\cK_\rho$ is a non-negative self-adjoint operator on $L^2(\cA, \tau)$ for each $\rho \in \Dens_+$. 
Assuming that $\cP_t$ is ergodic, the operator $\cK_\rho : \cA_0 \to \cA_0$ is invertible for each $\rho \in \Dens_+$ (see Corollary \ref{cor:Krho-bij} below for a proof of this statement). Since $\cK_\rho$ depends smoothly on $\rho$, it follows that $\cK_\rho$ induces a Riemannian metric on $\Dens_+$ defined by \eqref{eq:metric-K}.

The following result shows that the Kolmogorov forward equation $\partial_t \rho = \cL^\dagger \rho$ can be formulated as the gradient flow equation for $\Ent_{\sigma}$.

\begin{proposition}\label{prop:Lindblad-grad-flow}
For $\rho \in \Dens_+$ we have the identity
	\begin{align*}
		 \cL^\dagger \rho = - \cK_\rho \rmD \Ent_{\sigma}(\rho)\ ,
	\end{align*}
	hence the gradient flow equation of $\Ent_\sigma$ with respect to the Riemannian metric induced by $(\cK_\rho)_\rho$ is the master equation $\partial_t \rho = \cL^\dagger \rho$. 
\end{proposition}

\begin{proof}
	Using the identity \eqref{eq:diff-entropy}, the chain rule from Lemma \ref{lem:chain-log}, and the fact that $\{V_j\} = \{V_j^*\}$ and $\omega_{j^*} = -\omega_j$, we obtain 
	\begin{align*}
		\cK_\rho \rmD \Ent_{\sigma}(\rho)
		 & =  \sum_{j \in \cJ} 
 	\partial_{j}^\dagger \big( \hrhop \#  \partial_{j} (\log \rho - \log \sigma) \big)  
		\\& =   \sum_{j \in \cJ} 
 	\partial_{j}^\dagger \big(  e^{-\omega_j/2}V_j \rho  - e^{\omega_j/2}\rho V_j \big)   
		\\&  = \frac12 \sum_{j \in \cJ} \Big(
 	\partial_{j}^\dagger \big(  e^{-\omega_j/2}V_j \rho  - e^{\omega_j/2}\rho V_j \big)   
 	    + 
 	\partial_{j} \big(  e^{\omega_{j}/2}V_j^* \rho  - e^{-\omega_j/2}\rho V_j^* \big)   \Big)
		 \\& = - \frac12\sum_{j \in \cJ}
		 e^{-\omega_j/2} \Big( [V_j, \rho V_j^*] +  [V_j \rho, V_j^* ] \Big) 
		+ e^{\omega_j/2} \Big(  [V_j^*, \rho V_j] +  [V_j^* \rho, V_j ]\Big) 
		 \\& = - \sum_{j \in \cJ}
		 e^{-\omega_j/2} \Big( [V_j, \rho V_j^*] +  [V_j \rho, V_j^* ] \Big) 
		 \\& = - \cL^\dagger \rho \ ,
	\end{align*}
which is the desired identity.
\end{proof}

In this paper we extend this result into various directions: we consider more general entropy functionals, more general Riemannian metrics, and nonlinear evolution equations. 

\begin{remark}\label{rem:classical-JKO}
The gradient flow structure given in Proposition \ref{prop:Lindblad-grad-flow} can be viewed as a non-commutative analogue of the Kantorovich gradient flow structure obtained by Jordan, Kinderlehrer and Otto \cite{JKO98} for the \emph{Kolmogorov backward equation}
\begin{align*}
\frac{\partial}{\partial t}\rho(x,t) = \Delta \rho(x,t) - \nabla \cdot ( \rho(x,t) \nabla \log \sigma(x) ) \ .
\end{align*} 
This structure is formally given in terms of the operator $K_\rho$ defined by 
\begin{align*}
 K_{\rho} \psi  = - \nabla\cdot(\rho \nabla \psi) \ ,
\end{align*}
for probability densities $\rho$ on $\R^n$ and suitable functions $\psi : \R^n \to \R$ in analogy with \eqref{eq:K-rho}.
As the differential of the relative entropy $\Ent_\sigma(\rho) = \int_{\R^n} \rho(x)\log \frac{\rho(x)}{\sigma(x)} \dd x$ is given by $\rmD \Ent_\sigma(\rho) = 1 + \log \frac{\rho}{\sigma}$, we have
\begin{align*}
   K_\rho \rmD \Ent_\sigma(\rho) = - \Delta \rho + \nabla\cdot (\rho \nabla \log \sigma)\ ,
\end{align*}
which is the commutative counterpart of Proposition \ref{prop:Lindblad-grad-flow}.
\end{remark}

\subsection{The necessity of BKM-detailed balance}

In the classical setting of irreducible  finite Markov chain, Dietert \cite{D15} has proven that if the Kolmogorov forward equation for a Markov semigroup can be written as gradient flow for the relative entropy with respect to the unique invariant measure for some continuously differentiable Riemannian metric, then the Markov chain is necessarily reversible. 
That is, it satisfies the classical detailed balance condition. 

\begin{theorem}\label{necBKM}  Let $(\cP_t)_{t\geq 0}$ be an ergodic QMS with generator $\cL$ and invariant state $\sigma\in \Dens_+$. 
If there exists a continuously differentiable Riemannian metric $(g_\rho)$ on $\Dens_+$ such that the quantum master equation $\partial \rho = \cL^\dagger \rho$ is the gradient flow equation for $\Ent_\sigma$ with respect to $(g_\rho)$, then each $\cP_t$ is self-adjoint with respect to the BKM inner product associated to $\sigma$. 
\end{theorem}

Before beginning the proof, we recall some relevant facts, and introduce some notation. Regarding $\sigma$ as an element of $\M_n(\C)$, we define the operator $\cM$ on $\M_n(\C)$ by 
\begin{align*}
\cM A = \int_0^1 \sigma^{1-s}A \sigma^s \dd s \ .
\end{align*}
A simple calculation shows that ${\cM}$ is the derivative of the matrix exponential function. Its inverse is the derivative of the matrix logarithm function:
\begin{align*}
\cM^{-1} A = \int_0^\infty \frac{1}{t+\sigma} A \frac{1}{t+\sigma} \dd t \ ,
\end{align*}
(see Example~\ref{ex:entropy} below for more details).
While the matrix logarithm function is monotone, the matrix exponential is not. Thus $\cM^{-1}$ preserves positivity, but $\cM$ does not. In fact $A \mapsto \cM^{-1} A$ is evidently completely positive. 
The BKM inner product can now be written as 
$$\ip{A,B}_{L^2_{\rm BKM}(\sigma)} 
	= \tau[A^* \cM B ]
	= \tau[\cM(A^*) B ] \ .$$

\begin{proof}[Proof of Theorem~\ref{necBKM}]

As before, it will be convenient to consider the operators $(\cK_\rho)$ defined by \eqref{eq:metric-K}.
Since $\rmD \Ent_{\sigma}(\rho) = \log \rho - \log \sigma$, the gradient flow equation $\partial_t \rho = - \cK_{\rho} \rmD \Ent_{\sigma}(\rho)$ becomes 
\begin{align}\label{eq:gf}
	\cL^\dagger \rho = - \cK_\rho (\log \rho - \log \sigma) \ . 
\end{align}
Applying this identity to $\rho_\eps = \sigma + \eps A$ for $A \in \cA_0$, and differentiating at $\eps = 0$, we obtain using the identity
$\partial_\eps|_{\eps = 0} \log \rho_\eps = \cM^{-1} A$ that
\begin{align}\label{db1}
	\cL^\dagger A = - \cK_\sigma \cM^{-1} A \ ,
\end{align}
Consequently, for $A, B \in \cA$,
\begin{align*}
	\ip{\cL A, B}_{L^2_{\rm BKM}(\sigma)}
	= \tau[ (\cL A)^* \cM B] 
	= \tau[ A^* \cL^\dagger \cM B] 
	= - \tau[ A^* \cK_\sigma B] \ .
\end{align*}
As $g_\sigma$ is a symmetric bilinear form, the operator $\cK_\sigma$ is self-adjoint with respect to the Hilbert-Schmidt scalar product. This implies the result.
\end{proof}

We are unaware of any investigation of the nature of the class of QMS generators that are self-adjoint for the BKM inner product associated to their invariant state $\sigma$. Therefore we briefly demonstrate that this class {\em strictly} includes the class of QMS generators  that are self-adjoint for the GNS inner product associated to their invariant state $\sigma$. 

Let $\cP$ be a unital completely positive map such that $\cP^\dagger \sigma = \sigma$, and define 
$$\widetilde\cP(A) =  {\cM}^{-1}(\sigma^{1/2} \cP(A) \sigma^{1/2})\ .$$
Note that 
$$ {\cM}^{-1}(\sigma^{1/2} A \sigma^{1/2}) = \int_0^\infty \frac{\sigma^{1/2}}{t + \sigma} A \frac{\sigma^{1/2}}{t + \sigma} \dd t$$
defines a completely positive and unital operator, and hence $\widetilde\cP$ is completely positive and unital. Moreover, 
$$\widetilde\cP^\dagger(A) = \cP^\dagger( {\cM}^{-1}(\sigma^{1/2} A \sigma^{1/2}))\ ,$$
and hence $\widetilde\cP^\dagger \sigma = \sigma$.  Now observe that $\widetilde\cP$ is self-adjoint with respect to the BKM inner product if and only if 
$\cP$ is self-adjoint for the KMS inner product. In fact, for all $A, B \in \cA$,
$$\ip{\widetilde\cP A , B}_{L^2_{\rm BKM}(\sigma)} 
	= \tau[ \cM( \cM^{-1} (\sigma^{1/2} \cP(A^*) \sigma^{1/2})) B] = \ip{\cP A, B}_{L^2_{\rm KMS}(\sigma)}\ .$$

Next, it is clear that $\widetilde\cP$ commutes with $\Delta_\sigma$ if and only if $\cP$ commutes with $\Delta_\sigma$.  Since there exist completely positive unital maps $\cP$ satisfying $\cP^\dagger \sigma = \sigma$ that are KMS symmetric but do not commute with $\Delta_\sigma$, there exists 
completely positive unital maps $\widetilde\cP$ satisfying $\widetilde\cP^\dagger \sigma = \sigma$ that are BKM symmetric but do not commute with $\Delta_\sigma$.

Moreover, the class of completely positive unital maps $\widetilde\cP$ satisfying 
$\widetilde\cP^\dagger \sigma = \sigma$ that are BKM symmetric is in some sense larger 
than the class of   completely positive unital maps $\cP$ satisfying $\cP^\dagger \sigma = \sigma$ 
that are KMS symmetric: The map
$\cP\mapsto \widetilde\cP$  is invertible, but ${\cM}$ is not even positivity 
preserving, let alone completely positive, so that 
$$\cP(A) =  \sigma^{-1/2}  {\cM}(\widetilde\cP(A) )\sigma^{-1/2}  $$
need not be completely positive.  
It is therefore an interesting problem to characterize the QMS generators that are self-adjoint with respect to the 
BKM inner product.

\section{Beurling--Deny Theory in finite-dimensional von Neumann algebras}
\label{sec:Beurling--Deny}

In this section we recall some key results of Beurling--Deny theory that will be used in our construction of 
Dirichlet forms in Section \ref{sec:Dirichlet}. We present some proofs of known results for the reader's 
convenience, especially when available references suppose a familiarity with the Tomita--Takesaki theory. 
However, Theorem~\ref{kmssd}, which singles out the KMS inner product, is new. 

\subsection{Abstract Beurling--Deny Theory}
\label{sec:abstract-Beurling--Deny}

In this subsection,  $\cH$ always denotes a {\em real} Hilbert space with inner product $\langle \cdot,\cdot\rangle$. 
Let $\cone$ be a cone in $\cH$.  That is, $\cone$ is a convex subset of $\cH$ such that if 
$\phi\in \cone$, then $\lambda \phi\in \cone$ for all $\lambda>0$.   The cone $\cone$ is {\em pointed} in case $\phi\in \cone$ 
and $-\phi\in \cone$ together imply that $\phi =0$. In particular, a subspace of $\cH$ is a cone, but it is not a pointed cone. 

\begin{definition}[Dual cone]  The {\em dual cone} $\cone^{\circ}$ of a cone $\cone$ is the set
\begin{equation}\label{dualcone}
\cone^\circ := \{ \psi\in \cH\ :\ \langle \psi,\phi\rangle\geq 0\quad{\rm for\ all}\ \phi\in \cone\ \}\ .
\end{equation}
A cone $\cone$ is {\em self-dual} in case $\cone^\circ = \cone$. 
\end{definition}  

Let $\cone$ be a non-empty self-dual cone in $\cH$, and take $\phi \in \cH$. 
Since $\cone$ is a non-empty 
closed, convex set, the Projection Lemma ensures the existence of $\Proj_{\cone}(\varphi)\in \cone$ such that
\begin{equation}\label{abd2}
\|\phi - \Proj_{\cone}(\varphi)\| < \| \phi - \psi\| \quad{\rm for\ all}\ \psi\in \cone,\  \psi \neq \Proj_{\cone}(\varphi)\ .
\end{equation}

\begin{theorem}[Decomposition Theorem]
\label{thm:Moreau} Let $\cone$ be a non-empty self-dual cone in $\cH$. Then for each $\phi\in \cH$, there exists a unique pair $\phi_+,\phi_-\in \cone$
such that 
\begin{equation}\label{abd1}
\phi = \phi_+ - \phi_-\quad{\rm and}\quad \langle \phi_+,\phi_-\rangle = 0 \ .
\end{equation}
In fact, $\phi_+ = \Proj_{\cone}(\phi)$ and  $\phi_- = \Proj_{\cone}(-\phi)$, where $\Proj_\cone$ denotes projection onto (the closed convex set) $\cone$.  
\end{theorem}

\begin{proof}  Define $\phi_+ := \Proj_{\cone}(\phi)$.  Then define $-\phi_- := \phi - \phi_+$. 
We claim that $\phi_-\in \cone$. Indeed, for any $\psi\in \cone$ and any $\epsilon>0$, $\phi_+  + \epsilon \psi\in \cone$, and hence,
\begin{eqnarray*}
\|\phi_-\|^2
  = \|\phi - \phi_+\|^2 
   < \|\phi - (\phi_+ +\epsilon\psi)\|^2
= \|\phi_-\|^2 + 2\epsilon\langle \phi_-,\psi\rangle + \epsilon^2\|\psi\|^2\ .
\end{eqnarray*}
Therefore, $\langle \phi_-,\psi\rangle\geq 0$ for all $\psi\in \cone$. Since $\cone$ is self-dual, the claim follows. 

To see that $\phi_+$ and $\phi_-$ are orthogonal, let $\epsilon\in (-1,1)$, so that $(1+\epsilon)\phi_+ \in \cone$.  It follows that
$
\|\phi_-\|^2
 = \|\phi- \phi_+\|^2 \leq \|\phi - (1+\epsilon)\phi_+\|^2
 = \|\phi_-\|^2 + 2\epsilon\langle \phi_-,\phi_+\rangle + \epsilon^2 \|\phi_+\|^2$
which yields a contradiction for negative  $\epsilon$ sufficiently close to zero, unless $\langle \phi_-,\phi_+\rangle = 0$. This proves existence of the decomposition.  Now the fact that $\phi_- =  \Proj_{\cone}(-\phi)$ follows from a theorem of Moreau \cite{Moreau-1962}, as does the uniqueness of the decomposition, though both points can be proved directly by variations on the arguments just provided. 
\end{proof}

\begin{definition} Let $\cH$ be a real Hilbert space with a non-empty self-dual cone $\cone$. For $\phi$ in $\cH$, define $\phi_+$ and $\phi_-$ as in Theorem \ref{thm:Moreau}. 
Then $\phi_+$ is the {\em positive part of} $\phi$, $\phi_-$ is the {\em negative part of} $\phi$, and $|\phi|:= \phi_+ + \phi_-$ is the {\em absolute value of} $\phi$. If $\phi_- = 0$, we write $\phi\geq 0$.
\end{definition}

We next recall some elements of the abstract theory of symmetric Dirichlet forms.
A \emph{bilinear form} on a real Hilbert space $\cH$ is a bilinear mapping $\sE: \cD \times \cD \to \R$ where $\cD \subseteq \cH$ is a linear subspace (called the \emph{domain} of $\sE$).
We say that $\sE$ is 
\emph{non-negative} if $\sE(\phi,\phi)\geq 0$ for all $\phi \in \cD$; 
\emph{symmetric} if $\sE(\phi,\psi) = \sE(\psi,\phi)$ for all $\psi, \psi \in \cD$;
\emph{closed} if $\cD$ is complete when endowed with the norm $\| \phi \|_\sE = (\| \phi \|^2 + \sE(\phi,\phi))^{1/2}$;
and \emph{densely defined} if $\cD$ is dense in $\cH$.

\begin{definition}[Dirichlet form]  Let $\cH$ be a real Hilbert space  with a non-empty self-dual cone $\cone$. A non-negative, symmetric, closed bilinear form $\sE$ on $\cH$ with dense domain $\mathcal{D}$  is a {\em Dirichlet form} in case $|\phi|  \in {\mathcal D}$ for all $\phi \in {\mathcal D}$ , and
\begin{equation}\label{abd21}
\sE(|\phi|,|\phi|) \leq \sE(\phi,\phi) \ ,
\end{equation}
or equivalently, if for all $\phi\in \cD$, 
\begin{equation}\label{abd21b}
\sE(\phi_+,\phi_-) \leq 0 \ .
\end{equation}
\end{definition}

To see the equivalence of \eqref{abd21} and \eqref{abd21b}, note that
$$\sE(|\phi|,|\phi|) - \sE(\phi,\phi)  =4\sE(\phi_+,\phi_-)\ .$$

Given a non-negative, symmetric, closed bilinear form $\sE$, the \emph{operator $\cL : \cD_\cL \subseteq \cH \to \cH$ associated to $\sE$} is defined by
\begin{align*}
\cD_{\cL} &:= \{ \psi \in \cD  \ | \ \exists \xi \in \cH :  \sE(\phi,\psi) = -\ip{\phi, \xi} \ \ \forall \phi \in \cD \} \ , \qquad
\cL \psi := \xi \ .
\end{align*}
This operator is well-defined since $\cD_\cL$ is dense. Moreover, $\cL$ is non-positive and self-adjoint. 

The following abstract result by Ouhabaz \cite{Ouha96} characterizes the invariance of closed convex sets under the associated semigroup (in a more general setting that includes nonsymmetric Dirichlet forms).

\begin{theorem}[Ouhabaz' Theorem]
\label{thm:Ouhabaz}
Let $\cH$ be a real Hilbert space, and let $\sE$ be a non-negative, symmetric, closed bilinear form with domain $\cD$ and associated operator $\cL$.  
Let $\cC \subseteq \cH$ be closed and convex. 
Then, the following assertions are equivalent: \begin{enumerate}
\item $e^{t\cL} \phi \in \cC$ for all $\phi \in \cC$ and all $t \geq 0$;
\item $\Proj_\cC \phi \in \cD$ and $\sE(\Proj_\cC \phi, \phi - \Proj_\cC \phi) \leq 0$ for all $\phi \in \cD$.
\end{enumerate}
\end{theorem}

Combining Theorem \ref{thm:Moreau} and Theorem \ref{thm:Ouhabaz} we obtain the following result.

\begin{corollary}[Abstract Beurling--Deny Theorem] 
\label{cor:BD}
Let $\cH$ be a real Hilbert space with a non-empty self-dual cone $\cone$. Let $\sE$ be a non-negative, symmetric, closed bilinear form with domain $\cD$. Then, $\sE$ is a Dirichlet form if and only if  $e^{t\cL}\phi \geq 0$ for all $t \geq 0$ and all $\phi\geq 0$.
\end{corollary}

\subsection{Completely Dirichlet forms} 
\label{sec:completely-Dirichlet}

Let $\sE$ be a Dirichlet form on $\big(\cA, \ip{\cdot, \cdot}_{L^2_{\rm KMS}(\sigma)}\big)$ with the KMS inner product specified by a faithful state 
$\sigma$. 
Here, the notion of Dirichlet form is understood with respect to the self-dual cone consisting of all positive semidefinite matrices belonging to $\cA$; see Lemma \ref{lem:pos-selfdual-A} below.
Let $\cP_t = e^{t\cL}$ where $\cL$ is the semigroup generator associated to $\sE$. 
Recall that the Dirichlet form $\sE$ is said to be {\em completely Dirichlet} in case for each $t$, $\cP_t$ is completely positive. 

The condition that $\sE$ be completely Dirichlet may be expressed in terms of $\sE$ itself, permitting one to check the property directly from a specification of $\sE$. 

For $m\in \N$, let $E_{ij}$ denote the matrix whose $(i,j)$-entry is $1$, with all other entries being $0$. Alternatively, $E_{ij}$ represents the linear transformation taking ${\bf e}_j$ to ${\bf e}_i$, while annihilating ${\bf e}_k$ for $k\neq j$. (Here $\{{\bf e}_1,\dots,{\bf e}_m\}$ is the standard orthonormal basis of $\C^m$.)  It follows that $E_{ij}E_{k\ell} = \delta_{jk}E_{i\ell}$.
The general element of $\cA \otimes \M_m(\C)$ can be written as
\begin{equation}\label{comdir1}
{\bf A} = \sum_{i,j=1}^m A_{ij}\otimes E_{ij}
\end{equation}
where each $A_{ij}\in \cA$. 
With $\tau_m$ denoting the normalized trace on $\M_m(\C)$, the state $\sigma\otimes \tau_m$ on $\cA \otimes \M_m(\C)$ is defined by
\begin{align*}
	\sigma\otimes \tau_m({\bf A}) := 
	 \frac1m 
	 \sum_{j=1}^m \sigma(A_{jj}) \ ,
\end{align*}
where ${\bf A}$ is given by \eqref{comdir1}. 
The corresponding KMS inner product on $\cA \otimes \M_m(\C)$ is denoted $\langle \cdot,\cdot\rangle_{L^2_{\rm KMS}(\sigma\otimes \tau_m)}$.
One readily checks that for ${\bf A},{\bf B}\in \cA \otimes \M_m(\C)$,
$$\langle {\bf B},{\bf A}\rangle_{L^2_{\rm KMS}(\sigma\otimes \tau_m)} 
	= 
	 \frac1m 
	\sum_{i,j=1}^m 
	\ip{ B_{ij}, A_{ij} }_{L^2_{\rm KMS}(\sigma)} \ .$$
Define $\cP_t^{(m)}$ on $\cA \otimes \M_m(\C)$ by
\begin{equation}\label{comdir2}
\cP_t^{(m)} {\bf A} = \sum_{i,j=1}^m \cP_t A_{ij}\otimes E_{ij}
\end{equation}
where ${\bf A}$ is given by \eqref{comdir1}. 
One then computes
\begin{align*}
 - \frac{{\rm d}}{{\rm d}t} 
 	\langle {\bf A},\cP_t^{(m)} {\bf A} \rangle_{L^2_{\rm KMS}	(\sigma\otimes \tau_m)}\bigg|_{t=0} 
  = 
  -  	 \frac1m 
	\sum_{i,j=1}^m 
		\langle A_{ij},
			 \cL A_{ij}\rangle_{L^2_{\rm KMS}(\sigma)} 
  =  
  	\frac1m 
   \sum_{i,j=1}^m\sE(A_{ij}, A_{ij})\ .
\end{align*}
Thus, we define $\sE^{(m)}$ on $\big(\cA \otimes \M_m(\C), \langle \cdot,\cdot\rangle_{L^2_{\rm KMS}(\sigma\otimes \tau_m)}\big)$ by 
\begin{equation}\label{comdir3}
\sE^{(m)} ({\bf A},{\bf A}) 
	= 
	  	 \frac1m 
\sum_{i,j=1}^m\sE(A_{ij}, A_{ij})
\end{equation}
where ${\bf A}$ is given by \eqref{comdir1}. In view of Corollary \ref{cor:BD}, $\sE$ is completely Dirichlet if and only if for each $m\in \N$, $\sE^{(m)}$ is Dirichlet. 

A QMS $(\cP_t)_t$ is not only completely positive; it also satisfies $\cP_t\one = \one$ for all $t$. This too may be expressed in terms of the Dirichlet form $\sE$: A Dirichlet form $\sE$ is {\em conservative} in case $\sE(A, \one) = 0$ for all $A \in \cA$, and one readily sees that this is equivalent to the condition that $\cP_t\one = \one$ for all $t$.

\subsection{Moreau decomposition with respect to the cone of positive matrices}
\label{sec:Moreau-decomposition}

Let $\H_n(\C)$ denote the set of self-adjoint $n \times n$ matrices, which contains a distinguished pointed cone $\cone$, namely the cone of positive semidefinite matrices $A$.
If we equip $\H_n(\C)$ with the Hilbert--Schmidt inner product $\langle X,Y\rangle = \Tr[X Y]$, then $\cone$ is self-dual: for $X\in \H_n(\C)$, $\langle X,A\rangle \geq 0$ for all $A\in \cone$ if and only if $\langle v,Xv\rangle \geq 0$ for all $v \in \C^n$, as one sees by considering rank one projections and using the spectral theorem. 

The next result characterizes the Moreau decomposition in $(\H_n(\C),\ip{\cdot,\cdot})$ in spectral terms. For $X\in \H_n(\C)$, there is the {\em spectral decomposition}  $X = X_{(+)} - X_{(-)}$ where
\begin{equation}\label{specde}
X_{(+)} = X\one_{(0,\infty)}(X)  \quad {\rm and}\quad   X_{(-)} = -X\one_{(-\infty,0)}(X) \ .
\end{equation}

\begin{theorem}[Moreau decomposition for  Hilbert--Schmidt] \label{HScase} Let $\cH$ be $\H_n(\C)$ equipped with the Hilbert--Schmidt inner product, and let $\cone$ be the cone of positive semidefinite matrices.
Then the spectral decomposition of $X\in \cH$ coincides with the decomposition of $X$ into its positive and negative parts with respect to $\cone$.
\end{theorem}

\begin{proof}
Let $X \in \H_n(\C)$, and let $X = X_+-X_-$ be the decomposition determined by $\cone$.  
Then, for $v$ in the range of $X_+$, we have
$X_+ - \epsilon|v\rangle\langle v| \in \cone$ for all sufficiently small $\epsilon > 0$. Therefore, 
\begin{align*}
\|X_-\|^2
 =    \| X - X_+ \|^2 
 \leq \|X - (X_+ - \epsilon|v\rangle\langle v|)\|^2 
 =   \|X_-\|^2  -2\epsilon \langle v, X_-v\rangle   + \epsilon^2\|v\|^2\ .
\end{align*}
It follows that $\langle v, X_-v\rangle \leq 0$, but since $X_-\in \cone$, this yields $\langle v, X_-v\rangle = 0$.  
Hence the range of $X_+$ lies in the null-space of $X_-$, so that $X_-X_+ =0$. Taking the adjoint, we find that $X_+ X_- =0$. Therefore, $X_-$ and $X_+$ commute with each other, 
and hence with $X$. 
Thus, the projectors onto the ranges of $X_+$ and $X_-$ are both spectral projectors of $X$. 
Since $X = X_+ - X_-$ it follows that $X_+ = X_{(+)}$ and $X_- = X_{(-)}$.
\end{proof}

The situation is more interesting for other inner products on $\H_n(\C)$. Let $\sigma$ be
an invertible density matrix. For $s\in [0,1]$, let $\langle \cdot, \cdot\rangle_s$ be the inner product on $\M_n(\C)$ given by $\langle A, B\rangle_s = \Tr[ A^*\sigma^s B \sigma^{1-s}]$. 

\begin{theorem}\label{kmssd}  
Let $\sigma$ be an invertible $n\times n$ density matrix that is not a multiple of the identity. Then the cone $\cone$ of positive matrices in $\H_n(\C)$ is self-dual with respect to the inner product
$\langle \cdot, \cdot\rangle_s$ determined by $\sigma$ if and only if $s=\frac12$.
\end{theorem}

\begin{proof}  Let $X\in \H_n(\C)$ and $A\in \cone$. 
Then
$\langle X,A\rangle_s =  \Tr[ X\sigma^s A \sigma^{1-s}] = \Tr[ (\sigma^{1-s} X\sigma^s) A ]$.
Therefore, $\langle X,A\rangle_s \geq 0$ for all $A\in \cone$ if and only if $\sigma^{1-s} X\sigma^s \in \cone$.  If $\sigma^{1-s} X\sigma^s \in \cone$, then $\sigma^{1-s} X\sigma^s$ is self-adjoint, and hence
$\sigma^{1-s} X\sigma^s = \sigma^{s} X\sigma^{1-s}$, or, what is the same, 
$[\sigma^{1-2s},X] = 0$.
Let $X := |v \rangle \langle v|$ with $v$ chosen {\em not} to be an eigenvector of $\sigma$. Then for $s\neq \frac12$,  $[\sigma^{1-2s},X] \neq 0$.
Therefore, $X\in \cone$, but $X\notin \cone^\circ$.  Hence, $\cone$ is not self-dual  when $\cH$ is equipped with the inner product $\langle \cdot, \cdot\rangle_s$ for $s\neq \frac12$. 

One the other hand, 
$$\langle X,A\rangle_{1/2} =  \Tr[ X\sigma^{1/2} A \sigma^{1/2}] = 
\Tr[ (\sigma^{1/4} X\sigma^{1/4}) (\sigma^{1/4}A\sigma^{1/4}) ]\ . $$
Since $\sigma$ is invertible, as $A$ ranges over $\cone$, $\sigma^{1/4}A\sigma^{1/4}$ 
ranges over $\cone$,  and so $\langle X,A\rangle_{1/2}\geq 0$ for all $A\in \cone$  if and only if
$\sigma^{1/4} X\sigma^{1/4} \in \cone$. Again, since $\sigma$ is invertible, this is the case if and only if $X\in \cone$. Hence, $\cone$ is self-dual for $\langle \cdot, \cdot\rangle_{1/2}$, the KMS inner product. 
\end{proof}

The Moreau decomposition for the KMS scalar product can easily be obtained from Theorem \ref{HScase}
by a unitary transformation.

\begin{theorem}[Moreau decomposition for KMS] \label{closestMN} Let $\sigma$ be an invertible $n\times n$ density matrix and let $X\in \H_n(\C)$. Then, with respect to the KMS norm  on $\H_n(\C)$,
\begin{equation}\label{abd31} 
\| X -   \sigma^{-1/4}(\sigma^{1/4}X\sigma^{1/4})_{(+)} \sigma^{-1/4}\|_{L^2_{\rm KMS}(\sigma)} \leq \|X - A\|_{L^2_{\rm KMS}(\sigma)} 
\end{equation}
for all $A\in \cone$. Consequently, the positive part of $X$ in the decomposition according to $\cone$, $X_+$, is given by  
\begin{equation}\label{abd31b}
X_+ = \sigma^{-1/4}(\sigma^{1/4}X\sigma^{1/4})_{(+)} \sigma^{-1/4}\ .
\end{equation}
\end{theorem}

\begin{proof} The map $Y \mapsto \sigma^{1/4}Y\sigma^{1/4}$ is unitary from $\H_n(\C)$ equipped with the KMS inner product to $\H_n(\C)$ equipped with the Hilbert--Schmidt inner product. 
That is,
\begin{align*}
	\|X -A\|_{L^2_{\rm KMS}(\sigma)}^2 = \Tr[  \sigma^{1/4}X\sigma^{1/4} - \sigma^{1/4}A\sigma^{1/4}]^2
\end{align*}
for $X, A \in \H_n(\C)$.
By Theorem~\ref{HScase},
$\min\{ \Tr[  \sigma^{1/4}X\sigma^{1/4} - B]^2\ :\ B\in \cone\}$ is achieved at $B = (\sigma^{1/4}X\sigma^{1/4})_{(+)}$.
\end{proof}

We conclude the section by extending the results above to an arbitrary $*$-subalgebra $\cA$ of $\M_n(\C)$. Let $\sigma$ be an invertible $n\times n$ density matrix belonging to $\cA$.

\begin{lemma} \label{lem:pos-selfdual-A} Let $\cH$ be  $\cA_{h}$ equipped with the KMS inner product induced by $\sigma$, 
and let $\cone$ be the 
positive matrices in $\M_n(\C)$, and let $\cone_\cA := \cone\cap \cA$.   Then $\cone_\cA$ is 
self-dual in $\cH$.
\end{lemma}

\begin{proof} 
%
Let $X \in \cP_\cA$. For any $A \in \cone_\cA$ we have $\sigma^{1/2}A\sigma^{1/2} \geq 0$, hence
$
	\ip{X,A}_{L^2_{\rm KMS}(\sigma)} 
	= \Tr[X
		  \sigma^{1/2}A\sigma^{1/2}] 
	\geq 0		  
$,
which shows that $X \in \cP_\cA^\circ$.

Conversely, suppose that $X \in \cA_h$ belongs to $\cP_\cA^\circ$. For every $A \in \cP_\cA$ we then have
$
	\Tr[X \sigma^{1/2} A \sigma^{1/2}]
	= \ip{X,A}_{L^2_{\rm KMS}(\sigma)}
	\geq 0$.
Since $\sigma$ is invertible, it follows that 
	$\Tr[X B] \geq 0$
for every $B \in \cP_\cA$. Therefore, the spectrum of $X$ is non-negative, which implies that $X$ belongs to $\cP$ and hence to $\cP_\cA$. 
\end{proof}

\begin{lemma} 
\label{closestMA}
Let $X$ be a self-adjoint element of $\cA$. Then the decomposition of $X$ with respect to
$\cone_\cA$ is given by $X = X_+ - X_-$ where
$$X_+ := \sigma^{-1/4}(\sigma^{1/4}X\sigma^{1/4})_{(+)}\sigma^{-1/4}\quad{\rm and}\quad
X_- := \sigma^{-1/4}(\sigma^{1/4}X\sigma^{1/4})_{(-)}\sigma^{-1/4}\ .$$
\end{lemma}

\begin{proof}
Let $X$ be a self-adjoint element of $\cA$.  
Then by Theorem~\ref{closestMN}, $\min\{\|X - A\|_{L^2_{\rm KMS}(\sigma)}\ : A \in \cone\}$ is achieved at 
$A = \sigma^{-1/4}(\sigma^{1/4}X\sigma^{1/4})_{(+)}\sigma^{-1/4}$, and since this belongs to $\cA$,
this same choice of $A$ also achieves the minimum in $\min\{\|X - A\|_{L^2_{\rm KMS}(\sigma)}\ : A \in \cone_\cA\}$.
\end{proof}

\section{Construction of Dirichlet forms on a finite-dimensional von Neumann algebra}
\label{sec:Dirichlet}

Motivated by the results in Section \ref{sec:QMS} and Section~\ref{sec:Beurling--Deny} we introduce a general framework in which various gradient flow structures can be studied naturally. This setting unifies and extends several previous approaches to gradient flows, in particular for reversible Markov chains on finite spaces \cite{Ma11,Mie11b}, the fermionic Fokker-Planck equation \cite{CM12}, and Lindblad equations with detailed balance \cite{2016-Carlen-Maas,MiMi17}

While the results in Section \ref{sec:QMS} show that the  general  QMS satisfying the $\sigma$-DBC can be represented in terms of a Dirichlet form specified in terms of derivations,   our applications require us to work with  representations for the 
generator $\cL$ in terms of  ``partial derivative operators'' $\partial_j$ that are not simply derivations. 
The reason is that, to obtain functional inequalities and sharp rates of convergence to equilibrium, 
it will be important to obtain commutation relations of the form $[\partial_j, \cL] = -a \partial_j$ for 
$a \in \R$.
We shall demonstrate that such commutation relations may hold for the general class of representations introduced in this section, but not for the 
simpler representation in terms of derivations discussed in Section \ref{sec:QMS}.

Our starting point is a finite-dimensional von Neumann algebra $\cA$ which we may regard as a subalgebra of 
$\M_n(\C)$ for some $n\in \N$.  On account of the finite-dimensionality of $\cA$, there is always a tracial positive linear functional $\tau$ on $\cA$: One choice is the normalized trace $\tau[A] = n^{-1}\Tr[A]$. However, if $\cA$ is commutative (hence isomorphic to $\ell_n^\infty$), there will be many other tracial positive linear functionals --- any positive measure on $\{1,\ldots, n\}$ specifies such a positive linear functional.   In what follows, $\tau$ will denote any faithful positive 
linear functional on $\cA$ that is tracial; i.e., such that $\tau[AB] = \tau[BA]$ for all $A,B\in \cA$. 
Since $\tau$ is faithful, every state $\sigma$ on $\cA$ can be represented as $\sigma(A) = \tau[\sigma A]$, where on the right side $\sigma\in \cA \subseteq \M_n(\C)$ is the $n\times n$ density matrix belonging to $\cA$ determined by the state $\sigma$. 

The basic operation in terms of which we shall construct completely Dirichlet forms on $\cA$ has several components.

Let $\cB$ be another finite-dimensional von Neumann algebra with tracial state $\tau_\cB$. 
A unital $*$-homomorphism $\ell$ from $(\cA,\tau)$ to $(\cB,\tau_\cB)$ is  ($\tau,\tau_\cB$)-\emph{compatible} in case for all $A\in \cA$, 
\begin{equation}\label{compat}
\tau_{\cB}[\ell(A)] = \tau[A]\ .
\end{equation}
Equivalently, $\ell$ is $(\tau,\tau_\cB)$-compatible in case its adjoint $\ell^\dagger:L^2(\cB,\tau_\cB) \to L^2(\cA,\tau)$ satisfies
$\ell^\dagger(\one_\cB) = \one_\cA$. 

Let $0 \neq V\in \cB$, and let $\ell$ and $r$ be a pair of $(\tau,\tau_\cB)$-compatible unital $*$-homomorphisms from $\cA$ into $\cB$. Then define the operator $\partial_V: \cA \to \cB$ by
\begin{equation}\label{pardef}
\partial_V A := Vr(A) - \ell(A)V\ .
\end{equation}
If $\cB = \cA$ and both $\ell$ and $r$ are the identity, this reduces to \eqref{simpleder}. 
The following Leibniz rule shows that $\partial_V$ is an $(\ell, \err)$-skew derivation.

\begin{lemma}[Leibniz rule for $\partial_V$]\label{lem:prod-rule}
For $A, B \in \cA$ we have
\begin{align}\label{eq:prod-rule}
 \partial_V (AB) = (\partial_V A )  \err (B) + \ell (A) \partial_V B \ .
\end{align}
\end{lemma}

\begin{proof}
Since $\ell$ and $\err$ are algebra homomorphisms, 
\begin{align*}
 \partial_V (AB)
&  =  V \err(AB) - \ell(AB) V
\\&  =  \big(V \err(A) - \ell(A) V\big) \err(B) 
    + \ell(A)\big( V \err(B) - \ell(B) V\big)
 = (\partial_VA) \err(B) + \ell(A) \partial_V B\ ,   
\end{align*}
which is the desired identity.
\end{proof}

\begin{remark}\label{rem:homo}
Since $\ell$ and $\err$ are algebra $*$-homomorphisms, it follows that
\begin{align}\label{eq:dual-algebra}
 \ell^\dagger\big(\ell(A_1) B \ell(A_2)\big) = A_1 \ell^\dagger(B) A_2 
  \quad \text{and} \quad
 \err^\dagger\big(\err(A_1) B \err(A_2)\big) = A_1 \err^\dagger(B) A_2  
\end{align} 
for all $A_1, A_2 \in \cA$ and $B \in \cB$. 
Moreover, $\ell^\dagger(B)^* = \ell^\dagger(B^*)$ and $\err^\dagger(B)^* = \err^\dagger(B^*)$ for all $B \in \cB$.
\end{remark}

Let $\sigma\in \cA$ be the density matrix (with respect to $\tau$)  of a faithful state on $\cA$. 
Since $\ell$ and $r$ are $(\tau,\tau_\cB)$-compatible, $\ell(\sigma)$ and $r(\sigma)$ 
are density matrices (with respect to $\tau_\cB$ on $\cB$). The inner product 
that we use on $\cB$ is a KMS inner product based on both $\ell(\sigma)$ and $r(\sigma)$ defined in terms of the {\em relative modular operator} $\Delta_{\ell(\sigma),r(\sigma)}$:
\begin{equation}\label{relmod}
\Delta_{\ell(\sigma),r(\sigma)}(B) := \ell(\sigma) B r(\sigma)^{-1}\ .
\end{equation}
It is easily verified that $\Delta_{\ell(\sigma),r(\sigma)}$ is a positive operator on $L^2(\cB,\tau_\cB)$, and hence we may define an inner product on $\cB$ through
\begin{equation}\begin{aligned}\label{^{1/2}}
 \langle B_1,B_2\rangle_{L^2_{\rm KMS}(\cB,\ell(\sigma),\err(\sigma))} 
  &:= \langle  B_1r (\sigma)^{1/2},  
 \Delta_{\ell(\sigma),r(\sigma)}^{1/2}(B_2\err(\sigma)^{1/2})\rangle_{L^2(\cB,\tau_\cB)}\nonumber\\
 &\mbox{\phantom{:}}=
 \tau_\cB[ B_1^* \ell(\sigma^{1/2}) B_2 \err(\sigma^{1/2})] \ .
\end{aligned}\end{equation}

Given a faithful state $\sigma$ on $\cA$,  $V\in \cB$, and {\em two} pairs $(\ell,\err)$ and $(\ell_*,\err_*)$ of ($\tau,\tau_\cB$)-compatible $*$-homomorphisms of $\cA$ into $\cB$, define $\partial_V$ by \eqref{pardef}, and define
$$\partial_{V^*} = V^*r_*(A) - \ell_*(A)V^*$$
in accordance with \eqref{pardef}, but using $V^*$, $\ell_*$ and $\err_*$. 
Then define a sesquilinear form $\sE$ on $\cA$  by 
\begin{equation}\label{skewdirdef}
\sE(A_1,A_2) = \langle \partial_V A_1,\partial_V A_2\rangle_{L^2_{\rm KMS}(\cB,\ell(\sigma),\err(\sigma))} +
\langle \partial_{V^*} A_1,\partial_{V^*} A_2\rangle_{L^2_{\rm KMS}(\cB,\ell_*(\sigma),\err_*(\sigma))}\ .
\end{equation}
Our immediate goal in this section is to determine conditions on $V$, $(\ell,\err)$ and $(\ell_*,\err_*)$ under which $\sE$ is a conservative completely Dirichlet form on $\cA$ equipped with the KMS inner product induced by $\sigma$. 

It is first of all necessary that the operator
$\cL$ determined by $\sE$ through $\sE(A_1,A_2) = -\langle B,\cL A\rangle_{L^2_{\rm KMS}(\sigma)}$ be real; i.e., $(\cL(A))^* = \cL A^*$.  Since 
$\langle A_1, A_2\rangle_{L^2_{\rm KMS}(\sigma)} = 
\langle A_2^*,A_1^*\rangle_{L^2_{\rm KMS}(\sigma)}$ for all $A_1,A_2\in \cA$, it is easily seen that 
$\cL$ is real if and only if $\sE(A_1,A_2) = \sE(A_2^*,A_1^*)$ for all $A_1,A_2\in \cA$.

\begin{lemma}\label{lem:symm}  
Under the condition that for all $A_1, A_2 \in \cA$,
\begin{align}\label{eq:symmetryA}
\tau_\cB[ V^* \ell (A_1) V \err(A_2) ] 
 =  \tau_\cB[  V^* \err_{*} (A_1) V \ell_{*} (A_2)  ] \ ,
\end{align}
we have $\sE(A_1, A_2) = \sE(A_2^*, A_1^*)$ for all $A_1, A_2\in \cA$.
\end{lemma}

\begin{remark} One can satisfy \eqref{eq:symmetryA} in a trivial way by taking $\ell$, $r$, $\ell_*$ and $\err_*$ each to be the identity.  
Almost as trivially, one may take $\ell_* = \err$ and $\err_* = \ell$. 
However, we shall see that one can also satisfy \eqref{eq:symmetryA} with $\ell_* = \ell$ and $\err_* = \err = I_\cB$ with a non-trivial $*$-homomorphism $\ell$; see the discussion in the next section on the Clifford Dirichlet form. Other non-trivial realizations of \eqref{eq:symmetryA} arise in practice.
\end{remark}

\begin{proof}[Proof of Lemma~\ref{lem:symm}]We compute
\begin{eqnarray}
\langle \partial_V A_1,\partial_V A_2\rangle_{L^2_{\rm KMS}(\cB,\ell(\sigma),\err(\sigma))} 
&=& \tau_\cB[\err(A_1^*)V^*\ell(\sigma^{1/2})V\err(A_2)\err(\sigma^{1/2})]\label{pos1}\\
&+& \tau_\cB[V^*\ell(A_1^*)\ell(\sigma^{1/2})\ell(A_2)V\err(\sigma^{1/2})]\label{pos2}\\
&-& \tau_\cB[\err(A_1^*)V^*\ell(\sigma^{1/2})\ell(A_2)V\err(\sigma^{1/2})]\nonumber\\
&-& \tau_\cB[V^*\ell(A_1^*)\ell(\sigma^{1/2})V\err(A_2)\err(\sigma^{1/2})] \ .\nonumber
\end{eqnarray}
By cyclicity of the trace $\tau_\cB$, the homomorphism property of $\ell$ and 
$\err$, and \eqref{eq:symmetryA},
\begin{eqnarray*}
\tau_\cB[\err(A_1^*)V^*\ell(\sigma^{1/2})V\err(A_2)\err(\sigma^{1/2})] &=&
\tau_\cB[\err(A_2\sigma^{1/2}A_1^*)V^*\ell(\sigma^{1/2})V] \\
&=&\tau_\cB[\ell_*(A_2\sigma^{1/2}A_1^*)V^*\err_*(\sigma^{1/2})V]\\
&=&\tau_\cB[V\ell_*(A_2)\ell_*(\sigma^{1/2})\ell_*(A_1^*)V^*\err_*(\sigma^{1/2})] \ .
\end{eqnarray*}
This shows that the quantity in \eqref{pos1} is what we obtain from the quantity in \eqref{pos2} if we replace $\ell$ by $\ell_*$, $\err$ by $\err_*$,  $V$ by $V^*$, $A_1$ by $A_2^*$, and $A_2$ by $A_1^*$.  
Similar computations then yield the identity
$$\langle \partial_V A_1,\partial_V A_2\rangle_{L^2_{\rm KMS}(\cB,\ell(\sigma),\err(\sigma))}  = 
\langle \partial_{V^*} A_2^*,\partial_{V^*} A_1^*\rangle_{L^2_{\rm KMS}(\cB,\ell_*(\sigma),\err_*(\sigma))} \ ,$$
and this implies $\sE(A_1,A_2) = \sE(A_2^*,A_1^*)$.
\end{proof}

Thus, the condition \eqref{eq:symmetryA} suffices to ensure that the sesquilinear form $\sE$ defined in \eqref{skewdirdef} is real. 
In the rest of this section, we suppose that this condition is satisfied, and then since $\sE$ is real, it suffices to consider its bilinear restriction to $\cA_h$. 

One further condition is required to ensure that $\sE$ be a Dirichlet form on $\cA_h$, and we shall see that under this same condition $\sE$ is actually a completely Dirichlet form. 
The assumption is that $V$ (resp. $V^*$) is an eigenvector of the relative modular operator $\Delta_{\ell(\sigma),r(\sigma)}$ (resp. $\Delta_{\ell_*(\sigma),r_*(\sigma)}$).
Since the relative modular operator is positive, there exist $\omega, \omega_*\in \R$ such that
\begin{equation}\label{modeig}
\Delta_{\ell(\sigma),r(\sigma)}V = e^{-\omega} V 
\quad \text{and} \quad 
\Delta_{\ell_*(\sigma),r_*(\sigma)}V^* = e^{-\omega_*} V^* \ .
\end{equation}
There are several equivalent formulations of this condition that will be useful.

\begin{lemma}\label{prop:symmetry-properties} The first condition in \eqref{modeig} is equivalent to the condition
\begin{align} \label{eq:modular-ev}
\partial_V  \log \sigma = \omega V \ ,
\end{align}
and to the condition that 
for all $t \in \R$,
\begin{align} \label{eq:modular-group}
	\Delta^t_{\ell(\sigma),r(\sigma)} V
  = e^{- t \omega}  V \ .
\end{align}
Moreover, \eqref{modeig} implies that 
\begin{align} \label{eq:omega-symmetry}
\omega_* = - \omega \ .
\end{align}
\end{lemma}

\begin{proof}
Note that $( \Delta^t_{\ell(\sigma),r(\sigma)})_{t\in\R}$ is a group of linear operators on $\cB$, and the generator $\cG$ of this group  is given by $\cG  B = \ell(\log \sigma) B - B \err(\log \sigma)$, thus $\cG V = - \partial_V \log \sigma$.
The equivalences thus follow from basic spectral theory.

Using \eqref{eq:symmetryA} with $A_1 = \sigma$ and $A_2 = \sigma^{-1}$, and two applications of \eqref{eq:modular-group}, we obtain
\begin{align*}
	e^{-\omega} \tau_\cB[ V^* V] 
	& = \tau_\cB[ V^* \ell(\sigma) V \err(\sigma^{-1}) ] 
	= \tau_\cB[ V \ell_*(\sigma^{-1}) V^* \err_*(\sigma) ] 
	= e^{\omega_*} \tau_\cB[ V V^*] \ .
\end{align*}
Since $V \neq 0$, this yields \eqref{eq:omega-symmetry}.
\end{proof}

We are now ready to state the main result of this section.

\begin{theorem}\label{thm:Dirichlet} Let $\sigma$ be a faithful state on $\cA$. 
Let $V\in \cB$ and two pairs $(\ell,\err)$ and $(\ell_*,\err_*)$ of $(\tau,\tau_\cB)$-compatible $*$-homomorphisms be given. 
Suppose also that \eqref{eq:symmetryA} is satisfied, and suppose that $V$ (resp. $V^*$) is an eigenvector of the relative modular operator $\Delta_{\ell(\sigma),r(\sigma)}$ (resp. $\Delta_{\ell_*(\sigma),r_*(\sigma)}$) satisfying \eqref{modeig}. 
Then the sesquilinear form $\sE : \cA \times \cA \to \C$ given by \eqref{skewdirdef} defines a conservative completely Dirichlet form on $L_{\rm KMS}^2(\cA_h,\sigma)$.
\end{theorem}

\begin{proof}
To explain the crucial role of the assumption that $V$ is an eigenvector of the relative modular operator, so that \eqref{modeig} is satisfied,  we fix $V, W \in \cB$ and (temporarily) define the operators $\partial, \partial_*: \cA \to \cB$ by $\partial A :=  V \err(A) - \ell(A) W$ and $\partial_* A :=  V^* \err_*(A) - \ell_*(A) W^*$, and set
\begin{align*}
\sE(A_1,A_2) = \langle \partial A_1,\partial A_2\rangle_{L^2_{\rm KMS}(\cB,\ell(\sigma),\err(\sigma))} +
\langle \partial_* A_1,\partial_* A_2\rangle_{L^2_{\rm KMS}(\cB,\ell_*(\sigma),\err_*(\sigma))}\ .
\end{align*}
 We will show:
\begin{enumerate}
\item If $W = e^{\omega/2} \Delta^{1/2}_{\ell(\sigma),r(\sigma)} V$ and $W^* = e^{\omega_*/2} \Delta^{1/2}_{\ell_*(\sigma),r_*(\sigma)} V^*$ for some $\omega, \omega_* \in \R$, then $\sE$ defines a Dirichlet form on $L_{\rm{KMS}}^2(\cA_h,\sigma)$.
\item If, in addition, \eqref{modeig} holds, then  
$\sE(\one,A) = 0$ for all $A \in \cA_h$, hence $\sE$ is conservative.
\end{enumerate}
 
Consider the unitary transformation $\cU : L_{\rm KMS}^2(\cA,\sigma) \to  L^2(\cA, \tau)$ given by 
$\cU A := \sigma^{1/4} A \sigma^{1/4}$. For brevity we write $\cT B := \Delta^{1/4}_{\ell(\sigma),r(\sigma)} B = 
\ell(\sigma^{1/4}) B \err(\sigma^{-1/4})$, and likewise, $\cT_* B := \Delta^{1/4}_{\ell_*(\sigma),r_*(\sigma)} B = 
\ell_*(\sigma^{1/4}) B \err_*(\sigma^{-1/4})$.

For $A \in \cA_h$ we need to show that $\sE(A_+, A_-) \leq 0$. 
For $A_1, A_2 \in \cA$ we have
\begin{equation}\begin{aligned}\label{eq:partialA12}
	  &\langle  \partial A_1, \partial A_2\rangle_{L^2_{\rm KMS}(\cB,\ell(\sigma),\err(\sigma))}
	\\& = \tau_\cB\Big[ 
			\err (\sigma^{1/4}) 
			\Big( \err (A_1^* ) V^*  -  W^* \ell (A_1^*) \Big)
			\ell (\sigma^{1/2}) 
			 \Big(V  \err (A_2) - \ell (A_2) W \Big) 
			 \err (\sigma^{1/4}) \Big]			 
	\\& =  \tau_\cB\Big[ 
			\Big( \err (\cU A_1^* ) (\cT  V )^* 
				- (\cT^{-1} W )^* \ell (\cU A_1^*) \Big)
			 \Big( ( \cT V ) \err (\cU A_2) 
			 		 -  \ell (\cU A_2) \cT ^{-1} W   \Big)  \Big]	
	\\& = 
	\tau_\cB\Big[ 
			 (\cT  V  ) \err (\cU A_2 \cU A_1^* ) (\cT  V )^* 
		-
	 (\cT  V )^* \ell (\cU A_2) \cT ^{-1} W   \err (\cU A_1^* )
	\\& \qquad\qquad	-
	 \err (\cU A_2) (\cT^{-1} W )^* \ell (\cU A_1^*) (\cT  V  )
		+			 
		(\cT^{-1} W )^* \ell (\cU A_1^* \cU A_2) \cT ^{-1} W  
							  \Big]	\ .
\end{aligned}\end{equation}
For $A \in \cA_h$ we have $A_\pm = \cU^{-1}(\cU A)_{(\pm)}$ by Lemma \ref{closestMA}, thus
\begin{align*}
	\cU A_+ \cU A_-
	  = (\cU A)_{(+)}(\cU A)_{(-)} = 0 
	  \quad \text{and} \quad
	  	\cU A_- \cU A_+
	  = (\cU A)_{(-)}(\cU A)_{(+)} = 0 \ .
\end{align*}
We obtain
\begin{equation}\begin{aligned}\label{abd34}
	\langle \partial A_+, \partial A_-\rangle_{L^2_{\rm KMS}(\cB,\ell(\sigma),\err(\sigma))}
		& = - 	\tau_\cB\Big[ 
	 (\cT  V )^* \ell (\cU A_-) \cT ^{-1} W   \err (\cU A_+ )
		\\& \qquad\qquad	+
	  (\cT ^{-1} W )^* \ell (\cU A_+) (\cT  V  ) \err (\cU A_-)
							  \Big]	\ .
\end{aligned}\end{equation}	
Since $\err(\cU A_{\pm}) \geq 0$, it follows that 
$\langle \partial A_+, \partial A_-\rangle_{L^2_{\rm KMS}(\cB,\ell(\sigma),\err(\sigma))}\leq 0$ if we can show that 
\begin{align}\label{eq:positivity}
 (\cT V)^* \ell(\cU A_-) \cT^{-1} W \geq 0
 	  \quad \text{and} \quad
	  (\cT^{-1} W)^* \ell(\cU A_+) (\cT V ) \geq 0 \ .
\end{align}
To show this, we make the assumption that $W = e^{\omega/2} \Delta_{\ell(\sigma),r(\sigma)}^{1/2} V$ for some 
$\omega \in \R$. Equivalently, this means that $\cT^{-1} W = e^{\omega/2} \cT V$, and since 
$\ell(\cU A_\pm) \geq 0$, we obtain \eqref{eq:positivity}. This proves that 
$\langle \partial A_+, \partial A_-\rangle_{L^2_{\rm KMS}(\cB,\ell(\sigma),\err(\sigma))} \leq 0$.

An entirely analogous argument shows that  
$\langle \partial_* A_+, \partial_* A_- \rangle_{L^2_{\rm KMS}(\cB,\ell_*(\sigma),\err_*(\sigma))} \leq 0$, and this proves that 
$\sE(A_+,A_-)$ is a Dirichlet form. 

Observe now that $\partial \one = V - W$ and $\partial_* \one = V^* - W^*$. Thus, to conclude that $\partial \one = \partial_* \one = 0$, we need to assume that $V$ is an eigenvector of $\Delta_{\ell(\sigma),r(\sigma)}$ with eigenvalue $e^{-\omega }$, and that $V^*$ is an eigenvector of $\Delta_{\ell_*(\sigma),r_*(\sigma)}$ with eigenvalue $e^{-\omega_*}$. It immediately follows that $\sE(\one,A) = 0$ for all $A \in \cA_h$, hence $\sE$ is conservative. 

\medskip

It remains to prove that under the given conditions, $\sE$ is completely Dirichlet. 
Let $\Tr$ be the standard trace on $\M_m(\C)$.
Let ${\bf H}$ be a self-adjoint element of $\cA \otimes \M_m(\C)$, and let ${\bf H}_+$ and ${\bf H}_-$ be the elements of its decomposition ${\bf H} = {\bf H}_+ - {\bf H}_-$ in $L^2_{\rm KMS}(\sigma\otimes \Tr)$, where ${\bf H}_+$ and ${\bf H}_-$ are positive and such that $ \langle {\bf H}_+,{\bf H}_-\rangle_{L^2_{\rm KMS}(\sigma\otimes \Tr)} =0$.

Let ${\bm \sigma} = \sum_{j=1}^m \sigma \otimes E_{jj}$ and write $\widetilde{\bf H} = \bm \sigma^{1/4} {\bf H}  \bm \sigma^{1/4}$ for brevity.  By Theorem~\ref{closestMN}, ${\bf H}_+ = \bm \sigma^{-1/4} \widetilde{\bf H}_{(+)} \bm \sigma^{-1/4}$, hence $[{\bf H}_+]_{ij} = \sigma^{-1/4} [\widetilde{\bf H}_{(+)}]_{ij} \sigma^{-1/4}$. It follows that 
\begin{align*}
\sum_{i,j =1}^m \tau\big[ \cU( [{\bf H}_+]_{ij} ) \cU( [{\bf H}_-]_{ij} )\big]
  = \ip{\widetilde{\bf H}_{(+)} , \widetilde{\bf H}_{(-)} }_{L^2(\tau \otimes \Tr)} 
  = 0 \ .
\end{align*}
Using this identity, \eqref{abd34} with $V = W$ yields
\begin{align*}
&\sum_{i, j = 1}^m \langle  \partial [{\bf H}_+]_{ij}, \partial [{\bf H}_-]_{ij}\rangle_{L^2_{\rm KMS}(\cB,\ell(\sigma),\err(\sigma))}
 \\ & = 
    - \sum_{i, j = 1}^m
    	\tau_\cB\Big[ 
	  V^* \ell (\cU [{\bf H}_-]_{ij}) V \err (\cU [{\bf H}_+]_{ij})
	+
	  V^* \ell (\cU [{\bf H}_+]_{ij}) V \err (\cU [{\bf H}_-]_{ij})
							  \Big]	
 \\ & = 
    - \sum_{i, j = 1}^m
    	\tau_\cB\Big[ 
	  V^* \ell ([\widetilde{\bf H}_{(-)}]_{ij}) V \err ([\widetilde{\bf H}_{(+)}]_{ij})
	+
	  V^* \ell ([\widetilde{\bf H}_{(+)}]_{ij}) V \err ([\widetilde{\bf H}_{(-)}]_{ij})
							  \Big]	 \ 
 \\ & = -  \tau_\cB \otimes \Tr  
\Big[ 
	  (V \otimes \one_m)^* {\ell} (\widetilde{\bf H}_{(-)}) (V \otimes \one_m) {\err} (\widetilde{\bf H}_{(+)})
	{\,+\,}
 (V \otimes \one_m)^* {\ell} (\widetilde{\bf H}_{(+)}) (V \otimes \one_m) {\err} (\widetilde{\bf H}_{(-)})
							  \Big],
\end{align*}
where $\one_m$ denotes the identity matrix in $\M_m(\C)$, and in the last line, we simply write $\ell$ and $\err$ to denote their canonical extensions  $\ell \otimes I$ and $\err \otimes I$. Since ${\err} (\widetilde{\bf H}_{(\pm)}) \geq 0$ and $ (V \otimes I)^* {\ell} (\widetilde{\bf H}_{(\mp)}) (V \otimes \one_m) \geq 0$, it is now evident that the right-hand side is non-positive. An analogous argument applies if we replace $\partial$ by $\partial_*$, and therefore,
\begin{align*}
 \sE^{(m)}({\bf H}_+, {\bf H}_-) 
  	= \sum_{i,j = 1}^m  \sE([{\bf H}_+]_{ij}, [{\bf H}_-]_{ij}) \leq 0 \ .
\end{align*}
In summary, this proves that $\sE^{(m)}$ is a Dirichlet form for all $m\in \N$, and hence that $\sE$ is completely Dirichlet. 
\end{proof}

Evidently, the sum of a finite set of conservative completely Dirichlet forms on $\cA$ is a 
conservative completely Dirichlet form. Thus, we may construct a large class of conservative completely 
Dirichlet forms by taking sums of forms of the type considered in Theorem~\ref{thm:Dirichlet}.  
In the remainder of this section, we consider such a conservative, completely Dirichlet form and 
the associated QMS $\cP_t$. 

It will be convenient going forward to streamline our notation. In the rest of this section we are working in the framework specified as follows:

\begin{definition}\label{ass:setup}
Let $\cA$ be a finite-dimensional von Neumann algebra $\cA$ endowed with a faithful tracial positive linear functional $\tau$. A {\em differential structure} on $\cA$ consists of the following:
 
\begin{enumerate}
\item A finite index set $\cJ$, and for each $j\in \cJ$, a finite
dimensional von Neumann algebra $\cB_j$ endowed with a faithful tracial positive linear functional $\tau_j$.
\item For each $j\in \cJ$, a pair $(\ell_j,\err_j)$ of unital $*$-homomorphisms from $\cA$ to $\cB_j$ such that for each $A\in \cA$ and each $j\in \cJ$, $\tau_j(\ell_j(A)) = \tau_j(\err_j(A)) = \tau(A)$, and a non-zero $V_j\in \cB_j$. 
\item  It is further required that for each $j\in \cJ$, there is a unique $j^*$ such that $V_j^* = V_{j^*}$, hence $\{ V_j \}_{j \in \cJ} = \{ V_j^* \}_{j \in \cJ}$ and $\cB_{j^*} = \cB_j$.  Moreover, for $j\in \cJ$ and $A_1, A_2 \in \cA$,
\begin{align}\label{eq:symmetry-equiv}
\tau_j[ V_j^* \ell_j (A_1) V_j \err_j(A_2) ] 
 =  \tau_j[  V_j^* \err_{j^*}(A_1) V_j \ell_{j^*} (A_2)  ]  \ .
\end{align}
\item An invertible density matrix $\sigma \in \Dens_+$, such that, for each $j\in \cJ$, $V_j$ is an eigenvector of the relative modular operator $\Delta_{\ell_j(\sigma),\err_j(\sigma)}$ on $\cB_j$
with
\begin{equation}\label{relmod3}
\Delta_{\ell_j(\sigma),\err_j(\sigma)}(V_j) = e^{-\omega_j}V_j 
\end{equation}
for some $\omega_j \in \R$.
\end{enumerate}
Then for each $j\in \cJ$, we define the linear operator $\partial_j: \cA \to \cB_j$ by
\begin{align}\label{eq:partial-j}
 \partial_{j} A := V_j \err_j(A) - \ell_j(A) V_j
\end{align}
for $A \in \cA$, and set
\begin{align*}
\nabla A := (\partial_j A)_{j \in \cJ} \in \cB\ ,
\qquad \cB =\prod_{j \in \cJ}  \cB_j\ .
\end{align*}
We refer to $\nabla A$ as the gradient of $A$, or derivative of $A$, with respect to the differential structure on $\cA$ defined above. 
We will denote the differential structure by the triple $(\cA, \nabla, \sigma)$.
\end{definition}

For $s \in [0,1]$ we endow $\cB_j$ with the inner product 
\begin{align*}
\ip{B_1, B_2}_{s,j} := \tau_j[B_1^* \ell_j(\sigma^s) B_2 \err_j(\sigma^{1-s})] \ .
\end{align*}
The most relevant case for our purposes is $s = \frac12$, in which case we write
\begin{align*}
\KMSj{B_1, B_2} := \ip{B_1, B_2}_{1/2,j} \ .
\end{align*}
It follows immediately from Theorem~\ref{thm:Dirichlet} that 
\begin{equation}\label{eq:Dirichlet}
\sE(A_1,A_2) := 
	  \sum_{j \in \cJ} \KMSj{\partial_{j} A_1, \partial_{j} A_2} 
	 \end{equation}
is a conservative completely Dirichlet form on $L^2_{\rm KMS}(\cA_h, \sigma)$. 

\begin{remark} As we have seen earlier in this section, {\it (3)} ensures that the sesquilinear form $\sE$ defined by 
\eqref{eq:Dirichlet} is real  and leads to the symmetry condition \eqref{eq:omega-symmetry}, and then {\it (4)} ensures that $\sE$ is completely Dirichlet.
\end{remark}

Having the gradient $\nabla$ at our disposal, we can define a corresponding divergence operator by trace duality. 
For $\BB = (B_j)_{j\in \cJ} \in \cB$ we shall use the notation
\begin{align}\label{eq:def-div}
 \dive \BB = - \sum_{j\in \cJ}   \partial_j^\dagger B_j \ .
\end{align}

\begin{proposition}\label{prop:sigma-adjoint}
Let $s \in [0,1]$.
The adjoint of the differential operator $\partial_j :  (\cA,\ip{\cdot,\cdot}_s) \to  (\cB_j,\ip{\cdot,\cdot}_{s,j})$ is given by 
\begin{align}\label{eq:sigma-adjoint-s}
	 \partial_{j,\sigma}^{\dagger,(s)} B
  = e^{-s\omega_j} \err_j^\dagger(V_j^*B)
       - e^{(1-s)\omega_j} \ell_j^\dagger(B V_j^*) \ .
\end{align}
In particular, the adjoint of the operator $\partial_j :  L_{\rm KMS }^2(\cA,\sigma) \to  L_{{\rm KMS},j}^2(\cB_j,\sigma)$ is given by 
\begin{align}\label{eq:sigma-adjoint}
	 \partial_{j,\sigma}^\dagger B
  = e^{-\omega_j/2} \err_j^\dagger(V_j^*B)
       - e^{\omega_j/2} \ell_j^\dagger(B V_j^*)
\end{align}
for $B \in \cB_j$.
\end{proposition}

\begin{proof}
For $A \in \cA$ we obtain using \eqref{eq:dual-algebra} and \eqref{eq:modular-group},
\begin{align*}
	&  \ip{\partial_j A , B}_{s,j}
\\&  =  \tau_j\big[ \big(V_j \err_j(A) - \ell_j(A) V_j\big)^* \ell_j(\sigma^{s}) B \err_j(\sigma^{1-s})\big] 
\\&
  =  \tau_j\big[ \err_j(A)^*V_j^*  \ell_j(\sigma^{s}) B \err_j(\sigma^{1-s}) 
			- \ell_j(A)^*   \ell_j(\sigma^{s}) B \err_j(\sigma^{1-s}) V_j^* \big] 
\\&  =  \tau\big[ A^* \err_j^\dagger\big(V_j^*  \ell_j(\sigma^{s}) B \err_j(\sigma^{1-s})\big) 
	 - A^* \ell_j^\dagger\big(  \ell_j(\sigma^{s}) B \err_j(\sigma^{1-s}) V_j^*\big) \big] 
\\&  =  \tau\Big[ A^* \sigma^{s} \Big( \err_j^\dagger\big( \err_j(\sigma^{-s})V_j^*  \ell_j(\sigma^{s}) B \big) 
	 -  \ell_j^\dagger\big(  B \err_j(\sigma^{1-s}) V_j^* \ell_j(\sigma^{s-1})\big) \Big) \sigma^{1-s} \Big]
\\&  =  \tau \big[  A^* \sigma^{s} \big( e^{-s\omega_j} \err_j^\dagger(V_j^* B)
	 -  e^{(1-s)\omega_j} \ell_j^\dagger(B V_j^*) \big) \sigma^{1-s} \big] 
\\&  =  \ip{A,  e^{-s\omega_j}   \err_j^\dagger(V_j^* B)
	 -  e^{(1-s)\omega_j} \ell_j^\dagger(B V_j^*) }_s \ ,
\end{align*}
which proves \eqref{eq:sigma-adjoint-s}.
\end{proof}

The following result provides an explicit expression for $\cL$.

\begin{proposition}\label{prop:generator-formula}
The operator $\cL$ associated to the Dirichlet form \eqref{eq:Dirichlet} is given by
\begin{align*}
 \cL A 
& =  \sum_{j \in \cJ}
  	e^{-\omega_j/2} \err_j^\dagger\Big( - \err_j(A) V_j^* V_j + 2 V_j^* \ell_j(A) V_j  - V_j^*  V_j \err_j(A) \Big)
 \\& = \sum_{j \in \cJ} 
  	e^{\omega_j/2} \ell_j^\dagger\Big(V_j \err_j(A) V_j^* - \ell_j(A) V_j V_j^* \Big) 
  -  e^{-\omega_j/2} \err_j^\dagger\Big(V_j^*  V_j \err_j(A) - V_j^*\ell_j(A) V_j\Big)
\end{align*}
for $A \in \cA$. Its Hilbert space adjoint with respect to $L^2(\cA, \tau)$ is given by
\begin{align*}
 \cL^\dagger \rho 
 & 
 = \sum_{j \in \cJ} 
  	e^{-\omega_j/2} \Big(  - \err_j^\dagger\big(\err_j(\rho) V_j^* V_j \big) + 2 \ell_j^\dagger\big(V_j \err_j(\rho) V_j^*\big) - \err_j^\dagger\big( V_j^* V_j \err_j(\rho) \big)  \Big)
 \\& 
 = \sum_{j \in \cJ}
  	e^{\omega_j/2} \Big( \err_j^\dagger\big(V_j^* \ell_j(\rho) V_j\big) - \ell_j^\dagger\big(\ell_j(\rho) V_j V_j^* \big)  \Big)
  \\& \qquad\quad  
  -  e^{-\omega_j/2} \Big( \err_j^\dagger \big(V_j^*  V_j \err_j(\rho) \big) 
  					 - \ell_j^\dagger \big( V_j \err_j(\rho) V_j^* \big) \Big)
\end{align*}
for $\rho \in \cA$. 
\end{proposition}

\begin{proof}
Using Proposition \ref{prop:sigma-adjoint} we obtain
\begin{align*}
  \cL A & = - \sum_{j \in \cJ}  \partial_{j,\sigma}^\dagger \partial_j A
  	 \\ & = - \sum_{j \in \cJ} e^{-\omega_j/2} \err_j^\dagger\big(V_j^*\partial_j A\big)
		       - e^{\omega_j/2}  \ell_j^\dagger\big((\partial_j A) V_j^*\big) 
	 \\ & = - \sum_{j \in \cJ} e^{-\omega_j/2} \err_j^\dagger\big(V_j^*  V_j \err_j(A) - V_j^*\ell_j(A) V_j\big)
		       - e^{\omega_j/2}  \ell_j^\dagger\big(V_j \err_j(A) V_j^* - \ell_j(A) V_j V_j^* \big) \ ,
\end{align*}
which yields the second expression for $\cL$. The first expression is  obtained using \eqref{eq:symmetry-equiv} and the fact that $\omega_{j^*} = - \omega_j$.
The formulas for $\cL^\dagger$ follow by direct computation.
\end{proof}

The following result is an immediate consequence.

\begin{proposition}\label{prop:domains-Dirichlet}
We have
\begin{align*}
  \Ker(\cL) = \Ker(\nabla)   \quad \text{and} \quad
  \Ran(\cL^\dagger) = \Ran(\dive) \ .
\end{align*}
\end{proposition}

\begin{proof}
The identity $\cL A = - \sum_{j \in \cJ}  \partial_{j,\sigma}^\dagger \partial_j A
$ implies that $\Ker(\nabla) \subseteq \Ker(\cL)$. The reverse inclusion follows from the identity $- \KMS{\cL A, A} = \sum_{j \in \cJ} \KMSj{\partial_j A, \partial_j A}$. The identification of the ranges is a consequence of duality.
\end{proof}

\begin{proposition}\label{prop:D-detailed-balance}
For $s \in [0,1]$ and $A_1, A_2 \in \cA$ we have the identity
\begin{align*}
 - \ip{\cL A_1, A_2}_s = \sum_{j \in \cJ} e^{(s-\frac12)\omega_j} \ip{\partial_j A_1, \partial_j A_2}_{s,j} \ .
\end{align*}
Consequently, the operator $\cL$ is self-adjoint with respect to $\ip{\cdot, \cdot}_s$ for all $s \in [0,1]$, and in particular, the detailed balance condition holds in the sense of Definition \ref{detbaldef}.
\end{proposition}

\begin{proof}
This follows from a direct computation using \eqref{eq:sigma-adjoint-s}.
\end{proof}

\section{Examples} 
\label{sec:examples}

We provide a number of examples of conservative completely Dirichlet forms defined in the context of a differential structure on a finite-dimensional von Neumann algebra $\cA$ equipped with a faithful state $\sigma$.

\subsection{Generators of quantum Markov semigroups in Lindblad form} We have seen in Section \ref{sec:QMS} that generators of quantum Markov semigroups satisfying detailed balance (see Theorem \ref{thm:structure}) naturally fit into the framework of Section \ref{sec:Dirichlet} by taking $\cA = \cB_j = B(\sH)$ and $\ell_j = \err_j = I_\cA$.

The framework also includes quantum Markov semigroups on subalgebras $\cA$ of $B(\sH)$. In this case we set $\cB_j = B(\sH)$, so that the situation in which $V_j \notin \cA$ is covered. Such a situation also arises naturally in the following example.

\subsection{Classical reversible Markov chains in the Lindblad framework}
For $n \geq 2$, Let $\{e_1, \ldots, e_n\}$ be an orthonormal basis of $\R^n$ and set $E_{k p} = \phys{e_k}{e_p}$. 
Note that $E_{k p}E_{r s} = \delta_{p r} E_{ks}$ and $E_{k p}^* = E_{p k}$.
We consider the algebra $\cA \subseteq \M_n(\C)$ consisting of all operators that are diagonal in the basis given by the $e_i$'s: 
\begin{align*}
	\cA = \bigg\{ \sum_{i = 1}^n \psi_i E_{ii} \ : \ 
			\psi_1, \ldots, \psi_n \in \C \bigg\} \	.
\end{align*}
Furthermore, for each $k, p$, we set $\cB_{k p} = \M_n(\C)$, and we endow $\cA$ and $\cB_{k p}$ with the usual normalized trace given by $\tau(B) = \frac1n \sum_i \ip{B e_i, e_i}$.
Let $\ell_{k p} = \err_{k p}$ be the canonical embedding from $\cA$ into $\cB_{k p}$. It then follows that $\ell_{k p}^\dagger(B) = \err_{k p}^\dagger(B) = \sum_{i} \ip{Be_i,e_i}E_{ii}$.

For $k \neq p$, let $q_{k p} \geq 0$ be the transition rate of a continuous-time 
Markov chain on $\{1, \ldots, n\}$. We set $V_{k p} = 2^{-1/2}(q_{k p} q_{p k})^{1/4} 
E_{k p}$ so that $V_{k p}^* = V_{p k}$. Moreover, it is immediate to see that the identity in \eqref{eq:symmetry-equiv} holds. 
Fix positive weights $\pi_1, \ldots, \pi_n$. It then follows that $\sigma = \sum_i \pi_i E_{ii}$ satisfies \eqref{relmod3} with $\omega_{k p} = \log(\pi_p /\pi_k)$.

By Proposition \ref{prop:generator-formula}, the operator $\cL$ associated to the Dirichlet form \eqref{eq:Dirichlet} is given by
\begin{align*}
 \cL A 
 & = \frac12 \sum_{k \neq p} \sqrt{\frac{q_{k p} q_{p k} \pi_k}{\pi_p}}
    \Big(  E_{k p}^*[A, E_{k p}] 
 + [E_{k p}^*, A] E_{k p}  \Big) 
\end{align*} 
for $A \in \cA$. 
Assume now that $\pi_1, \ldots, \pi_n$ satisfy the classical detailed balance condition, i.e.,  $\pi_k q_{kp} = \pi_p q_{p k}$ for all $k, p$. Then we have 
\begin{align*}
 \cL A 
 = \frac12 \sum_{k \neq p} q_{p k}   
 		 \Big(  E_{k p}^*[A, E_{k p}] 
 + [E_{k p}^*, A] E_{k p}  \Big) \ .
 \end{align*}
More explicitly, 
\begin{align*}
\cL\Big(\sum_i \psi_i E_{ii}\Big) = \sum_{k,p} q_{k p} (\psi_p - \psi_k) E_{kk} \ .
\end{align*}
Hence, under the identification $(\psi_1, \ldots, \psi_n) \leftrightarrow \sum_{i = 1}^n \psi_i E_{ii}$, the operator $\cL$ corresponds to the operator $\cL_{\rm M}$ given by $(\cL_{\rm M} \psi)_k = \sum_{p} q_{kp} (\psi_p - \psi_k)$, which is 
the generator of the continuous-time Markov chain on $\{1, \ldots, n\}$ with transition rates from $k$ to $p$ given by $q_{kp}$.

\subsection{Another approach to reversible Markov chains}\label{ex:Markov}

Let us now give an alternative way to put reversible Markov chains in the framework of this paper, which corresponds to the construction in \cite{Ma11,Mie11a}.
As above, let $q_{kp} \geq 0$ be the transition rate of a continuous-time Markov chain on $\{1, \ldots, n\}$, and assume that the positive weights $\pi_1, \ldots, \pi_n$ satisfy the detailed balance condition $\pi_k q_{kp} = \pi_p q_{p k}$.
Let $\cJ := \{ (k,p) \ : \ q_{kp} > 0 \}$ be the edge set of the associated graph.
We consider the (non-)commutative probability spaces $(\cA,\tau)$ and $(\cB_{kp}, \tau_{kp})$ given by
\begin{align*}
 \cA := \ell_n^\infty\ ,
  \quad  \tau(A) := \sum_{i=1}^n A_i \pi_i\ , 
  \qquad \cB_{kp} = \C\ ,
  \quad  \tau_{kp}(B) := \frac{B}{2} \pi_k q_{kp} \ .
\end{align*} 
The operators $\partial_{kp}$ are determined by
$V_{kp} = 1$, $\ell_{kp} (A) = A_k$, and $\err_{kp} (A) = A_p$ for $A \in \ell_n^\infty$.
It follows that $\ell_{kp}^\dagger(B) = \frac{B}{2} q_{kp} e_k$ and  $\err_{kp}^\dagger(B) = \frac{B}{2} q_{pk} e_p$, where $e_k$ is the $k$'th unit vector in $\ell_n^\infty$. Therefore, 
\begin{align*}
 \partial_{kp} A = A_p - A_k \quad \text{and} \quad
 \partial_{kp}^\dagger B = 
    \frac{B}{2} ( q_{pk} e_p - q_{kp} e_k ) \ .
\end{align*}
Moreover, as $\sigma = \one$ satisfies  \eqref{relmod3} with $\omega_{kp} = 0$, it is readily checked that this defines a differentiable structure in the sense of Definition \ref{ass:setup}. 
Using Proposition \ref{prop:generator-formula}, we infer that the operator $\cL$ is given by
\begin{align*}
 (\cL A)_k &=  \sum_{p} q_{kp} (A_p - A_k) \ ,
\end{align*}
so that $\cL$ is indeed the generator of the continuous time Markov chain with transition rates $q_{kp}$.

\subsection{The discrete hypercube}

For a given Markov chain generator, there are different ways to write the generator in the framework of this paper, and it is often useful to represent $\cL$ using set $\cJ$ that is smaller than in Example \ref{ex:Markov}; see also \cite{FM16}. 
We illustrate this for the simple random walk on the discrete hypercube $\cQ^n = \{-1, 1\}^n$. Set $\cJ = \{1, \ldots, n\}$, and let $s_j:\cQ^n \to \cQ^n$ define the $j$-th coordinate swap defined by $s_j (x_1,  \ldots, x_n) = (x_1, \ldots, - x_j, \ldots, x_n)$.

Consider the (non-)commutative probability spaces $(\cA,\tau)$ and $(\cB_j, \tau_j)$ determined by
\begin{align*}
 \cA := \ell^\infty(\cQ^n)\ ,
  \quad  \tau(A) := 2^{-n}\sum_{x \in \cQ^n} A(x) \ , 
  \qquad \cB_{j} = \cA \ ,
  \quad  \tau_{j} := \tau \ .
\end{align*}  
Furthermore, set $\sigma = \one$ and $\omega_j = 0$.
We define
$V_{j} = 1$, $\ell_{j} = I$, and $\err_j A(x) = A(s_j x)$, so that $\partial_j A(x) = A(s_j x) - A(x)$. This defines a differential structure with $\sigma = \one$. It follows that $\err_j^\dagger = \err_j$ and 
\begin{align*}
 \partial_j A(x) =  \partial_j^\dagger A(x) = A(s_j x) - A(x)\ . 
\end{align*}
It follows that
\begin{align*}
 \cL A(x) = 2\sum_{j=1}^n (A(s_j x) - A(x))\ ,
\end{align*}
which is the discrete Laplacian on $\cQ^n$ that generates the simple random walk.

\subsection{The Fermionic Ornstein-Uhlenbeck equation}
\label{sec:fermion}

A non-commutative example in which it is advantageous to work with $\ell_j$ not equal to the identity, is the Fermionic Ornstein-Uhlenbeck operator, for which a non-commutative transport metric was constructed in \cite{CM12}.
Let $(Q_1, \ldots, Q_n)$ be self-adjoint operators on a finite-dimensional Hilbert space satisfying the \emph{canonical anti-commutation relations} (CAR):
\begin{align*}
 Q_i Q_j + Q_j Q_i = 2 \delta_{ij}\ .
\end{align*}
The \emph{Clifford algebra} $\Cln$ is the $2^n$-dimensional algebra generated by $\{Q_j\}_{j=1}^n$. Let $\Gamma : \Cln \to \Cln$ be the principle automorphism on $\Cln$, i.e., the unique algebra homomorphism  satisfying $\Gamma(Q_j) = - Q_j$ for all $j$. Let $\tau$ be the canonical trace on $\Cln$, determined by $\tau(Q_1^{\alpha_1} \cdots Q_n^{\alpha_n}) := \delta_{0,|\aa|}$ for all $\aa = (\alpha_j)_j \in \{0,1\}^n$, where $|\aa| := \sum_j \alpha_j$.
We then set $\cJ = \{1, \ldots, n\}$, $\cA := \cB_{j} := \Cln$, and $\tau_{j} := \tau$.
Furthermore we set $V_{j} = Q_j$, $\ell_{j} = \Gamma$, and $\err_{j} = I$. 
Then $\ell_j^\dagger = \Gamma$, and the operators $\partial_j$ and $\partial_j^\dagger$ are skew-derivations given by
\begin{align*}
 \partial_j A =  Q_j A - \Gamma(A) Q_j\ , \qquad
 \partial_j^\dagger A =  Q_j A + \Gamma(A) Q_j\ .
\end{align*}
Taking $\sigma = \one$ and $\omega_j = 0$ we obtain
\begin{align*}
 \cL A = 2\sum_{j=1}^n (Q_j A Q_j - A)\ ,
\end{align*}
which implies that $\cL = - 4\mathcal{N}$, where $\mathcal{N}$ is the \emph{fermionic number operator} (see \cite{CL1,CM12} for more details).

\subsection{The depolarizing channel}
\label{sec:depolarizing-generator}

This is one of the simplest non-commutative examples. Given a non-commutative probability space $(\cA, \tau)$ and $\gamma > 0$, the generator is defined by 
\begin{align}\label{eq:depolarizing-generator}
 \cL A = \gamma\big(\tau[A] \one - A\big) \ .
\end{align}

In the case where $\cA = \cB_j = \M_2(\C)$ and $\tau$ is the usual trace, this operator can be written in Lindblad form using the Pauli matrices
\begin{align*}
  \sigma_x = \sigmax\ ,\qquad
  \sigma_y = \sigmay\ ,\qquad
  \sigma_z = \sigmaz\ .
\end{align*}
We set $V_j = \sqrt{\gamma} \sigma_j$ and $\ell_j = r_j = I_\cA$, so that the differential operators $\partial_x, \partial_y$ and $\partial_z$ are the commutators
\begin{align*}
	\partial_j A = \sqrt{\gamma} [\sigma_j, A]
\end{align*}
for $j \in \{ x, y, z\}$.
This yields a differentiable structure with $\sigma = \one$ and $\omega_j = 0$, and a direct computation shows that $\cL$ is indeed given by \eqref{eq:depolarizing-generator}.

\section{Non-commutative functional calculus}

Let $\cA$ be a finite-dimensional $C^*$-algebra. Let $A, B \in \cA$ be self-adjoint with spectral decompositions
\begin{align}\label{eq:spec-decomp}
 A = \sum_i \lambda_i  A_i\quad \text{ and } \quad
 B = \sum_k \mu_k  B_k
\end{align}
 for some eigenvalues $\lambda_i, \mu_k \in \R$ and spectral projections $ A_i,  B_k \in \cA$ satisfying $A_i A_k = \delta_{ik} A_i$, $B_i B_k = \delta_{ik} B_i$, and
$\sum_i  A_i = \sum_k  B_k = \one_{\cA}$.
For a function $\theta : \spec(A) \times \spec(B) \to \R$ we define $\theta(A,B) \in \cA \times \cA$ to be the \emph{double operator sum}
 \begin{align}\label{eq:funct-calc}
    \theta(A,B) = \sum_{i,k} \theta(\lambda_i,\mu_k)
     A_i \ot  B_k\ .
  \end{align}
  
\begin{remark}\label{rem:doi}
A systematic theory of infinite-dimensional generalizations of $\theta(A,B)$ has been developed under the name of \emph{double operator integrals}, see, e.g., \cite{BirSol03,PotSuk11}.
\end{remark}
Double operator sums are compatible with the usual functional calculus, in the sense that
\begin{align}\label{eq:fc}
 \theta(f(A), g(B)) = (\theta \circ (f,g))(A,B)
\end{align} 
for all $f : \spec(A) \to \R$, $g : \spec(B) \to \R$ and $\theta : \R \times \R \to \R$. Moreover, recalling that the contraction operator has been defined in \eqref{eq:contr}, we have
\begin{align}\label{eq:fc2}
   \theta_2(A,B) \# \big( \theta_1(A,B) \# C\big) =  (\theta_2 \cdot \theta_1)(A,B) \# C
\end{align}
The straightforward proof of these identities is left to the reader.
  
Let $\cI \subseteq \R$ be an interval. Of particular relevance for our purposes is the special case where $\theta = \delta f : \cI \times \cI \to \R$ is the \emph{discrete derivative} of a differentiable function $f : \cI \to \R$, defined by 
  \begin{align}\label{eq:qd}
   \delta f(\lambda,\mu) := 
\left\{ \begin{array}{ll}
\displaystyle   \frac{f(\lambda) - f(\mu)}{\lambda-\mu}\ , \quad & \lambda \neq \mu\ ,\\
  f'(\lambda)\ ,        \quad &\lambda = \mu\ .
\end{array} \right.
\end{align}
Using the contraction operator we can write the following useful chain rule:
\begin{align}\label{eq:chain-rule-one}
 f(A) - f(B) = \delta f(A,B) \# (A - B)\ .
\end{align}
We can also formulate a chain rule for the operator $\partial_V$ defined in \eqref{pardef}, which plays a crucial role in the sequel.

 \begin{proposition}[Chain rule for $\partial_V$]\label{prop:chain-rule}
Let $A \in \cA_h$. For any function $f : \spec(A) \to \R$ we have
 \begin{align} \label{eq:chain-rule-two}
  \partial_V f(A)  = \delta f(\ell(A),\err(A)) \# \partial_V A\ .
 \end{align}
 \end{proposition}
 
\begin{proof}
Let $A = \sum_i \lambda_i  A_i$ be the spectral decomposition with eigenvalues $\lambda_i \in \R$ and spectral projections $ A_i \in \cA$ satisfying $A_i A_k = \delta_{ik} A_i$ and $\sum_i  A_i = \one_\cA$.
Since $\ell(\one_\cA) = \err(\one_\cA) = \one_{\cA}$ by assumption, it follows that $\sum_i \ell(A_i) = \sum_i \err(A_i) = \one_{\cB}$ for all $j$. Therefore, 
\begin{align*}
 \partial_V A
  & = \sum_i \lambda_i 
  		\Big(V  \err(A_i) - \ell(A_i) V\Big)
\\& = \sum_{i,k} (\lambda_k - \lambda_i)
		\ell(A_i) V \err(A_k)\ .	
\end{align*}
Consequently, since $ 
\ell(A_p) \ell(A_i) 
= \ell(A_p A_i)
= \delta_{pi}\ell(A_i)$ and $\err(A_k) \err(A_m) =  \delta_{km}\err(A_k)$,
\begin{equation}\begin{aligned}\label{eq:chain-comp}
\delta f(\ell(A),\err(A)) \# \partial A
&  = \sum_{i,k,p,m} \delta f(\lambda_p,\lambda_m) 
						(\lambda_k - \lambda_i)
  				 \ell(A_p A_i) V \err(A_k A_m)
\\&= \sum_{i,k} \delta f(\lambda_i,\lambda_k) 
						(\lambda_k - \lambda_i)
  				 \ell(A_i) V \err(A_k)
\\&= \sum_{i,k} \big(f(\lambda_k) - f(\lambda_i) \big)
  				 \ell(A_i) V \err(A_k)
\\&= \partial_V f(A) \ .
\end{aligned}\end{equation}
\end{proof}

\begin{remark}\label{rem:non-differentiable}
Note that the function $f$ is not required to be differentiable in Proposition \ref{prop:chain-rule}. In this case, $\delta f$ is not defined on the diagonal, but the second line in \eqref{eq:chain-comp} shows that its diagonal value is irrelevant.
\end{remark}

The following well-known chain rule can also be formulated in terms of $\delta f$. 

\begin{proposition}\label{prop:chain-t}
Let $A: \cI \to \cA_{h}$ be differentiable on an interval $\cI \subseteq \R$ and let $f$ be a real-valued function on an interval containing $\spec(A(t))$ for all $t \in \cI$. Then: 
\begin{align}
 \label{eq:chain-t}
 \ddt f(A(t)) &= \delta f\big(A(t), A(t)\big) \# A'(t)\ ,\\
 \label{eq:chain-trace}
  \ddt \tau\big[f(A(t))\big] &= \tau\big[f'(A(t))A'(t)\big]\ .
\end{align}
\end{proposition}

\begin{proof}
The first assertion follows by passing to the limit in \eqref{eq:chain-rule-one}. The second identity follows easily using the definition of $\delta f$ and the cyclicity of the trace.
\end{proof}

\begin{example}\label{ex:entropy}
We illustrate the proposition above with a well-known computation that will be useful below.
For $\rho, \sigma \in \Dens_+(\cA)$ and $\nu \in \cA_h$ with $\tau[\nu] = 0$, set $\rho_t := \rho + t \nu$. It follows from \eqref{eq:chain-trace} that 
\begin{align}\label{eq:Ent-first}
\partial_t \Ent_\sigma(\rho_t) = \tau[\nu(\log \rho_t - \log \sigma)] \ .
\end{align}
Since $\delta \log(r,s) = \frac{\log r - \log s}{r - s} = \int_0^\infty (x + r)^{-1} (x + s)^{-1} \dd x$, we have $\delta \log(R,S) = \int_0^\infty (x + R)^{-1} \otimes (x + S)^{-1} \dd x$. Thus, \eqref{eq:chain-t} yields
\begin{align}\label{eq:Ent-expansion}
  \partial_t^2 \Ent_\sigma(\rho_t) =  \int_0^\infty \tau \Big[\nu \frac{1}{x + \rho_t} \nu \frac{1}{x + \rho_t} \Big] \dd x \ .
\end{align}
\end{example}

We finish this subsection with some useful properties of the sesquilinear form $(A,B)\mapsto  \ip{ A, \phi(R,S) \# B }_{L^2(\tau)}$ on $\cA$.

\begin{lemma}\label{lem:Schur-bound}
Let $R, S \in \cA$ be self-adjoint and let
 $\phi : \spec(R) \times \spec(S) \to \R_+$ be given. Then, for all $A \in \cA$,
\begin{align*}
 \ip{ A, \phi(R,S) \# A }_{L^2(\tau)} 
 \geq 0\ .
\end{align*}
\end{lemma}

\begin{proof}
Using the spectral decompositions $R = \sum_i \lambda_i R_i$ and $S = \sum_k \mu_k S_k$ we may write
\begin{align*}
 \ip{ A, \phi(R,S) \# A }_{L^2(\tau)}  
& = \sum_{i,k} \phi(\lambda_i,\mu_k)
				 \tau [A^* R_i A S_k] \ .
\end{align*}
Since $\tau [A^* R_i A S_k] = \tau[(R_i A S_k)^*(R_i A S_k)] \geq 0$ the result follows.
\end{proof}

\begin{proposition}\label{prop:ip}
Let $R, S \in \cA$ be self-adjoint and suppose that
 $\phi : \spec(R) \times \spec(S) \to \R$ is strictly positive.
Then the sequilinear form 
\begin{align*}
	(A,B)\mapsto  \ip{ A, \phi(R,S) \# B }_{L^2(\tau)}
\end{align*}
defines a scalar product on $\cA$.
\end{proposition}

\begin{proof}
Consider the spectral decompositions $R = \sum_i \lambda_i  R_i$ and $S = \sum_k \mu_k S_k$.
Using basic properties of the trace, we obtain 
\begin{align*}
   		\overline{ \tau[A^* R_i B S_k ]} 
		  =  \tau[(A^* R_i B S_k)^*] 
		  =  \tau[ S_k  B^* R_i A] 
		  =  \tau[ B^* R_i A S_k] \ ,
\end{align*}
and therefore, since $\phi$ is real-valued,
\begin{align*}
\overline{ \ip{ A, \phi(R,S) \# B }_{L^2(\tau)}}
&  = \sum_{i,k} \phi(\lambda_i, \mu_k) 
  		\overline{ \tau[A^* R_i B S_k ]}
\\&  = \sum_{i,k} \phi(\lambda_i, \mu_k) 
  		\tau[ B^* R_i A S_k] 
 	 =  \ip{ B, \phi(R,S) \# A }_{L^2(\tau)}\ .
\end{align*}
Moreover, since $\phi$ is strictly positive on the finite set $\spec(R) \times \spec(S)$, we have $\phi \geq \eps$ for some $\eps > 0$. Thus Lemma \ref{lem:Schur-bound}  implies that $ \ip{ A, \phi(R,S) \# A } \geq \eps \| A \|_{L^2(\tau)}^2$. It follows that $ \ip{ A, \phi(R,S) \# A } \geq 0$, with equality if and only if $A = 0$.  
\end{proof}

\subsection*{Higher order expressions}

In the sequel we will use versions of Propositions \ref{prop:chain-rule} and \ref{prop:chain-t} for higher order derivatives, for which we need to introduce more notation.  
For $x = (x_1, \ldots, x_n) \in \R^n$ and $1 \leq i \leq m \leq n$ we will use the shorthand notation $x_i^m = (x_i, x_{i+1}, \ldots, x_{m-1}, x_m)$.
For a function $\phi : \R^n \to \R$ and $j = 1, \ldots, n$ we consider the discrete derivative $\delta_j \phi : \R^{n+1} \to \R$ defined by
\begin{align}
	\label{eq:qpd}
 \delta_j \phi(x_1^{j-1}, (x_j, \tilde x_j), x_{j+1}^n)
  := \delta \phi(x_1^{j-1}, \cdot, x_{j+1}^n) (x_j, \tilde x_j)\ ,
\end{align} 
where $\delta$ denotes the discrete derivative given by \eqref{eq:qd}.
Iterating this procedure, one arrives at expressions that can be naturally encoded using rooted planar binary trees. Indeed, for a given function $\theta : \R \times \R \to \R$ and $x, y \in \R$, we write 
\begin{align*}
\theta(x,y) & = \theta\big(\treelr{\bullet}{x}{y}\big) \ .
\end{align*}
The left and right child in this tree correspond to the variables $x$ and $y$ in $\theta(x,y)$ respectively. More complicated trees are then constructed by iteratively replacing one of the children $\bullet$ by $\2$. This will correspond to discrete differentiation with respect to the respective variables, e.g., 
\begin{align}
\label{eq:tree-1}
\delta \theta((x,y),z)
 & = \delta\theta\Big(\treellr{\bullet}{z}{\bullet}{x}{y}\Big) 
 := 
 {\frac{\theta(x, z)- \theta(y, z)}{x - y}}  \ , \\
 \label{eq:tree-2} 
\delta \theta(x,(y,z)) 
& 
 = \delta\theta\Big(\treelrr{\bullet}{x}{\bullet}{y}{z}\Big) 
 := 
 {\frac{\theta(x, y)- \theta(x, z)}{y - z}} \ , \\
\delta \theta \big(\big((x,y),z\big),w\big)
& =       \delta\theta\Big(\treelllr{\bullet}{\bullet}{\bullet}{x}{y}{z}{w}\Big) 
  := 
 \frac{{\frac{\theta(x, w)- \theta(z, w)}{x - z}} 
        - {\frac{\theta(y, w)- \theta(z, w)}{y - z}} } 
        {x - y}
\ .
\label{eq:tree-3}
\end{align}
The middle expressions are valid whenever the variables are distinct. 
If some of the variables are equal, finite differences are to be interpreted as derivatives. For instance, if $x = y \neq z$ in \eqref{eq:tree-3}, we have
\begin{align*}
\delta\theta\Big( \treelllr{\bullet}{\bullet}{\bullet}{x}{x}{z}{w} \Big)
&  = \partial_x  \Big(\delta\theta \Big( \treellr{\bullet}{w}{\bullet}{x}{z} \Big)\Big)
\\& = \frac{(x-z) \partial_1 \theta(x,w) - (\theta(x,w) - \theta(z,w))}{(x-z)^2}\ .
\end{align*}
If $x = y = z$ in \eqref{eq:tree-3}, then the formula above becomes 
\begin{align*}
 \delta\theta\Big(\treelllr{\bullet}{\bullet}{\bullet}{x}{x}{x}{w} \Big) = \frac12 \partial_1^2\theta(x,w)\ .
\end{align*}

The functional calculus \eqref{eq:funct-calc} generalizes naturally to functions of several variables.
Let $A^{(1)}, \ldots, A^{(n)}$ be self-adjoint elements in $\cA$ with spectral decompositions
\begin{align*}
 A^{(k)} = \sum_i \lambda_i^{(k)}  A_i^{(k)}
\end{align*}
for some eigenvalues $\lambda_i^{(k)} \in \R$ and spectral projections $A_i^{(k)} \in \cA$ with 
$\sum_i  A_i^{(k)} =  \one_{\cA}$.
For a function $\theta : \spec(A^{(1)}) \times \cdots \times \spec(A^{(n)}) \to \R$ we define $\theta(A_1,  \ldots ,A_n) \in \cA^{\ot n}$ to be the \emph{multiple operator sum}
 \begin{align}\label{eq:mult-funct-calc}
    \theta(A_1, \ldots, A_n) = \sum_{i_1, \ldots, i_n} \theta\big(\lambda_{i_1}^{(1)},\ldots, \lambda_{i_n}^{(n)}\big)
     A_{i_1}^{(1)} \ot \cdots \ot  A_{i_n}^{(n)}\ .
\end{align}
In the sequel we shall apply this definition to $\delta \theta$ in order to define expressions such as $\delta \theta(A,B)$. The tree notation is useful when considering generalizations of the contraction operation \eqref{eq:contr} to higher order tensor products. Each of the nodes that is a parent can be used to indicate the position at which an operator for contraction is inserted: e.g., we write
\begin{align*}
  \theta\big(\widetreelr{B}{A^{(1)}}{A^{(2)}}\big) &= \theta(A^{(1)},A^{(2)}) \# B = 
   \sum_{i, j} 
     \theta\big(\lambda_{i}^{(1)}, \lambda_{j}^{(2)}\big)
  	     A_{i}^{(1)} B A_{j}^{(2)} \ , \\
 \delta\theta\Big( \widetreellr{\bullet}{A^{(3)}}{B}{A^{(1)}}{A^{(2)}} \Big)
& =  \sum_{i, j, k} 
    {\frac{\theta(\lambda_{i}^{(1)}, \lambda_{k}^{(3)})- \theta(\lambda_{j}^{(2)}, \lambda_{k}^{(3)})}{\lambda_{i}^{(1)} - \lambda_{j}^{(2)}}}  
  	     A_{i}^{(1)} B A_{j}^{(2)} \ot  A_{k}^{(3)} \ , \\
\delta\theta\Big( \widetreelrr{\bullet}{A^{(1)}}{B}{A^{(2)}}{A^{(3)}}  \Big)
& =  \sum_{i, j, k} 
    {\frac{\theta(\lambda_{i}^{(1)}, \lambda_{j}^{(2)})- \theta(\lambda_{i}^{(1)}, \lambda_{k}^{(3)})}{\lambda_{j}^{(2)} - \lambda_{k}^{(3)}}}  
  	     A_{i}^{(1)} \ot A_{j}^{(2)} B  A_{k}^{(3)} \ , \\
\end{align*}
where the fractions at the right-hand side are to be understood in the sense of limits if the denominator vanishes.
These expressions appear naturally in the following chain rule that will be useful in Section \ref{sec:Riemann}.

\begin{proposition}\label{prop:chain-double-simple}
Let $A, B : \cI \to \cA_{h}$ be differentiable on an interval $\cI \subseteq \R$, and let $\theta: \R \times \R \to \R$ be differentiable. Then: 
\begin{align*}
 \partial_t \theta(A_t, B_t) 
   =  \delta\theta\Big(\widetreellr{\bullet}{B_t}{\partial_t A_t}{A_t}{A_t} \Big)
     + \delta\theta\Big(\widetreelrr{\bullet}{A_t}{\partial_t B_t}{B_t}{B_t}\Big) \ .
\end{align*}
\end{proposition}

\begin{proof}
We have $\partial_t \theta(A_t, B_t) = \partial_s|_{s = t} \theta(A_s, B_t) + \partial_s|_{s = t} \theta(A_t, B_s)$. Since we can write $\theta(A_t,B_s) = \sum_{k} \theta(A_t,\mu_{s,k})\otimes F_{s,k}$, where $B_s = \sum_{k} \mu_{s,k} F_{s,k}$ denotes the spectral decomposition of $B_s$, the result follows by applying \eqref{eq:chain-t} from Proposition \ref{prop:chain-t} twice.
\end{proof}

Higher order derivatives can also be naturally expressed in terms of trees, but since this will not be needed in the sequel, we will not go into details here.

\section{Riemannian structures on the space of density matrices}

\label{sec:Riemann}

In this section we shall analyze a large class of Riemannian metrics on the space of density matrices. Throughout the section we fix a differentiable structure $(\cA, \nabla, \sigma)$ in the sense of Definition \ref{ass:setup}. The generator of the associated quantum Markov semigroup $(\cP_t)_t$ will be denoted by $\cL$.

\subsection{Riemannian structures on density matrices}
\label{sec:Riemann-dens}

Consider the $\R$-linear subspace
\begin{align*}
 \cA_0 := \Ran(\cL^\dagger) \cap \cA_h \ .
\end{align*} 
We shall study Riemannian structures on relatively open subsets of $\Dens_+$, the set of all strictly positive elements in $\Dens$. These subsets are of the form 
\begin{align*}
 \cM_\rho := \Big( \rho +  \cA_0 \Big) \cap \Dens_+ \ ,
\end{align*}
where $\rho \in \Dens_+$. At each point of $\cM_\rho$, the tangent space of $\cM_\rho$ is thus naturally given by $\cA_0$.

\begin{remark}\label{rem:ergodic}
Of special interest is the \emph{ergodic} case, i.e., the case where $\Ker(\cL) = \lin\{\one\}$. In this case we have $\cA_0 = \{ A \in \cA_h : \tau[A] = 0 \}$, and therefore $\cM_\rho = \Dens_+$ for all $\rho \in \Dens_+$.
\end{remark}
In order to define a Riemannian structure, we shall fix for each $j \in \cJ$ a function $\theta_j : [0,\infty) \times [0,\infty) \to \R$ satisfying the following properties:

\begin{assumption}\label{ass:theta}
For $j \in \cJ$ the functions $\theta_j : [0,\infty) \times [0,\infty) \to \R$ are continuous. Moreover, on $(0,\infty) \times (0,\infty)$, the function $\theta_j$ is $C^\infty$ and strictly positive, and we have the symmetry condition
\begin{align}\label{eq:theta-sym}
	\theta_j(r,s) = \theta_{j^*}(s,r) \ .
\end{align}
\end{assumption}
Recalling the definition of the double operator sum in \eqref{eq:funct-calc}, we will use the shorthand notation
\begin{align}\label{eq:iota}
 \hrho_j &= \theta_j(\ell_j(\rho), \err_j(\rho)) \in \cB_j \ot \cB_j\ , \qquad \hrho = (\hrho_j )_{j \in \cJ}\ \quad \text{for } \rho \in \Dens\ ,\\ 
 \lrho_j &= \frac{1}{\theta_j}(\ell_j(\rho), \err_j(\rho)) \in \cB_j \ot \cB_j\,,\qquad
 \lrho = (\lrho_j)_{j \in \cJ} \ \quad \text{for } \rho \in \Dens_+ \ .
\end{align}Let us now define the class of quantum transport metrics that we are interested in.
For $\rho \in \Dens$, we define the operator $\cK_\rho : \cA \to \cA$ by
\begin{align}\label{eq:def-K}
  \cK_\rho A := - \dive(\widehat\rho \# \nabla A) = \sum_{j \in \cJ} \partial_j^\dagger(\widehat\rho_j \# \partial_j A) \ ,
\end{align}
where we use the vector notation
$\widehat\rho \# \nabla A = (  \widehat\rho_j \# \partial_j A  )_{j \in \cJ}$ and we recall that the divergence operator has been defined in \eqref{eq:def-div}.
To define the Riemannian metric we need a lemma concerning the unique solvability of the continuity equation in the class of ``gradient vector fields''. Therefore we need to identify the kernel and the range of the linear operator $\cK_\rho$.

\begin{lemma}[Mapping properties of $\cK_\rho$]\label{lem:Ran-Ker}
For $\rho \in \Dens_+$ the operator $\cK_\rho$ is non-negative and self-adjoint on $L^2(\cA,\tau)$. Moreover,  we have
\begin{align}\label{eq:Ran-Ker}
 \Ran(\cK_\rho) = \Ran(\cL^\dagger) = \Ran(\dive)\ , \qquad
 \Ker(\cK_\rho) =  \Ker(\cL) = \Ker(\nabla)\ .
\end{align}
Furthermore, $\cK_\rho$ is real, i.e., for $A \in \cA$ we have $(\cK_\rho A)^* = \cK_\rho A^*$.
\end{lemma}

\begin{proof}
For $A, B \in \cA$, Lemma \ref{prop:ip} yields
\begin{align*}
  \ip{\cK_\rho A, B}_{L^2(\tau)}
    & = \sum_{j \in \cJ} \ip{ \widehat\rho_j \# \partial_j A,\partial_j B }_{L^2(\tau_j)}
      = \sum_{j \in \cJ} \ip{\partial_j A, \widehat\rho_j \# \partial_j B }_{L^2(\tau_j)}  
      = \ip{A, \cK_\rho B}_{L^2(\tau)} \ ,
\end{align*}
hence $\cK_\rho$ is self-adjoint on $L^2(\cA, \tau)$.

The identities $\Ker(\cL) = \Ker(\nabla)$ and $\Ran(\cL^\dagger) = \Ran(\dive)$ have already been proved in Proposition \ref{prop:domains-Dirichlet}.
Clearly, $\Ker(\nabla) \subseteq \Ker(\cK_\rho)$. To prove the opposite inclusion, we note that since $\rho \in \Dens_+$, there exists $c > 0$ with $\theta_j|_{\spec(\rho)} \geq c > 0$ for all $j \in \cJ$. Lemma \ref{lem:Schur-bound} implies that 
\begin{align*}
\ip{ A, \cK_\rho A}_{L^2(\cA,\tau)} 
  = \sum_{j \in \cJ} \ip{\partial_j A, \hrho_j \# \partial_j A}_{L^2(\cB_j,\tau_j)}
  \geq c \sum_{j \in \cJ} \| \partial_j A \|_{L^2(\cB_j,\tau_j)}^2 \ ,
\end{align*}
from which we infer that $\Ker(\cK_\rho)\subseteq \Ker(\nabla)$. This proves the second identity in \eqref{eq:Ran-Ker}, and the nonnegativity of $\cK_\rho$ follows as well. The first identity in \eqref{eq:Ran-Ker} follows using elementary linear algebra, since the self-adjointness of $\cK_\rho$ in $L^2(\cA, \tau)$ yields
\begin{align*}
 \Ran(\cK_\rho) = ( \Ker(\cK_\rho) )^\perp
    =   (\Ker(\nabla))^\perp
    = \Ran(\dive)\ .
\end{align*}
To prove that $\cK_\rho$ preserves self-adjointness, we consider the spectral decomposition $\rho = \sum_k \lambda_k E_k$, and write $\theta_j^{km} := \theta_j(\lambda_k, \lambda_m)$ for brevity. 
We have
\begin{equation}\begin{aligned}\label{eq:Krho-adjoint}
	  &   \ip{B,\cK_\rho A^*}_{L^2(\tau)}
	\\& = \sum_{j \in \cJ} \tau_j[(\partial_j B)^* \hrho_j \# \partial_j A^*]
	\\& = \sum_{j \in \cJ}\sum_{k,m}  \theta_j^{km} \tau_j\Big[\Big(  V_j \err_j(B) - \ell_j(B) V_j \Big)^* \ell_j(E_k) \Big(V_j \err_j(A^*) - \ell_j(A^*) V_j\Big)  \err_j(E_m)\Big]
	\\& = \sum_{j \in \cJ}\sum_{k,m} \theta_j^{km}\tau_j \Big[
        \err_j(B^*)  V_j^*  \ell_j(E_k) V_j \err_j(A^*) \err_j(E_m)
   - 	   V_j^*\ell_j(B^*) \ell_j(E_k) V_j \err_j(A^*) \err_j(E_m)
	\\&	   \qquad   \qquad  \qquad\quad 
	     - \err_j(B^*) V_j^* \ell_j(E_k) \ell_j(A^*) V_j   \err_j(E_m)
         + V_j^* \ell_j(B^*) \ell_j(E_k) \ell_j(A^*) V_j  \err_j(E_m)
	 \Big]
	\\& = \sum_{j \in \cJ}\sum_{k,m} \theta_j^{km}\tau_j \Big[
           V_j^* \ell_j(E_k) V_j \err_j(A^* E_m B^*) 
   		-  V_j^* \ell_j(B^* E_k) V_j \err_j(A^* E_m)
	\\&	   \qquad   \qquad  \qquad\quad  
		-  V_j^* \ell_j(E_k A^*) V_j \err_j(E_m B^*) 
		+  V_j^* \ell_j(B^* E_k A^*) V_j \err_j(E_m)
	 \Big] \ .
\end{aligned}\end{equation}
On the other hand, 
\begin{align*}
	  &   \ip{B,(\cK_\rho A)^*}_{L^2(\tau)}
	\\& = \sum_{j \in \cJ} \tau_j[(\partial_j B^*) (\hrho_j \# \partial_j A)^*]
	\\& = \sum_{j \in \cJ}\sum_{k,m}  \theta_j^{km} \tau_j\Big[\Big(  V_j \err_j(B^*) - \ell_j(B^*) V_j \Big) \err_j(E_m) \Big( V_j \err_j(A) - \ell_j(A) V_j  \Big)^*\ell_j(E_k)\Big]
	\\& = \sum_{j \in \cJ}\sum_{k,m}  \theta_j^{km} \tau_j\Big[ 
				 V_j \err_j(B^*)  \err_j(E_m) \err_j(A^*)V_j^*\ell_j(E_k)
  			   - V_j \err_j(B^*) \err_j(E_m)  V_j^* \ell_j(A^*) \ell_j(E_k)
	\\&	   \qquad   \qquad  \qquad\quad     
	  - \ell_j(B^*) V_j \err_j(E_m)  \err_j(A^*) V_j^*\ell_j(E_k)
      +  \ell_j(B^*) V_j \err_j(E_m) V_j^*\ell_j(A^*)   \ell_j(E_k)\Big]
	\\& = \sum_{j \in \cJ}\sum_{k,m}  \theta_j^{km} \tau_j\Big[ 
				  V_j^* \ell_j(E_k) V_j \err_j(B^* E_m A^*)
  			   -  V_j^* \ell_j(A^* E_k)V_j \err_j(B^* E_m) 
	\\&	   \qquad   \qquad  \qquad\quad     
	           -  V_j^* \ell_j(E_k B^*) V_j \err_j(E_m A^*)
      		   +  V_j^* \ell_j(A^* E_k B^*) V_j \err_j(E_m) \Big] \ .
\end{align*}
Thus, using \eqref{eq:symmetry-equiv}, and then changing $j$ by $j^*$ and using that $\theta_j^{km} = \theta_{j^*}^{mk}$ by Assumption \ref{ass:theta}, we obtain 
\begin{align*}
	  &\ip{B,(\cK_\rho A)^*}_{L^2(\tau)}
	 \\&  = \sum_{j \in \cJ}\sum_{k,m} \theta_j^{km} \tau_j\Big[ 
				  V_j \ell_{j^*}(B^* E_m A^*) V_j^* \err_{j^*}(E_k)
  			   -  V_j \ell_{j^*}(B^* E_m) V_j^* \err_{j^*}(A^* E_k) 
	\\&	   \qquad   \qquad  \qquad\quad     
	           -  V_j \ell_{j^*}(E_m A^*) V_j^* \err_{j^*}(E_k B^*)
      		   +  V_j \ell_{j^*}(E_m) V_j^* \err_{j^*}(A^* E_k B^*) \Big]
	 \\&  = \sum_{j \in \cJ}\sum_{k,m} \theta_j^{mk} \tau_j \Big[ 
				  V_j^* \ell_{j}(B^* E_m A^*) V_j \err_{j}(E_k)
  			   -  V_j^* \ell_{j}(B^* E_m) V_j \err_{j}(A^* E_k) 
	\\&	   \qquad   \qquad  \qquad\quad     
	           -  V_j^* \ell_{j}(E_m A^*) V_j \err_{j}(E_k B^*)
      		   +  V_j^* \ell_{j}(E_m) V_j \err_{j}(A^* E_k B^*) \Big]
\end{align*}
which coincides with \eqref{eq:Krho-adjoint} after interchanging $m$ and $k$.
\end{proof}

The following result expressing the unique solvability of the continuity equation is now an immediate consequence.
\begin{corollary}\label{cor:Krho-bij}
For $\rho \in \Dens_+$, the linear mapping $\cK_\rho$ is a bijection on $\cA_0$ that depends smoothly ($C^\infty$) on $\rho$.
\end{corollary}

\begin{proof}
It follows from Lemma \ref{lem:Ran-Ker} that $\cK_\rho$ maps $\cA_0$ into itself. Since the restriction of a self-adjoint operator to its range is injective, the result follows. Smooth dependence on $\rho$ follows from the smoothness of $\theta$.
\end{proof}

The following elementary variational characterization is of interest.

\begin{proposition}\label{prop:variational}
Fix $\rho \in \Dens_+$ and $\nu \in \cA_0$. Among all vector fields $\bbB \in \cB$ satisfying the continuity equation
\begin{align}\label{eq:cont-eqVF}
 \nu + \dive(\hrho \# \bbB) = 0
\end{align}
there is a unique one that is a gradient. Moreover, among all vector fields $\bbB$ solving \eqref{eq:cont-eqVF}, this vector field is the unique minimizer of the ``kinetic energy functional'' $\sE_\rho$ given by
\begin{align*}
  \sE_\rho(\bbB) = \sum_{j \in \cJ}  \ip{\hrho_j \# B_j, B_j}_{L^2(\tau_j)} \ .
\end{align*}
\end{proposition}

\begin{proof}
Existence of a gradient vector $\bbB$ field solving \eqref{eq:cont-eqVF} follows from Corollary \ref{cor:Krho-bij}. To prove uniqueness, suppose that $\dive(\hrho \# \nabla A) = - \nu = \dive(\hrho \# \nabla \tilde A)$ for some $A, \tilde A \in \cA$. This means that $\cK_\rho A = \cK_\rho \tilde A$, hence Lemma \ref{lem:Ran-Ker} yields $\nabla  A = \nabla \tilde A$. The remaining part follows along the lines of the proof of \cite[Theorem 3.17]{CM12}.
\end{proof}

We are now ready to define a class of Riemannian metrics that are the main object of study in this paper.

\begin{definition}[Quantum transport metric]\label{def:quantum-transport}
Fix $\rho \in \Dens_+$ and let $\theta = (\theta_j)_j$ satisfy Assumption \ref{ass:theta}. 
The associated \emph{quantum transport metric} is the Riemannian metric on $\cM_\rho$ induced by the operator $\cK_\rho$, i.e., for $\dot\rho_1, \dot\rho_2 \in \cA_0$,
\begin{align*}
 \ip{\dot\rho_1, \dot\rho_2}_\rho = \ip{\cK_\rho^{-1} \dot\rho_1, \dot\rho_2}_{L^2(\tau)} \ ,
\end{align*}
or, more explicitly,
\begin{equation}\begin{aligned}
	\label{eq:scalar-product}
	\ip{\dot\rho_1, \dot \rho_2}
	&  =  \ip{ \nabla A_1, \hrho \# \nabla A_2 }_{L^2(\tau)}
   \\& = \sum_j \bip{ \partial_j A_1, \theta_j\big(\ell_j(\rho), \err_j(\rho)\big) \# \partial_j A_2 }_{L^2(\tau_j)}  \quad \text{ for $\rho \in \Dens_+$}\ ,
\end{aligned}\end{equation}
where, for $i = 1,2$, $A_i$ is the unique solution in $\cA_0$ to the continuity equation
\begin{align*}
	 \dot\rho_i + \dive(\hrho \# \nabla A_i) = 0 \ .
\end{align*}
\end{definition}

It follows from Lemma \ref{lem:Ran-Ker} and Corollary \ref{cor:Krho-bij} that $\cK_\rho$ indeed induces a Riemannian metric on $\cM_\rho$.

\subsection{Gradient flows of entropy functionals}\label{sec:grad-flow}

In this section we shall show that various evolution equations of interest can be interpreted as gradient flow equations with respect to suitable quantum transport metrics introduced in Section \ref{sec:Riemann-dens}.

We consider the operator $\cK_\rho : \cA_0 \to \cA_0$ given by
\begin{align*}
   \cK_\rho A := \sum_{j \in \cJ} \partial_j^\dagger(\widehat\rho_j \# \partial_j A)  \ ,
\end{align*}
where $\widehat\rho_j  = \theta_j(\ell_j(\rho), \err_j(\rho))$ is defined in terms of a well-chosen function $\theta_j$ that depends on the context and will be determined below.

\begin{theorem}[Gradient flow structure for the relative entropy]\label{thm:grad-flow-ent}
Consider the operator $\cK_\rho$ defined using the functions $\theta_j$ given by $\theta_j(r,s) := \Lambda(e^{\omega_j/2}r, e^{-\omega_j/2}s)$, where $\Lambda(r,s) = \frac{r-s}{\log r - \log s}$ is the logarithmic mean. 
Then we have the identity
	\begin{align*}
		 \cL^\dagger \rho = - \cK_\rho \rmD \Ent_{\sigma}(\rho) 
	\end{align*}
for all $\rho \in \Dens_+$, thus the gradient flow equation for the relative von Neumann entropy functional $\Ent_\sigma$ with respect to the Riemannian metric on $\cM_\sigma$ induced by $(\cK_\rho)_{\rho \in \cM_\sigma}$ is the Kolmogorov forward equation $\partial_t \rho = \cL^\dagger \rho$. 
\end{theorem}

This result generalises the gradient flow structure from \cite{2016-Carlen-Maas,MiMi17} as described in Section \ref{sec:QMS}. 
The proof relies on the following version of the chain rule.

\begin{lemma}[Chain rule for the logarithm]
\label{lem:chain-log-general}
Define $\theta_j(r,s) := \Lambda(e^{\omega_j/2}r, e^{-\omega_j/2}s)$, where $\Lambda(r,s) = \frac{r-s}{\log r - \log s}$ is the logarithmic mean. 
Then, for all $\rho \in \Dens_+$ we have
	\begin{align}\label{eq:chain-rule-log}
		 e^{-\omega_j/2}V_j \err_j(\rho)  - e^{\omega_j/2}\ell_j(\rho) V_j
		 &  =    \hrhop \# \partial_{j}(\log \rho - \log \sigma) \ .
	\end{align}
\end{lemma}

\begin{proof}
Using \eqref{relmod3} we infer that	
	\begin{align*}
		\partial_{j}( \log \rho - \log \sigma  ) 
		   = V_j \log \big(e^{-\omega_j/2} \err_j(\rho)\big) - \log \big(e^{\omega_j/2} \ell_j(\rho)\big) V_j \ .
	\end{align*}
We consider the spectral decomposition $\rho = \sum_k \lambda_k E_k$ as before, and observe that
	\begin{align*}
		\hrho_{j} = \theta_j\big(\ell_j(\rho), \err_j(\rho)\big)
		 =  \sum_{k,m} \Lambda\big(e^{\omega_j/2} \lambda_k, e^{-\omega_j/2} \lambda_m\big) \ell_j(E_k) \ot \err_j(E_m) \ .
	\end{align*}
Using this identity, we obtain
\begin{align*}
&		\hrhop \# \big( \partial_{j}( \log \rho - \log \sigma  )  \big)
		\\& = \sum_{k, m, p}  \Lambda\big(e^{\omega_j/2} \lambda_k, e^{-\omega_j/2} \lambda_m\big) \ell_j(E_k) 
\\& \qquad \times
		 \Big( \log (e^{-\omega_j/2} \lambda_p)  V_j \err_j(E_p) - \log (e^{\omega_j/2} \lambda_p) \ell_j(E_p) V_j \Big) \err_j(E_m)
		 \\& = \sum_{k, m}  \Lambda(e^{\omega_j/2} \lambda_k, e^{-\omega_j/2} \lambda_m) 
		 \Big( \log (e^{-\omega_j/2} \lambda_m) - \log (e^{\omega_j/2} \lambda_k)  \Big)  \ell_j(E_k) V_j \err_j(E_m)
		 \\& = \sum_{k, m}   
		 \big( e^{-\omega_j/2} \lambda_m - e^{\omega_j/2} \lambda_k  \big)  \ell_j(E_k) V_j \err_j(E_m)
		 \\& =  e^{-\omega_j/2} V_j \err_j(\rho) - e^{\omega_j/2} \ell_j(\rho) V_j\ ,
\end{align*}
which yields \eqref{eq:chain-rule-log}.
\end{proof}

\begin{proof}[Proof of Theorem \ref{thm:grad-flow-ent}]
Since $\rmD \Ent_{\sigma}(\rho) = \log \rho - \log \sigma$, the chain rule from Lemma \ref{lem:chain-log-general} yields, using Proposition \ref{prop:sigma-adjoint},
	\begin{align*}
		\cK_\rho \rmD \Ent_{\sigma}(\rho)
		& =  \sum_{j \in \cJ} 
 	\partial_{j}^\dagger \big( \hrhop \#  \partial_{j} (\log \rho - \log \sigma) \big)  
		\\& =   \sum_{j \in \cJ} 
 	\partial_{j}^\dagger \big(   e^{-\omega_j/2}V_j \err_j(\rho)  - e^{\omega_j/2}\ell_j(\rho) V_j \big)   
		\\& =   \sum_{j \in \cJ} 
		 e^{-\omega_j/2} \Big\{\err_j^\dagger\big(V_j^*  V_j \err_j(\rho) \big) 
		 - \ell_j^\dagger\big(  V_j \err_j(\rho)  V_j^*\big) \Big\}
	\\& \qquad \quad	 - e^{\omega_j/2} \Big\{\err_j^\dagger\big(V_j^* \ell_j(\rho) V_j \big) 
		 - \ell_j^\dagger\big(  \ell_j(\rho) V_j V_j^*\big) \Big\} \ ,
	\end{align*}
which equals the expression for $- \cL^\dagger \rho$ given in Proposition \ref{prop:generator-formula}.
\end{proof}

Let us now consider the special case where $\sigma = \one$. Then \eqref{modeig} reduces to $\omega_j = 0$ for all $j \in \cJ$, and we will be able to formulate a natural nonlinear generalization of Theorem \ref{thm:grad-flow-ent}. Let $f \in C^2((0,\infty); \R)$ be strictly convex, and consider the functional $\cF : \Dens_+ \to \R$ given by
\begin{align*}
 \cF(\rho) = \tau[f(\rho)]\ ,
\end{align*}
where $f(\rho)$ is interpreted in the sense of functional calculus. 
Let $\phi \in C^1((0,\infty);\R)$ be strictly increasing, and consider the operator $\cK_\rho$ as defined before, with $\theta_j = \theta$ given by
\begin{align}\label{eq:theta-general}
\theta(\lambda,\mu) =  \left\{ \begin{array}{ll}
\frac{\phi(\lambda) - \phi(\mu)}{f'(\lambda) - f'(\mu)},
 & \text{if $\lambda \neq \mu$},\medskip\\
 \frac{\phi'(\lambda)}{f''(\lambda)},
 & \text{otherwise}.\end{array} \right.
\end{align}

The following result is a non-commutative analogue  of a seminal result by Otto \cite{O01}, which states that the porous medium equation is the gradient flow equation for the R\'enyi entropy in with respect to the $2$-Kantorovich metric. 

\begin{theorem}[Gradient flow structures with general entropy functionals]\label{thm:grad-flow}
Consider a differentiable structure with $\sigma = \one$, and let $\theta_j$ be given by \eqref{eq:theta-general}.
Then we have the identity
	\begin{align}\label{eq:grad-flow}
	 \cL \phi(\rho) = \cL^\dagger \phi(\rho) = - \cK_\rho \rmD \cF(\rho)
	\end{align}
for $\rho \in \Dens_+$, thus the gradient flow equation for $\cF$ with respect to the Riemannian metric on $\cM_\one$ induced by $(\cK_\rho)_{\rho \in \cM_\one}$ is given by
\begin{align*}
 \partial_t\rho = \cL \phi(\rho) \ .
\end{align*}
 \end{theorem}

\begin{proof}
The first identity in \eqref{eq:grad-flow} follows immediately from the construction of $\cL$ since $\sigma = \one$.
The chain rule \eqref{eq:chain-trace} implies that the derivative of $\cF$ is given by
\begin{align*}
 \rmD \cF(\rho)  = f'(\rho) \ .
\end{align*}
Recalling \eqref{eq:qd}, we note that $\theta_j$ is defined to satisfy the identity $\theta \cdot \delta f' = \delta \phi$.
Using \eqref{eq:fc2}, \eqref{eq:theta-general}, and the chain rule from Proposition \ref{prop:chain-rule} we infer that
\begin{align*}
 \hrho_j \# \partial_j \rmD \cF(\rho) 
  & = \theta\big(\ell_j(\rho), \err_j(\rho)\big) \#\Big( \delta f'\big(\ell_j(\rho), \err_j(\rho)\big) \# \partial_j\rho\Big)
\\&= \delta \phi\big(\ell_j(\rho), \err_j(\rho)\big) \# \partial_j\rho
\\& = \partial_j\phi(\rho) \ .
\end{align*}
We obtain
\begin{align*}
 \cK_\rho  \rmD \cF(\rho)  
    & = \sum_{j \in \cJ} \partial_{j}^\dagger \big( \hrhop \#  \partial_{j} \rmD \cF(\rho) \big)  
= \sum_{j \in \cJ} \partial_{j}^\dagger \partial_j\phi(\rho)  
= - \cL \phi(\rho)\; , 
\end{align*}
which is the desired identity.
\end{proof}

\begin{remark}\label{rem:sign-change}
The result remains true if $f$ is required to be strictly concave and $\phi$ is required to be strictly decreasing. Note that $\theta$ is positive in this case, so that $(\cK_\rho)_\rho$ induces a Riemannian metric.
\end{remark}

\begin{remark}\label{rem:special-cases}
This result contains various known results as special cases.
Take $f(\lambda) = \lambda \log \lambda$ and $\phi(r) = r$. Then the functional $\cF$ is the {von Neumann entropy} $\cF(\rho) = \tau[\rho \log \rho]$, and we recover the special case of Theorem \ref{thm:grad-flow-ent} with $\sigma = \one$. 
It also contains the gradient flow structure for the fermionic Fokker-Planck equation from \cite{CM12}. 
In the special case where $\cL$ is the generator of a reversible Markov chain, we recover the gradient flow structure for discrete porous medium equations obtained in \cite{ErMa14}. 
\end{remark}

\begin{remark}\label{rem:simplifications}
In some situations the expression for $\hrho_j = \theta(\ell_j(\rho), \err_j(\rho))$ can be simplified.
If $f(\lambda) = \lambda \log \lambda$ and $\phi(\lambda) = \lambda$, it follows that 
$\theta(\lambda,\mu) = \frac{\lambda - \mu}{\log \lambda - \log \mu}$ is the \emph{logarithmic mean}.  The integral representation $\theta(\lambda,\mu)
 = \int_0^1 \lambda^{1-s}\mu^s \dd s$ allows one to express $\hrho_j$ in terms of the functional calculus for $\ell_j(\rho)$ and $\err_j(\rho)$:
\begin{align*}
 \hrho_j = \theta(\ell_j(\rho), \err_j(\rho))
       = \int_0^1 \ell_j(\rho)^{1-s} \ot \err_j(\rho)^s \dd s\ .
\end{align*}
More generally, take $m \in \R\setminus \{0,1\}$, and set $\phi(\lambda) = \lambda^m$ and $f(\lambda) = \frac{1}{m-1}\lambda^m$. We shall consider the \emph{power difference means} defined by
\begin{align*}
 \theta_m(\lambda,\mu) := \frac{m-1}{m}\frac{\lambda^m - \mu^m}{\lambda^{m-1} - \mu^{m-1}}\ ,
\end{align*}
with the convention that $\theta_m(\lambda, \lambda) = \lambda$. A systematic study of the operator means associated to these functions has been carried out in \cite{HiKo99}. Various classical means are contained as special cases:
\begin{align*}
 \displaystyle\theta_{m}(\lambda,\mu)=  \left\{ \begin{array}{ll}
 \medskip\displaystyle \frac{2\lambda\mu}{\lambda+\mu}\ , & \text{$m = -1$ (harmonic mean)}\ ,\\
 \medskip\displaystyle \frac{\lambda\mu(\log \lambda - \log \mu)}{\lambda - \mu}\ , & \text{$m \to 0$}\ ,\\
 \medskip\displaystyle \sqrt{\lambda \mu}\ , & \text{$m = \frac12$ (geometric mean)}\ ,\\
 \medskip\displaystyle \frac{\lambda - \mu}{\log \lambda - \log \mu}\ ,& \text{$m \to 1$ (logarithmic mean)}\ ,\\
 \displaystyle\frac{\lambda+\mu}{2}\ ,& \text{$m = 2$ (arithmetic mean)}\ .\\
\end{array} \right.
\end{align*}
The following integral representation holds:
\begin{align}\label{eq:integral-rep}
 \theta_m(\lambda,\mu)= 
   \int_0^1 \Big( (1-\alpha) \lambda^{m-1} + \alpha \mu^{m-1} \Big)^{\frac{1}{m-1}} \dd \alpha\ .
\end{align}
If $m = 2$ and $\ell_j = \err_j = I_\cA$, one has $\widehat\rho_j \# A = \frac12(\rho A + A \rho)$, which corresponds to the anti-commutator case studied in \cite{CGT16}.

Another special case is obtained by taking $\phi(\lambda) = \lambda$ and $f(\lambda) = \lambda^2/2$, which yields $\theta(\lambda, \mu) \equiv 1$, so that $\cK_\rho = -\cL$ for all $\rho$, and $\cF(\rho) =  \frac12\tau[\rho^2] = \frac12 \| \rho\|_{L^2(\tau)}^2$. In this case, the distance associated to $\cK_\rho $ may be regarded as a non-commutative analogue of the Sobolev $H^{-1}$-metric.

\end{remark}

\subsection{Geodesics}

As before we consider the operator $\cK_\rho : \cA_0 \to \cA_0$ given by
\begin{align*}
   \cK_\rho A := \partial_j^\dagger(\widehat\rho_j \# \partial_j A)  \ .
\end{align*}
For fixed $\bar\rho \in \Dens_+$ we will compute the geodesic equations associated to the Riemannian structure on $\cM_{\bar\rho}$ induced by the operator $(\cK_\rho)_\rho$. The Riemannian distance $d_\cK$ is given by 
\begin{align*}
 d_\cK(\tilde\rho_0, \tilde\rho_1)^2 &= \inf_{\rho, A} \bigg\{ 
  \int_0^1 \ip{\cK_{\rho_t} A_t, A_t}_{L^2(\cA,\tau)} \dd t \ : \
 		  \partial_t \rho_t = \cK_{\rho_t} A_t\ , 
		  \rho_0 = \tilde\rho_0\ ,
		  \rho_1 = \tilde\rho_1 \bigg\}
 \\& = \inf_{\rho, A} \bigg\{  
 		\int_0^1 \ip{\widehat\rho_t \# \nabla A_t, \nabla A_t}_{L^2(\cA,\tau)} \dd t \,:\,
		\\& \qquad \qquad\qquad 
 		  \partial_t \rho_t +  \dive(\hrho_t\# \nabla A_t)= 0\,, 
		  \rho_0 = \tilde\rho_0\ ,
		  \rho_1 = \tilde\rho_1 \bigg\}\,,
\end{align*}
where the infimum runs over smooth curves $\{\rho_t\}_{t \in [0,1]}$ in $\cM_{\bar\rho}$ and $\{A_t\}_{t \in [0,1]}$ in $\cA_0$ satisfying the stated conditions.

The geodesic equations are the Euler-Lagrange equations associated to this constrained minimization problem, given by 
\begin{equation}\begin{aligned}\label{eq:EL}
  \partial_s \rho_s & = \cK_{\rho_s} A_s \\
  \partial_s A_s  & = - \frac12  \ip{\rmD\cK_{\rho_s} A_s, A_s}_{L^2(\cA, \tau)} \ .
\end{aligned}\end{equation}
Note that the latter equation is equivalent to 
\begin{align*}
\partial_s \tau[A_s B]  & = - \frac12  \ip{\rmD_B \cK_{\rho_s} A_s, A_s}_{L^2(\cA, \tau)}
\end{align*}
for $B \in \cA_0$, where $\rmD_B \cK_{\rho} = \lim_{\eps \to 0} \eps^{-1}\big(\cK_{\rho + \eps B} - \cK_{\rho}\big)$ denotes the directional derivative.

\begin{proposition}[Geodesic equations]\label{prop:geod}
The geodesic equations for $(\rho_s, A_s)_s$ are given by
\begin{align}
	\label{eq:geode1}
 \partial_s \rho_s  + \dive(\hrho_s \# \nabla A_s) &= 0\ ,\\
 \label{eq:geode2}
 \partial_s A_s + \Phi(\rho_s, A_s) &= 0 \ ,
\end{align}
where 
\begin{align*}
 \Phi(\rho, A)  
	   &= \sum_{j\in \cJ} \sum_{k,m,p}
		 \delta\theta_{j}\Big(\treellr{\bullet}{\lambda_p}{\bullet}{\lambda_k}{\lambda_m}\Big)
    	E_m \ell_j^\dagger\Big( (\partial_j A) \err_j(E_p)  (\partial_j A)^* \Big) E_k 
		\\&= 
		\sum_{j\in \cJ} \sum_{k,m,p} 
	   	  \delta\theta_{j}\Big(\treelrr{\bullet}{\lambda_p}{\bullet}{\lambda_m}{\lambda_k}\Big)
		    	E_k \err_j^\dagger\Big( (\partial_j A)^* \ell_j(E_p)  \partial_j A \Big) E_m 
		\ .
\end{align*}
Here, $\rho = \sum_k \lambda_k E_k$ denotes the spectral decomposition of $\rho$.
\end{proposition}

\begin{remark}\label{rem:geod}
In the sequel we will use \eqref{eq:geode2} in the weak formulation:
\begin{equation}\begin{aligned}\label{eq:geod-weak}
 \partial_s \tau[A_sB] 
 	& = - \sum_{j \in \cJ}
   	 \tau_j\Big[(\partial_j A_s)^* 
    \cN_{\rho,B}^{(\eta),j} \# (\partial_j A_s)
 \Big] \ ,
\end{aligned}\end{equation}
for all $B \in \cA$ and $\eta = 1,2$,
where 
\begin{align*}
 \cN_{\rho,B}^{(1),j} = \delta\theta_j\Big( \widetreellr{\bullet}{\err_j(\rho)}{\ell_j(B)}{\ell_j(\rho)}{\ell_j(\rho)} \Big) \ ,  \qquad
 \cN_{\rho,B}^{(2),j} = \delta\theta_j\Big( \widetreelrr{\bullet}{\ell_j(\rho)}{\err_j(B)}{\err_j(\rho)}{\err_j(\rho)} \Big) \ .
\end{align*}
\end{remark}

\begin{remark}\label{rem:simplified}
If $\theta_j(r,s) := \Lambda(e^{\omega_j/2}r, e^{-\omega_j/2}s)$ where $\Lambda$ is the logarithmic mean, the expression above can be simplified. In this case we have the integral representation 
\begin{align*}
\delta\Lambda\Big(\begin{adjustbox}{trim=0.9cm 0.5cm 0.1cm 0.9cm}
\phantom{$\theta_j$( \ }
\begin{tikzpicture}
\tikzstyle{level 1}=[level distance=0.3cm, sibling distance=0.6cm]
\tikzstyle{level 2}=[level distance=0.3cm, sibling distance=0.6cm]
  \node {\tiny{$\bullet$}} [grow'=up] 
    child {node {\tiny{$x$}}} 
    child {node {\tiny{$\bullet$}} 
      child {node {\tiny{$y$}}} 
      child {node {\tiny{$z$}}} 
    }; 
\end{tikzpicture}
\end{adjustbox}\Big)  
    = \int_0^1  \int_0^1 \int_0^\alpha 
  \bigg[ x^{1-\alpha} 
 	  \cdot   \frac{y^{\alpha - \beta}}{(1-s) + s y}
     \cdot \frac{z^\beta}{(1-s) + s z} 
     \bigg] \dd \beta \dd \alpha \dd s\ ,
\end{align*}
so that 
\begin{align*}
\delta\theta_j\Big(\begin{adjustbox}{trim=0.9cm 0.5cm 0.1cm 0.9cm}
\phantom{$\theta_j$( \ }
\begin{tikzpicture}
\tikzstyle{level 1}=[level distance=0.3cm, sibling distance=0.6cm]
\tikzstyle{level 2}=[level distance=0.3cm, sibling distance=0.6cm]
  \node {\tiny{$\bullet$}} [grow'=up] 
    child {node {\tiny{$x$}}} 
    child {node {\tiny{$\bullet$}} 
      child {node {\tiny{$y$}}} 
      child {node {\tiny{$z$}}} 
    }; 
\end{tikzpicture}
\end{adjustbox}\Big)  
    & = \int_0^1  \int_0^1 \int_0^\alpha e^{-\omega_j \alpha}
  \bigg[ x^{1-\alpha} 
 	  \cdot   \frac{y^{\alpha - \beta}}{(1-s) + s e^{-\omega_j/2} y} \cdot
 \\& \qquad\qquad\qquad\qquad\quad     \cdot \frac{z^\beta}{(1-s) + s e^{-\omega_j/2} z} 
     \bigg] \dd \beta \dd \alpha \dd s\ ,
\end{align*}
which implies that
\begin{align*}
  \Phi(\rho, A) & = \sum_{j \in \cJ}\int_0^1  \int_0^1 \int_0^\alpha   e^{\omega_j \alpha}
  \bigg[ 
 	     \frac{\rho^{\alpha - \beta}}{(1-s)I + s  e^{\omega_j/2}\rho}
		 \ell_j^\dagger\Big( (\partial_j A) \err_j(\rho^{1-\alpha})  (\partial_j A)^* \Big)
		 \\& \qquad\qquad\qquad\qquad\qquad \times
	     \frac{\rho^\beta}{(1-s)I + s  e^{\omega_j/2}\rho}      
     \bigg] \dd \beta \dd \alpha \dd s 
   \\   & = \sum_{j \in \cJ}\int_0^1  \int_0^1 \int_0^\alpha 
  e^{-\omega_j\alpha}  \bigg[ 
 	     \frac{\rho^{\alpha - \beta}}{(1-s)I + s e^{-\omega_j/2}  \rho}
		 \err_j^\dagger\Big( (\partial_j A)^* \ell_j(\rho^{1-\alpha})  \partial_j A \Big)
	   \\& \qquad\qquad\qquad\qquad\qquad \times \frac{\rho^\beta}{(1-s)I + s e^{-\omega_j/2} \rho}      
     \bigg] \dd \beta \dd \alpha \dd s \ .
\end{align*}
\end{remark}

\begin{proof}[Proof of Proposition \ref{prop:geod}]
Proposition \ref{prop:chain-double-simple} yields
\begin{align*}
  \partial_\eps\big|_{\eps=0} \theta_j\Big(\ell_j(& \rho + \eps B), \err_j(\rho + \eps B)\Big)
\\& = 	 
  \delta\theta_j\Big( \widetreellr{\bullet}{\err_j(\rho)}{\ell_j(B)}{\ell_j(\rho)}{\ell_j(\rho)} \Big)
+   \delta\theta_j\Big( \widetreelrr{\bullet}{\ell_j(\rho)}{\err_j(B)}{\err_j(\rho)}{\err_j(\rho)} \Big) \ ,
\end{align*}
and therefore
\begin{align*}
 & \partial_\eps\big|_{\eps=0}
     \ip{\cK_{\rho + \eps B} A, A}_{L^2(\cA, \tau)} 
\\& = 	 
    \sum_{j \in \cJ}
   	 \tau_j\Big[(\partial_j A)^* 
	 \Big\{
  \delta\theta_j\Big( \widetreellr{\partial_j A}{\err_j(\rho)}{\ell_j(B)}{\ell_j(\rho)}{\ell_j(\rho)} \Big)
+   \delta\theta_j\Big( \widetreelrr{\partial_j A}{\ell_j(\rho)}{\err_j(B)}{\err_j(\rho)}{\err_j(\rho)} \Big) \Big\} \Big] \ .
\end{align*}
Since $A$ is self-adjoint, it follows using \eqref{eq:theta-sym} and \eqref{eq:symmetry-equiv} that
\begin{align*}
 &  \tau_j\Big[ (\partial_{j^*} A)^*  \delta\theta_{j^*}\Big( \widetreelrr{\partial_{j^*} A}{\ell_{j^*}(\rho)}{\err_{j^*}(B)}{\err_{j^*}(\rho)}{\err_{j^*}(\rho)} \Big) \Big]
 \\& = \sum_{k,m,p}  
 \delta\theta_{j^*}\Big(\treelrr{\bullet}{\lambda_p}{\bullet}{\lambda_k}{\lambda_m}\Big)
 \tau_j\Big[  (\partial_{j^*} A)^*  \ell_{j^*}(E_p) \big(\partial_{j^*}A \big)\err_{j^*}(E_k B E_m)\Big]
 \\& = \sum_{k,m,p}
 \delta\theta_{j}\Big(\treellr{\bullet}{\lambda_p}{\bullet}{\lambda_k}{\lambda_m}\Big)
\\& \qquad \times\tau_j\Big[  \Big(\err_{j^*}(A) V_j - V_j\ell_{j^*}(A)\Big)  \ell_{j^*}(E_p) \Big(V_j^* \err_{j^*}(A) - \ell_{j^*}(A)V_j^*\Big)  \err_{j^*}(E_k B E_m)\Big]
 \\& = \sum_{k,m,p} 
  \delta\theta_{j}\Big(\treellr{\bullet}{\lambda_p}{\bullet}{\lambda_k}{\lambda_m}\Big)
\\& \qquad \times\tau_j\Big[  \Big(\ell_j(A) V_j - V_j\err_j(A)\Big)  \err_j(E_p) \Big(V_j^* \ell_j(A) - \err_j(A)V_j^*\Big)  \ell_j(E_k B E_m)\Big] 
 \\& = \sum_{k,m,p} 
  \delta\theta_{j}\Big(\treellr{\bullet}{\lambda_p}{\bullet}{\lambda_k}{\lambda_m}\Big)
\tau_j\Big[ (\partial_j A)^*  \ell_j(E_k B E_m) (\partial_j A) \err_j(E_p)  \Big] 
\\& =   
   	 \tau_j\Big[(\partial_j A)^* 
  \delta\theta_j\Big( \widetreellr{\partial_j A}{\err_j(\rho)}{\ell_j(B)}{\ell_j(\rho)}{\ell_j(\rho)} \Big)
\Big] \ .
\end{align*}
This implies the equality of the two sums in \eqref{eq:geod-weak}, and it also follows that 
\begin{align}\label{eq:derivative-of-K}
  \partial_\eps\big|_{\eps=0}
     \ip{\cK_{\rho + \eps B} A, A}_{L^2(\cA, \tau)} 
  = 2 \sum_{j \in \cJ}
   	 \tau_j\Big[(\partial_j A)^* 
  \delta\theta_j\Big( \widetreellr{\partial_j A}{\err_j(\rho)}{\ell_j(B)}{\ell_j(\rho)}{\ell_j(\rho)} \Big)
\Big] \ ,
\end{align}
which yields the weak formulation \eqref{eq:geod-weak} in view of \eqref{eq:EL}.
To obtain \eqref{eq:geode2}, we compute using \eqref{eq:dual-algebra},
\begin{align*}
&  	 \tau_j\Big[(\partial_j A)^* 
  \delta\theta_j\Big( \widetreellr{\partial_j A}{\err_j(\rho)}{\ell_j(B)}{\ell_j(\rho)}{\ell_j(\rho)} \Big)
\Big] 
  \\ &  = \sum_{k,m,p} 
  \delta\theta_{j}\Big(\treellr{\bullet}{\lambda_p}{\bullet}{\lambda_k}{\lambda_m}\Big)
\tau_j\Big[  \ell_j(B E_m) (\partial_j A) \err_j(E_p)  (\partial_j A)^* \ell_j(E_k) \Big] 
  \\ &  = \sum_{k,m,p} 
  \delta\theta_{j}\Big(\treellr{\bullet}{\lambda_p}{\bullet}{\lambda_k}{\lambda_m}\Big)
\tau_j\Big[  B \ell_j^\dagger\Big( \ell_j(E_m) (\partial_j A) \err_j(E_p)  (\partial_j A)^* \ell_j(E_k) \Big) \Big] 
  \\ &  = \sum_{k,m,p} 
  \delta\theta_{j}\Big(\treellr{\bullet}{\lambda_p}{\bullet}{\lambda_k}{\lambda_m}\Big)
\tau_j\Big[  B  E_m \ell_j^\dagger\Big( (\partial_j A) \err_j(E_p)  (\partial_j A)^* \Big) E_k  \Big] 
 \\& = \tau_j\Big[  B \Phi_j(\rho, A) \Big] \ ,
\end{align*}
where
\begin{align*}
	\Phi_j(\rho, A)  
	   = \sum_{k,m,p} 
	   	  \delta\theta_{j}\Big(\treellr{\bullet}{\lambda_p}{\bullet}{\lambda_k}{\lambda_m}\Big)
    	E_m \ell_j^\dagger\Big( (\partial_j A) \err_j(E_p)  (\partial_j A)^* \Big) E_k 
 \ .
\end{align*}
An analogous computation shows that
\begin{align*}
	\Phi_j(\rho, A)  
	   =  \sum_{k,m,p} 
	   	  \delta\theta_{j^*}\Big(\treelrr{\bullet}{\lambda_p}{\bullet}{\lambda_m}{\lambda_k}\Big)
		    	E_k \err_{j^*}^\dagger\Big( (\partial_{j^*} A)^* \ell_{j^*}(E_p) \partial_{j^*} A \Big) E_m \ .
\end{align*}
We thus obtain 
\begin{align*}
 \ip{\rmD \cK_{\rho} A, A}_{L^2(\cA, \tau)} 
   = 2 \Phi(\rho, A) \ ,
\end{align*}
hence the result follows from the Euler-Lagrange equations \eqref{eq:EL}.
\end{proof}

We will use the geodesic equations to compute the Hessian of some interesting functionals on $\cM_\rho$. Note that the Hessian is obtained from the formula
\begin{align*}
  \Hess_\cK \sE(\rho_0)[A_0, A_0]  
    := \partial_s^2\big|_{s=0} \sE(\rho_s) 
\end{align*}
for $A_0 \in \cA_0$, where $(\rho_s, A_s)_s$ evolves according to the geodesic equations \eqref{eq:EL} with initial conditions  $\rho\big|_{s = 0} = \rho_0$ and $\partial_s \big|_{s = 0} \rho_s = \cK_{\rho_0} A_0$.

\begin{proposition}\label{prop:Hessian-Ent}
For $\bar\rho \in \Dens_+$, let $\sE : \cM_{\bar\rho} \to \R$ be a smooth functional, and let us write $\cM(\rho) := \cK_\rho \rmD \cF(\rho)$ for the Riemannian gradient of $\cF$ induced by $(\cK_\rho)_\rho$. Then, the Hessian of $\sE$ is given by 
\begin{align}\label{eq:Hess1}
 \Hess_\cK \sE(\rho)[A, A] 
 =  \tau[ A \rmD_{\cK_{\rho} A} \cM(\rho) ] 
  -  
   	 \tau \Big[(\nabla A)^* 
    \cN_{\rho,\cM(\rho)}^{(\eta)} \# (\nabla A) \Big]
\end{align}
for $A \in \cA_0$ and $\eta = 1, 2$, where $\rmD_{B} \cM(\rho) = \lim_{\eps \to 0} \eps^{-1}\big(\cM(\rho + \eps B) - \cM(\rho)\big)$ denotes the directional derivative.
In particular, if $\cM(\rho) = -\cL^\dagger \rho$ (as is the case in setting of Theorem \ref{thm:grad-flow-ent}, where $\sE(\rho) = \Ent_\sigma(\rho)$), we have
\begin{align}\label{eq:Hess2}
 \Hess_\cK \sE(\rho)[A, A] 
  = 	- \tau[( \nabla \cL A)^* \hrho \# \nabla A] 
			+  \tau\Big[(\nabla A)^* 
    \cN_{\rho,\cL^\dagger \rho}^{(\eta)} \# (\nabla A) \Big] \ .
\end{align}
\end{proposition} 

\begin{proof}
Let $(\rho_s,A_s)_s$ satisfy the geodesic equations \eqref{eq:geode1}--\eqref{eq:geode2}. Then: 
\begin{align*}
 \partial_s \sE(\rho_s) 
     &   = \tau[ \rmD \sE(\rho_s) \partial_s \rho_s ]
		 = \tau[ \rmD \sE(\rho_s) \cK_{\rho_s} A_s ]
		 = \tau[ A_s \cK_{\rho_s} \rmD \sE(\rho_s) ]
		 = \tau[ A_s \cM(\rho_s) ] \ .
\end{align*}
Thus, by \eqref{eq:geod-weak},
\begin{align*}
 \partial_s^2 \sE(\rho_s) 
	& =  \tau[ A_s \partial_s \cM(\rho_s) ] 
			+\tau[ (\partial_s A_s) \cM(\rho_s) ] 
  \\& =  \tau[ A_s \rmD_{\cK_{\rho_s} A_s} \cM(\rho_s) ] 
  -  \sum_{j \in \cJ}
   	 \tau_j\Big[(\partial_j A_s)^* 
    \cN_{\rho_s,\cM(\rho_s)}^{(\eta),j} \# (\partial_j A_s) \Big]
\ ,
\end{align*}
for $\eta = 1, 2$,
which proves \eqref{eq:Hess1}.

If $\cM(\rho) = -\cL^\dagger \rho$ we have $\rmD_{B} \cM(\rho) = - \cL^\dagger B$, hence the expression above simplifies to
\begin{align*}
 \partial_s^2 \sE(\rho_s) 
	& = - \tau[ A_s  \cL^\dagger \cK_{\rho_s} A_s  ]  +  
	 \sum_{j \in \cJ}
   	 \tau_j\Big[(\partial_j A_s)^* 
    \cN_{\rho_s,\cL^\dagger \rho_s}^{(\eta),j} \# (\partial_j A_s) \Big]
	\\& = 	- \tau[( \nabla \cL A_s)^* \hrho_s \# \nabla A_s] 
  +  
	 \sum_{j \in \cJ}
   	 \tau_j\Big[(\partial_j A_s)^* 
    \cN_{\rho_s,\cL^\dagger \rho_s}^{(\eta),j} \# (\partial_j A_s) \Big]
			 \ .
\end{align*}
\end{proof}

\begin{remark}\label{rem:Hess-Ent}
In the setting of the theorem above, we remark that the following equivalent expression holds as well:
\begin{align*}
 \Hess_\cK \sE(\rho)[A, A] 
 =  \tau[ A \rmD_{ \cK_{\rho} A} \cM(\rho) ] 
  - \tau[ \Phi(\rho, A) \cM(\rho) ] \ .
\end{align*}
\end{remark}

\section{Preliminaries on quasi-entropies}\label{sec:quasi-entropies}

In this section we collect some known results on trace functionals that will be useful in the study of quantum transport metrics. Special cases of the results in this section already played a key role in the proof of functional inequalities in \cite{2016-Carlen-Maas}.

Let $\cA$ be a finite-dimensional von Neumann algebra endowed with a positive tracial linear functional $\tau$.
We consider the mapping $\cJ_{\theta,p} : \cA_+ \times \cA_+ \times \cA \to \R$ given by 
\begin{align*} 
 \qquad \cJ_{\theta,p}(R,S;A) := \bip{A, \theta^{-p}(R,S) \# A }
   = \sum_{k,\ell} \theta^{-p}(\lambda_k, \mu_\ell) \tau\big[A^* E_{k}^R A E_{\ell}^S \big]\ ,
\end{align*}
where $\theta : (0,\infty) \times (0,\infty) \to (0,\infty)$ and $p \in \R$, and $R = \sum_k \lambda_k E_k^R$ and $S = \sum_\ell \mu_\ell E_\ell^S$ denotes the spectral decomposition. The main cases of interest to us are $p = \pm 1$.

In this section we shall assume that the function $\theta$ is \emph{$1$-homogeneous}, i.e., $\theta(\lambda r,\lambda s) = \lambda \theta(r,s)$ for all $\lambda, r, s > 0$. 
Clearly, this assumption is satisfied if and only if there exists a function $f : (0,\infty) \to (0,\infty)$ such that $\theta(r,s) = s f(r/s)$ for all $r,s > 0$, in which case we have $f(r) = \theta(r,1)$. 
To simplify notation, we write $k(r) = 1/f(r)$. 

\begin{remark}[Relation to the relative modular operator]
	It is instructive to see how the definition of $\theta(R,S)$ can be formulated in terms of the relative modular operator, if $\theta$ is $1$-homogeneous.
	Given $S \in \cA_+$, let ${\sfL}_S$ and ${\sfR}_S$ denote the left- and right-multiplication operators defined by $\sfL_S(A) = S A$ and $\sfR_S(A) = A S$. Then the relative modular operator $\Delta_{R,S} : \cA \to \cA$ defined by $\Delta_{R,S} A = RAS^{-1}$ can be expressed as $\Delta_{R,S} = \sfL_R \circ \sfR_{S^{-1}} = \sfR_{S^{-1}} \circ \sfL_R$. 
	Let $\{\xi_k\}$ (resp. $\{\eta_\ell\}$) be an orthonormal basis of $\C^n$ consisting of eigenvectors of $R$ (resp. $S$), let $\{\lambda_k\}$ (resp. $\{\mu_\ell\}$) be the corresponding eigenvalues, and set $E_{k\ell} := \phys{\xi_k}{\eta_\ell}$. 
	It follows that $\Delta_{R,S}(E_{k \ell}) = \frac{\lambda_k}{\mu_\ell} E_{k \ell}$, hence the $E_{k\ell}$'s form a complete basis of eigenvectors of $\Delta_{R,S}$. 
	Moreover, the $E_{k\ell}$'s are orthonormal with respect to the Hilbert--Schmidt inner product $\ip{A,B}_{L^2(\Tr)} = \Tr[A^* B]$ on $\M_n(\C)$. 
	Consequently, the spectral decomposition of $\Delta_{R,S}$ is given by  
	 \begin{align*}
	 	\Delta_{R,S} = \sum_{k, \ell} \frac{\lambda_k}{\mu_\ell} \phys{E_{k\ell}}{E_{k\ell}} \ ,
	 \end{align*}
	 and for functions $f : (0,\infty) \to \R$ we find $f(\Delta_{R,S})(A) = \sum_{k, \ell}f (\lambda_k / \mu_\ell)  \ip{E_{k\ell}, A}_{L^2(\Tr)} E_{k\ell}$.
	 Note that 
	 \begin{align*}
	 	 \ip{E_{k\ell}, A}_{L^2(\Tr)} E_{k \ell}
	 	  = \sum_m \ip{E_{k\ell}\eta_m, A\eta_m } E_{k \ell}
	 	  = \ip{\xi_k, A\eta_\ell } E_{k \ell}
		  = E_{k}^R A E_{\ell}^S \ ,
	 \end{align*}
	 where $E_{k}^R = \phys{\xi_k}{\xi_k}$ and $E_{\ell}^S = \phys{\eta_\ell}{\eta_\ell}$.
	 It follows that 
	 \begin{align*}
	 	f(\Delta_{R,S})(A)
	 	 =  \sum_{k, \ell} f(\lambda_k / \mu_\ell) E_{k}^R A E_{\ell}^S \ ,
	 \end{align*}
	 and therefore, since $f(r/s) s = \theta(r,s)$,
	 \begin{align*}
	 	\theta(R,S) \# A = \big(\sfR_S \circ f(\Delta_{R,S}) \big)(A) \ .
	 \end{align*}
\end{remark}

\begin{example}\label{ex:power-means}
	Let us recall our main examples of interest. A central role is played by the  \emph{tilted logarithmic mean} $\theta_{1,\beta}$ given by
\begin{align*}
	\theta_{1,\beta}(r,s) = \int_0^1 (e^{-\beta/2}r)^{1-\alpha} (e^{\beta/2}s)^\alpha \dd \alpha 
	 			= \frac{e^{-\beta/2}r - e^{\beta/2}s}{- \beta + \log r - \log s} \ , \quad
	  	f_{1,\beta}(r) = \frac{e^{-\beta/2} r - e^{\beta/2}}{-\beta + \log r} \ ,
\end{align*}
for $\beta \in \R$.
More generally, in view of Remark \ref{rem:simplifications} we are interested in the class of power difference quotients $\theta_m$ given by $f_{m,\beta}(r) = \theta_{m,\beta}(r,1)$, where
\begin{align*}
 \theta_{m,\beta}(r,s)
& =  
\int_0^1 \Big( (1-\alpha)(e^{-\beta/2}r)^{m-1}  + \alpha (e^{\beta/2}s)^{m-1} \Big)^{\frac{1}{m-1}} \dd \alpha 
   \\   & = \frac{m-1}{m}\frac{(e^{-\beta/2}r)^m - (e^{\beta/2}s)^m}{(e^{-\beta/2}r)^{m-1} - (e^{\beta/2}s)^{m-1}} \ .
\end{align*}
\end{example}

Consider the mapping $\Upsilon_{f,p} : \cA_+\times \cA_+ \times \cA \to \R$ given by
\begin{align*}
	\Upsilon_{f,p}(R,S;A) :=  
	 \cJ_{\theta,p}(R,S;A) = \bip{A, \theta^{-p}(R,S) \# A  }
\end{align*}
Our goal is to characterize its convexity and contractivity properties in terms of $f$ and $m$. For this purpose we recall that a function $f : (0,\infty) \to (0,\infty)$ is said to be \emph{operator monotone}, whenever $f(A) \leq f(B)$ for all positive matrices $A \leq B$ in all dimensions. Each operator monotone function is continuous, non-decreasing and concave. We set $f(0) := \inf_{t > 0} f(t)$. 

The following result has been obtained in \cite[Theorem 2.1]{HP12CQE}. The implication ``$(2) \Rightarrow (1)$'', as well as the reverse implication for fixed $p =1$ had already been proved in \cite{HP12QEQ}.

\begin{theorem}[Characterization of convexity of $\Upsilon_{f,p}$]\label{thm:op-mon-cvx}
Let $f : (0,\infty) \to (0,\infty)$ be a function and let $p \in \R \setminus \{0 \}$. The following assertions are equivalent.
\begin{enumerate}
\item The function $\Upsilon_{f,p}$ is jointly convex in its three variables;
\item The function $f$ is operator monotone and $p \in (0,1]$.
\end{enumerate}
\end{theorem}

Applying this result to the functions $f = f_{m,\beta}$, we obtain the following result.

\begin{corollary}[Characterization of convexity of $\Upsilon_{f,p}$ for power difference quotients]\label{cor:conv-char}\mbox{}\\ For $m \in \R \setminus \{0\}$ and $\beta \in \R$, let $f = f_{m,\beta}$ and $\theta = \theta_{m,\beta}$ be as in Example \ref{ex:power-means}.
Then, the associated mapping $\Upsilon_{f,p}$ is jointly convex if and only if $m \in [-1,2]$, $p \in (0,1]$, and $\beta \in \R$.
In particular, the mapping 
\begin{align*}
	(R,S,A) \mapsto
	  \bip{A^*, \theta_{1,\beta}^{-1}(R,S) \# A  }
	= \tau\bigg[\int_0^\infty A^*  \frac{1}{x + e^{-\beta/2}R}   A \frac{1}{x + e^{\beta/2}S} \dd x \bigg]
\end{align*}
is jointly convex for all $\beta \in \R$.
\end{corollary}

\begin{proof}
	Since $f_{m,\beta}(s) = e^{\beta/2} f_{m,0}(e^{-\beta}s)$, the operator monotonicity of $f_{m,\beta}$ does not depend on $\beta$.
	It has been proved in \cite[Proposition 4.2]{HiKo99}, that $f_{m,0}$ is operator monotone if and only if $m \in [-1,2]$. 
	Hence, the first assertion follows from Theorem \ref{thm:op-mon-cvx}.
		The second assertion is the special case $m = p = 1$, noting that 
		\begin{align*}
			\frac{1}{\theta_{1,\beta}(r,s)} = \int_0^\infty \frac{1}{x + e^{-\beta/2}r} \, \frac{1}{x + e^{\beta/2}s} \dd x \ .
		\end{align*}
\end{proof}

\begin{remark}\label{rem:log-direct}
In the case where $\theta = \theta_{1,\beta}$, the operator monotonicity of $f_{1,\beta}$ can be checked elementarily, by writing $f_{1,\beta}(r) = \int_0^1 e^{-\beta(1/2-\alpha)} r^\alpha \dd \alpha$, and applying the L\"owner-Heinz Theorem (e.g., \cite[Theorem 2.6]{Car11}), which asserts that the function $r \mapsto r^\alpha$ is operator monotone for $\alpha \in [0,1]$.
\end{remark}

The following result is proved in \cite[Theorem 5]{HP12QEQ}.

\begin{theorem}[Contractivity of $\Theta_{f,p}$ under CPTC maps]\label{thm:contr-CPTC}
Suppose that $f : (0,\infty) \to (0,\infty)$ is operator monotone. Then, for any $R, S\in \cA_+$ and $A \in \cA$, and for any completely positive and trace preserving map $\cT : \cA \to \cA$, we have
\begin{align}\label{eq:contr-CPTC}
		\Upsilon_{f,1}\big(\cT(R),\cT(S);\cT(A)\big) \leq 	\Upsilon_{f,1}(R,S;A) \ .
\end{align}
\end{theorem}

In the case where $f = f_{m, \beta}$ as in Example \ref{ex:power-means}, we obtain the following result.

\begin{corollary}[Contractivity of $\Theta_{f,p}$ for power difference quotients]
\label{cor:contr-CPTC}
Let $m \in [-1,2]$ and $\beta \in \R$, and let $f = f_{m,\beta}$ and $\theta = \theta_{m,\beta}$ be as in Example \ref{ex:power-means}. 
Then, for any $R, S\in \cA_+$ and $A \in \cA$, and for any completely positive and trace preserving map $\cP : \cA \to \cA$, \eqref{eq:contr-CPTC} holds.
In particular, for $m = 1$ we obtain
\begin{align*}
 \tau\bigg[\int_0^\infty \cT(A)^* & \frac{1}{x  + e^{-\beta/2}\cT(R)} \cT(A) \frac{1}{x + e^{\beta/2}\cT(S)} \dd x \bigg] \\
&  \quad \leq
 \tau\bigg[\int_0^\infty A^*  \frac{1}{x + e^{-\beta/2}R}   A \frac{1}{x + e^{\beta/2}S} \dd x \bigg] \ .
\end{align*}
\end{corollary}

\begin{proof}
	This follows from Theorem \ref{thm:contr-CPTC}, as the operator monotonicity of $f_{m,p}$ had already been noted in Corollary \ref{cor:conv-char}. 
\end{proof}

\section{The Riemannian distance}

Fix a differentiable structure $(\cA, \nabla, \sigma)$ in the sense of Definition \ref{ass:setup} and a collection of functions $(\theta_j)_j$ satisfying Assumption \ref{ass:theta}. For simplicity we restrict ourselves to the ergodic case, so that $\cM_\rho = \Dens_+$ for all $\rho \in \Dens_+$. 

In this section we study basic properties of the Riemannian distance $\cW$ associated to the operators $(\cK_\rho)_\rho$ defined in \eqref{eq:def-K}. For $\rho_0, \rho_1 \in \Dens_+$ this distance is given by
\begin{equation}\begin{aligned}\label{eq:Riem-dist}
\cW(\rho_0, \rho_1)^2 
  & = 
 \inf_{} \bigg\{ \int_0^1 \ip{\cK_{\rho_t} A_t, A_t} \dd t \ : \
 		  \partial_t \rho_t = \cK_{\rho_t}A_t\ , 
		  \rho_t|_{t=0} = \rho_0\ ,
		  \rho_t|_{t=1} = \rho_1 \bigg\}
 \\ & = 
 \inf_{} \bigg\{ 
\int_0^1  \tau\big[(\nabla A_t)^* \hrho_t \# \nabla A_t \big] \dd t \ : \ \\& \qquad\qquad\qquad 
		  \partial_t \rho_t + \dive(\hrho_t\# \nabla A_t)= 0\ , 	\  \rho_t|_{t=0,1} = \rho_{0,1}  \bigg\}\ ,
\end{aligned}\end{equation} 
where the infimum runs over smooth curves $(\rho_t)_{t \in [0,1]}$ in $\Dens_+$  and $(A_t)_{t \in [0,1]}$ in $\cA_0$ satisfying the stated conditions.

In the classical theory of optimal transport, it is a useful fact that the following equivalent formulations hold for the $2$-Kantorovich distance on $\R^n$: 
\begin{equation}\begin{aligned}
	\label{eq:classical-distance}
	 W_2(\rho_0, \rho_1)^2 & = 
 \inf_{} \bigg\{ \int_0^1 |\nabla \psi_t(x)|^2  \dd \rho_t(x) \dd t \ : \
 		  \partial_t \rho_t + \dive(\rho_t \nabla \psi_t) = 0 \ , 
		  \rho_t|_{t=0,1} = \rho_{0,1} \bigg\}
		  \\& = \inf_{} \bigg\{ \int_0^1 \frac{|P_t(x)|^2}{\rho_t(x)}  \dd x \dd t \ : \
 		  \partial_t \rho_t + \dive P_t = 0 \ , 
		  \rho_t|_{t=0,1} = \rho_{0,1} \bigg\} \ .
\end{aligned}\end{equation}	
The latter formulation has the advantage that the minimisation problem is convex, due to the convexity of the function $(p,r) \mapsto \frac{|p|^2}{r}$  on $\R^n \times (0,\infty)$.

Using the convexity results presented in Section \ref{sec:quasi-entropies} we will show that an analogous result holds in the non-commutative setting. 
We use the shorthand notation 
\begin{align*}
	  \ip{\bbB, \bbC}_\rho = \sum_j \tau_j[B_j^* (\hrho_j \# C_j) ] \ , \qquad
 \ip{\bbB, \bbC}_{-1,\rho} = \sum_j \tau_j[B_j^* (\lrho_j \# C_j)]  \ ,
\end{align*}
to denote the scalar products that will frequently appear below. The corresponding norms are given by $\|\bbB\|_\rho = \sqrt{ \ip{\bbB, \bbB}_\rho}$ and $\|\bbB\|_{-1, \rho} = \sqrt{ \ip{\bbB, \bbB}_{-1,\rho}}$.
It will occasionally be convenient to write
\begin{align*}
	\sA(\rho; \bbB, \bbC)
	= \ip{\bbB, \bbC}_{-1,\rho} \quad \text{and} \quad
	\sA(\rho, \bbB)
	= \|\bbB\|_{-1, \rho}^2 \ .
\end{align*}

We start with a non-commutative analogue of \eqref{eq:classical-distance}.

\begin{lemma}\label{lem:BB-form}
For $\rho_0, \rho_1 \in \Dens_+$ we have
\begin{align}\label{eq:BB-form}
	\cW(\rho_0, \rho_1)^2 
	& = 
 \inf_{} \bigg\{ 
\int_0^1   \| \bbB_t \|_{-1, \rho_t}^2 \dd t \ : \ 
		  \partial_t \rho_t + \dive \bbB_t = 0\ , 	\  \rho_t|_{t=0,1} = \rho_{0,1}  \bigg\} \ ,
\end{align}
where the infimum runs over all smooth curves $(\rho_t)_{t \in [0,1]}$ in $\Dens_+$  and $(\bbB_t)_{t \in [0,1]}$ in $\cB$.
\end{lemma}

\begin{proof}
Any admissible curve $(A_t)$ in \eqref{eq:Riem-dist} yields an admissible curve $(\bbB_t)$ in \eqref{eq:BB-form} given by $\bbB_t = \hrho_t \nabla A_t$, that satisfies $\|\nabla A_t\|_{\rho_t} =  \| \bbB_t \|_{-1, \rho_t}$. This implies the inequality ``$\geq$'' in \eqref{eq:BB-form}. 

To prove the reverse inequality, we take an admissible curve $(\rho_t, \bbB_t)_t$ in \eqref{eq:BB-form}. We consider the linear space of gradient vector fields $\cG = \{\nabla A  \ :  A \in \cA_0 \}$, and let $\cD_t \subseteq \cG$ denote its orthogonal complement in $\cB$ with respect to the scalar product product $\ip{\cdot, \cdot}_{\rho_t}$. 
Consider the orthogonal decomposition 
\begin{align*}
	\lrho_j \# \bbB_t = \nabla A_t + D_t \in \cG \oplus  \cD_t \ .
\end{align*}
Since $\ip{\nabla \tilde A, D_t}_{\rho_t} = 0$ for all $\tilde A \in \cA_0$, it follows that $\dive (\hrho \# D_t) = 0$. Therefore, $\partial_t \rho_t + \dive (\hrho_t\# \nabla A_t) = 0$. Moreover, 
\begin{align*}
	\tau\big[(\nabla A_t)^* \hrho_t \# \nabla A_t \big]
		= \| \nabla A_t \|_{\rho_t}^2
		\leq \| \lrho_j \# \bbB_t \|_{\rho_t}^2
		= \| \bbB_t \|_{-1,\rho_t}^2 \ ,
\end{align*}
which yields the inequality ``$\leq$'' in \eqref{eq:BB-form}.
\end{proof}

\begin{proposition}[Extension of the distance to the boundary]\label{prop:dist-formula}
Suppose that $\theta_j(a,b) \geq C\min\{a,b\}^p$ for some $C > 0$ and $p<2$. 
Then the distance function $\cW : \Dens_+ \times \Dens_+ \to \R$ extends continuously to a metric on $\Dens$.
\end{proposition}

\begin{proof}
Let $\rho_0, \rho_1 \in \Dens$ and let $\{\rho_0^n\}_n, \{\rho_1^n\}_n$ be sequences in $\Dens_+$ satisfying $\tau\big[ |\rho_i^n  -\rho_i|^2 \big] \to 0$ as $n \to \infty$ for $i = 0, 1$. 
We claim that the sequence $\{ \cW(\rho_0^n, \rho_1^n) \}_n$ is Cauchy.

To prove this, it suffices to show that $\cW(\rho_i^n, \rho_i^m) \to 0$ as $n, m\to \infty$ for $i=0,1$, since
\begin{align*}
 |\cW(\rho_0^n, \rho_1^n) - \cW(\rho_0^m, \rho_1^m)|
  \leq  \cW(\rho_0^n, \rho_0^m) +  \cW(\rho_1^n, \rho_1^m) \ . 
\end{align*} 
Fix $\eps \in (0,1)$, and set $\tilde\rho := (1- \eps) \rho_0 + \eps \one$. Take $N \geq 1$ so large that $\tau\left[ |\rho_0^n  - \rho_0|^2 \right] \leq \eps^2$ whenever $n \geq N$. 
For $n \geq N$ we consider the linear interpolation $\rho^n_t = (1-t) \rho_0^n + t \tilde \rho$. 
Then $\dot \rho_t^n = \tilde \rho - \rho_0^n$ for all $t \in (0,1)$. 
Since $\cK_\one$ is invertible on $\cA_0$ by Lemma \ref{lem:Ran-Ker} and  ergodicity, we may define $A := \cK_\one^{-1} (\tilde \rho - \rho_0^n) \in \cA_0$, and we have $\dot \rho_t^n = \dive(\nabla A)$.
Since $\rho^n_t \geq t \eps \one$ for $t \in [0,1]$, we have $\frac{1}{\theta_j}\big|_{\spec(\rho^n_t)} \leq C (t\eps)^{-p}$, and thus $\tau [ (\nabla A)^* \lrho_t^n \# \nabla A] \leq C  (t\eps)^{-p} \tau [ |\nabla A|^2 ]$ by  Lemma \ref{lem:Schur-bound}.
It follows that
\begin{align*}
 \cW(\rho_0^n, \tilde \rho) 
& \leq \int_0^1 
  \sqrt{ \tau [ (\nabla A)^* \lrho_t^n \# \nabla A] } \dd t 
\leq C \eps^{-p/2} \|\nabla A\|_{L^2(\tau)}  \ ,
\end{align*}
since $p<2$.
Using the boundedness of $\nabla \circ \cK_\one^{-1}$ we obtain
\begin{align*}
C^{-1} \|\nabla A\|_{L^2(\tau)}
& \leq  \| \tilde \rho - \rho_0^n \|_{L^2(\tau)}
 \\   & = \| \rho_0 - \rho_0^n \|_{L^2(\tau)}  +  \eps \| \one - \rho_0 \|_{L^2(\tau)}
    \\& \leq \eps \big(1 + \| \one - \rho_0 \|_{L^2(\tau)}\big) \ .     
\end{align*}
We infer that $\cW(\rho_0^n, \tilde \rho)  \leq C \eps^{1-p/2}$ for some $C < \infty$ depending on $\rho_0$. 
It follows that $\cW(\rho_0^n, \rho_0^m) \leq C \eps^{1-p/2}$ for $n, m \geq N$. Since $p < 2$, this proves the claim.

We can thus extend $\cW$ to $\Dens$ by setting $\cW(\rho_0, \rho_1) = \lim_{n \to \infty}\cW(\rho_0^n, \rho_1^n)$. It immediately follows that $\cW$ is symmetric and the triangle inequality extends to $\Dens$. The fact that $\cW(\rho_0, \rho_1) \neq 0$ whenever $\rho_0$ and $\rho_1$ are distinct, follows from Proposition \ref{prop:Wass-1-comparison} below.
\end{proof}

Our next aim is to prove Proposition \ref{prop:Wass-1-comparison} below, which yields a lower bound on the distance $\cW$  in terms of a non-commutative analogue of the $1$-Kantorovich metric. To formulate the result, we use the notation
\begin{align*}
 \| \bbB \|_{\cB,2} := \sqrt{\frac12\sum_j   \| \ell_j^\dagger( B_j B_j^*) +  \err_j^\dagger(B_j^* B_j) \|_{\cA} }  
\end{align*}
for $\bbB = (B_j)_{j \in \cJ} \in \cB$.

\begin{lemma}\label{lem:Lip-bound}
There exists $M < \infty$ such that $\| \bbB \|_\rho \leq M \| \bbB \|_{\cB,2}$ for all $\rho \in \Dens_+$ and $\bbB \in \cB$.
If $\theta_j(r,s) \leq \frac12(r+s)$ for all $r,s > 0$, then this estimate holds with $M = 1$:
\begin{align*}
	\| \bbB \|_\rho \leq \| \bbB \|_{\cB,2} \ .
\end{align*}
\end{lemma}

\begin{proof} 
Recalling that $\|\cdot\|_{\cB_j}$ denotes the norm on $\cB_j$, we define 
\begin{align*}
M_j := \sup \Big\{ \tau_j[ \, | \hrho_j \# B | \, ] \ : \rho \in \Dens_+   \ , \ \|B\|_{\cB_j} \leq 1 \Big\} \ , \quad \text{and} \quad \tilde M := \sup_{j \in \cJ} M_j \ .
\end{align*}
Since our setting is finite-dimensional, $\tilde M$ is finite and all norms on $\cB$ are equivalent. Thus, for a suitable constant $M < \infty$, it follows that
\begin{align*}
  \| \bbB\|_{\rho}^2 
  &=  \sum_j \tau_j[B_j^* \hrho_j \# B_j ]
 \leq \sum_j \|B_j^*\|_{\cB_j}\tau_j[ | \hrho_j \# B_j |]
 \leq \tilde M \sum_j \|B_j\|_{\cB_j}^2 
 \leq  M \| \bbB \|_{\cB,2}^2 \ ,
\end{align*}
which proves the first statement.

Suppose now that $\theta_j(r,s) \leq \frac12(r+s)$. Since 
$\rho$ is positive and the operators $\ell_j$ and $\err_j$ preserve positivity, we obtain using Lemma \ref{lem:Schur-bound},
\begin{align*}
  \| \bbB\|_{\rho}^2 
   &= \sum_j \tau_j[B_j^* \hrho_j \# B_j ]
   \\& \leq \frac12\sum_j \tau_j[ \ell_j(\rho) B_j B_j^*
   				 + \err_j(\rho) B_j^* B_j  ]
   \\& = \frac12\sum_j \tau\big[\rho \big(\ell_j^\dagger( B_j B_j^*)  
   				 + \err_j^\dagger(B_j^* B_j) \big) \big]
   \\& \leq \frac12\sum_j \| \ell_j^\dagger( B_j B_j^*)  
   				 + \err_j^\dagger(B_j^* B_j) \|_{\cA} \ ,
\end{align*}
which yields the result.
\end{proof}

 For $\rho_0, \rho_1 \in \Dens$ we set
\begin{align}\label{eq:W1}
	W_1(\rho_0, \rho_1) 
	  := \sup \bigg\{ \tau[(\rho_1 - \rho_0)A] : A \in \cA \ , \|\nabla A\|_{\cB,2} \leq 1 \bigg\} \ .
\end{align}
By analogy with the dual Kantorovich formulation of the commutative $1$-Kantorovich metric $W_1$ in terms of Lipschitz functions, this metric can be seen as a non-commutative analogue of $W_1$. The following result generalizes a result from \cite{EM12} from the discrete to the non-commutative setting; see also \cite{RD17} for non-commutative results of this type.

\begin{proposition}\label{prop:Wass-1-comparison}
Let $M$ be as in Lemma \ref{lem:Lip-bound} and set $N := \sup\{ \|\nabla A\|_{\cB,2} \ : \ \|A\|_\cA \leq 1 \}$. Then, for $\rho_0, \rho_1 \in \Dens$ we have
\begin{align*}
 N^{-1} \tau[ |\rho_0 - \rho_1 |]  \leq W_1(\rho_0, \rho_1) \leq M \cW(\rho_0,\rho_1) \ .
\end{align*}
\end{proposition}

\begin{proof}
The first inequality follows from the definitions, since $\tau[|B|] = \sup_{\|A\|_\cA \leq 1} \tau[AB]$ for $B \in \cA$.

Fix $\eps > 0$, take $\bar\rho_0, \bar\rho_1 \in \Dens$, and let $(\rho_t, B_t)_t$ be a solution to the continuity equation with approximately optimal action, i.e.,  
\begin{align*}
 \partial_t \rho_t + \dive(\hrho_t \# \nabla B_t) = 0 \tand
 \bigg(\int_0^1 \| \nabla B_t\|_{\rho_t}^2 \dd t \bigg)^{\frac12} \leq \cW(\bar\rho_0, \bar\rho_1 ) + \eps\;.
\end{align*}
For any $A \in \cA_h$ we obtain using Lemma \ref{lem:Lip-bound}
\begin{align*}
  \big|  \tau[A (\bar\rho_0 - \bar\rho_1) ]  \big|
  & = \bigg|\int_0^1 \tau[ A \dot\rho_t ] \dd t\bigg|
   \\& = \bigg|\int_0^1 \tau[ A \dive(\hrho_t \# \nabla B_t)  ] \dd t\bigg|
   \\& = \bigg|\int_0^1 \ip{ \nabla A, \nabla B_t }_{\rho_t} \dd
  t\bigg| 
   \\& \leq \bigg( \int_0^1 \| \nabla A \|_{\rho_t}^2\dd t
  \bigg)^{1/2} \bigg( \int_0^1 \| \nabla B_t \|_{\rho_t}^2 \dd t
  \bigg)^{1/2} 
   \\& \leq  M \|\nabla A \|_{\cB,2} \big(  \cW(\bar\rho_0, \bar\rho_1) + \eps\big) \;.
\end{align*}
Since $\eps > 0$ is arbitrary, the result follows by definition of $W_1$.
\end{proof}

In the remainder of this section we impose the following natural additional conditions in addition to Assumption \ref{ass:theta}.

\begin{assumption}\label{ass:theta-ass}
The functions $\theta_j: [0,\infty) \times [0,\infty) \to [0,\infty)$ are $1$-homogeneous (which implies that $\theta_j(r,s) = s f_j(r/s)$ for some function $f_j$).
The functions $f_j$ are assumed to be operator monotone.
\end{assumption}

Under this assumption, we will prove some crucial convexity properties for the action functional and the squared distance.

\begin{proposition}[Convexity of the action]\label{prop:convex-action}
Let $\rho^i\in\Dens$ and $\bbB^i\in \cB$ for $i=0,1$. 
For $s\in[0,1]$ set $\rho^s := (1-s) \rho^0 + s \rho^1$ and $\bbB^s := (1-s) \bbB^0 + s \bbB^1$. 
Then we have
\begin{align*}
    \sA(\rho^s,\bbB^s) 
    	\leq (1-s)\sA(\rho^0, \bbB^0) 
			+ s \sA(\rho^1,\bbB^1)\ .
\end{align*}
\end{proposition}

\begin{proof}
This follows immediately from Theorem \ref{thm:op-mon-cvx} in view of Assumption \ref{ass:theta-ass}.
\end{proof}

\begin{theorem}[Convexity of the squared distance]
\label{thm:convex-distance}
\label{conj:convex}
  For $i = 0, 1$, let $\rho_0^i, \rho_1^i \in \Dens$, and for $s \in
  [0,1]$ set $\rho_0^s := (1-s) \rho_0^0 + s \rho_0^1$ and $\rho_1^s := (1-s) \rho_1^0 + s \rho_1^1$. Then:
\begin{align*}
 \cW(\rho_0^s,\rho_1^s)^2 
   \leq (1 - s) \cW(\rho_0^0, \rho_1^0)^2 
   			+ s \cW(\rho_0^1, \rho_1^1)^2\ .
\end{align*}
\end{theorem}

\begin{proof}
Fix $\eps > 0$. By continuity, it suffices to prove the inequality for $\rho_0^i, \rho_1^i \in \Dens_+$ and $i = 0,1$. Let $(\rho^i_t, \bbB^i_t)_t$ be such that $\partial_t \rho_t^i + \dive \bbB_t^i = 0$ and $\int_0^1 \sA(\rho_t^i, \bbB_t^i) \dd t \leq \cW(\rho_0^i, \rho_1^i)^2 + \eps$. For $s \in [0,1]$ we define
\begin{align*}
	\rho_t^s := (1-s) \rho_t^0 + s \rho_t^1 
		\quad \text{ and } 	\quad
    \bbB_t^s := (1-s) \bbB_t^0 + s \bbB_t^1 \ . 
\end{align*} 
It follows that $\partial_t \rho_t^s + \dive \bbB_t^s = 0$, and by Lemma \ref{lem:BB-form} and Proposition \ref{prop:convex-action} we obtain
\begin{align*}
	 \cW(\rho_0^s,\rho_1^s)^2 
	 & \leq \int_0^1 \sA(\rho_t^s, \bbB_t^s) \dd t
	 \\& \leq (1-s) \int_0^1 \sA(\rho_t^0, \bbB_t^0) \dd t
	 		+ s \int_0^1 \sA(\rho_t^1, \bbB_t^1) \dd t
	 \\& \leq (1 - s) \cW(\rho_0^0, \rho_1^0)^2 
   			+ s \cW(\rho_0^1, \rho_1^1)^2 + 2 \eps \ .
\end{align*}
Since $\eps > 0$ is arbitrary, the desired inequality follows. 
\end{proof}

Using these convexity properties, the existence of constant speed geodesics for the metric $\cW$ follows by standard arguments; cf. \cite[Theorem 3.2]{EM12}) for a proof in the commutative setting and \cite{RD17} for a proof in a non-commutative context.

\begin{theorem}[Existence of $\cW$-geodesics]\label{thm:existence-geod}
For any $\bar\rho_0, \bar\rho_1 \in \Dens$ there exists a curve $\rho : [0,1] \to \Dens$ satisfying $\rho_0 = \bar\rho_0$, $\rho_1 = \bar\rho_1$, and $\cW(\rho_s, \rho_t) = | s-t | \cW(\rho_0, \rho_1)$ for all $s, t \in [0,1]$.
\end{theorem}

\section{Geodesic convexity of the entropy}

In this section we will analyse geodesic convexity of the relative entropy functional $\Ent_\sigma$.
Throughout this section we fix a differential structure $(\cA, \nabla, \sigma)$ and assume that the associated quantum Markov semigroup $(\cP_t)$ is ergodic. We consider the transport metric $\cW$ defined in Theorem \ref{thm:grad-flow-ent} using the functions $\theta_j$ given by $\theta_j(r,s) := \Lambda(e^{\omega_j/2}r, e^{-\omega_j/2}s)$, so that the Kolmogorov forward equation $\partial_t \rho = \cL^\dagger \rho$ is the gradient flow of the relative von Neumann entropy $\Ent_\sigma$ with respect to the Riemannian metric induced by $(\cK_\rho)_\rho$. 

The following terminology will be useful.

\begin{definition}\label{def:convex}
Let $(\cX,d)$ be a metric space. A functional $\cF : \cX \to \R \cup \{ + \infty \} $ is said to be 
\begin{itemize}
\item \emph{weakly geodesically $\lambda$-convex} if any pair $x_0, x_1 \in \cX$ can be connected by a geodesic $(\gamma_t)_{t \in [0,1]}$ in $(\cX, d)$ along which $\cF$ satisfies the $\lambda$-convexity inequality
\begin{align}\label{eq:lambda-convex}
  \cF(\gamma_t) \leq
  (1-t) \cF(\gamma_0) + t \cF(\gamma_1) - \frac\kappa{2} t(1-t) d(x_0, x_1)^2\ .
\end{align}
\item \emph{strongly geodesically $\lambda$-convex} if \eqref{eq:lambda-convex} holds for any geodesic $(\gamma_t)_{t \in [0,1]}$ in $(\cX, d)$.
\end{itemize}
\end{definition}

The following result, shows in particular that these concepts are equivalent in our setting and provides several equivalent characterizations of geodesic $\lambda$-convexity. We shall use the notation
\begin{align*}
  \ddtr f(t) = \limsup_{h \downarrow 0} \frac{f(t+h) - f(t)}{h}\ .
\end{align*}
We refer to \cite{EM12} for a version of this result in the discrete setting, and to \cite{RD17} for the Lindblad setting.

\begin{theorem}[Characterizations of geodesic $\lambda$-convexity]\label{thm:Ric-equiv}
  Let $\lambda \in \R$. For a differential structure $(\cA, \nabla, \sigma)$ the following assertions are equivalent:
\begin{enumerate}
\item $\Ent_\sigma$ is weakly geodesically $\lambda$-convex on $(\Dens,\cW)$;
\item $\Ent_\sigma$ is strongly geodesically $\lambda$-convex on $(\Dens,\cW)$;
\item For all $\rho, \nu \in \Dens$, the following `evolution
  variational inequality' holds for all $t \geq 0$:
\begin{align}\label{eq:EVI}
  \frac{1}{2}\ddtr \cW^2(\cP_t^\dagger \rho, \nu) + \frac{\lambda}{2}
\cW^2(\cP_t^\dagger \rho, \nu) \leq \Ent_\sigma(\nu) - \cH(\cP_t^\dagger \rho)\;;
\end{align}
\item \label{it:Hess}For all $\rho \in \Dens_+$ and $A \in \cA_0$ we have 
\begin{align*}
 \Hess_\cK \Ent_\sigma(\rho)[A, A]  \geq \lambda \tau[A \cK_\rho A] \ .
\end{align*}
\end{enumerate}
\end{theorem}

\begin{proof}

``$(4) \Rightarrow (3)$'' This can be proved by an argument from \cite{DS08}; see \cite[Theorem 4.5]{EM12} for a proof in a similar setting.

``$(3) \Rightarrow (2)$'': This follows from an application of \cite[Theorem
3.2]{DS08} to the metric space $(\Dens, \cW)$.

``$(2) \Rightarrow (1)$'': Since $(\Dens, \cW)$ is a geodesic space, this implication is immediate.

``$(1) \Rightarrow (4)$'': Obvious. 
\end{proof}

In the classical setting, the Ricci curvature on a Riemannian manifold $\cM$ is bounded from below by $\lambda \in \R$ if and only if the entropy (with respect to the volume measure) is geodesically $\lambda$-convex in the space of probability measures $\cP(\cM)$ endowed with the Kantorovich metric $W_2$. This characterisation is the starting point for the synthetic theory of metric measure spaces with lower Ricci curvature bounds, which has been pioneered by Lott, Sturm and Villani.

By analogy, we make the following definition in the non-commutative setting, which extends the corresponding definition in the discrete setting \cite{EM12}.

\begin{definition}[Ricci curvature]\label{def:Ricci}
Let $\lambda \in \R$. We say that a differential structure $(\cA, \nabla, \sigma)$ has Ricci curvature bounded from below by $\lambda$ if the equivalent conditions of Theorem \ref{thm:Ric-equiv} hold. In this case, we write $\Ric(\cA, \nabla, \sigma) \geq \lambda$.
\end{definition}

It is possible to characterize Ricci curvature in terms of a gradient estimate in the spirit of Bakry--\'Emery; see \cite{ErFa18} for the corresponding statement in the setting of finite Markov chains and  \cite{RD17} for an implementation in the Lindblad setting. 

\begin{theorem}[Gradient estimate]\label{thm:gradient-estimate}
Let $\lambda \in \R$. A differential structure $(\cA, \nabla, \sigma)$ satisfies $\Ric(\cA, \nabla, \sigma) \geq \lambda$ if and only if the following gradient estimate holds for all $\rho \in \Dens$, $A \in \cA_0$ and $t \geq 0$:
\begin{align}\label{eq:gradient-estimate}
	\| \nabla \cP_t A \|_{\rho}^2 \leq e^{-2 \lambda t} 
		\| \nabla A \|_{\cP_t^\dagger \rho}^2 \ .
\end{align}
\end{theorem}

\begin{proof}
We follow a standard semigroup interpolation argument. 
Clearly, \eqref{eq:gradient-estimate} holds for any $\rho \in \Dens$ if and only if it holds for any $\rho \in \Dens_+$. 

Fix $t > 0$, $\rho \in \Dens_+$ and $A \in \cA_0$, and define $f: [0,t] \to \R$ by 
\begin{align*}
f(s)
 :=   e^{-2 \lambda s} \ip{\cK_{\cP_s^\dagger \rho} \cP_{t-s} A, \cP_{t-s} A}_{L^2(\cA,\tau)} 
	= e^{-2 \lambda s} \| \nabla \cP_{t-s} A \|_{\cP_s^\dagger\rho}^2 \ .
\end{align*} 
Writing $\rho_s = \cP_s^\dagger \rho$ and $A_s = \cP_s A$, it follows by \eqref{eq:derivative-of-K} and Proposition \ref{prop:Hessian-Ent} that
\begin{align*}
	f'(s) & = e^{-2\lambda s} 
	\tau\big[  (\nabla A_{t-s})^* (\cN^{(1)}_{\rho_s, \cL^\dagger \rho_s} 
		+ \cN^{(2)}_{\rho_s, \cL^\dagger \rho_s}) \# \nabla A_{t-s}
	\\& \qquad\qquad 
		- 2 (\nabla \cL A_{t-s})^* \hrho_s \# \nabla A_{t-s} 
		- 2 \lambda (\nabla A_{t-s})^* \hrho_s \# \nabla A_{t-s}   \big]
	\\& = 2 e^{-2\lambda s}  
	\Big( 
			 \Hess_\cK \Ent_\sigma(\rho_s)[A_{t-s}, A_{t-s}] 
			 	- \lambda\tau[A_{t-s} \cK_{\rho_s} A_{t-s}] \Big) \ .
\end{align*}

Assume now that $\Ric(\cA, \nabla, \sigma)$. Applying \eqref{it:Hess} from Theorem \ref{thm:Ric-equiv}, we obtain $f'(s) \geq 0$ for all $s$. This implies that $f(t) \geq f(0)$, which is \eqref{eq:gradient-estimate}. 

To prove the converse, set $g(t) = e^{2\lambda t}\| \nabla \cP_t A \|_{\rho}^2 $ and $h(t) = \| \nabla A \|_{\cP_t^\dagger \rho}^2$. 
Then \eqref{eq:gradient-estimate} implies hat $g(t) \leq h(t)$ for all $t \geq 0$. Since $g(0) = h(0)$, we infer that $g'(0) \leq h'(0)$. 
Since
\begin{align*}
	g'(0) & = 2 \tau\big[ (\nabla \cL A)^* \hrho \# \nabla A\big] + 2 \lambda \|  A \|_{\rho}^2  \ , \\
	h'(0) & = \tau\Big[(\nabla A)^* 
    \big(\cN_{\rho,\cL^\dagger \rho}^{(1)} + \cN_{\rho,\cL^\dagger \rho}^{(2)}\big) \# \nabla A \Big] \ ,
\end{align*}
we obtain $\Hess_\cK \Ent_\sigma(\rho)[A, A]  \geq \lambda \tau[A \cK_\rho A]$ in view of the expression for the Hessian in Proposition \ref{prop:Hessian-Ent}.
\end{proof}

An immediate consequence of a Ricci curvature bound is the following contractivity estimate for the associated semigroup, which was independently proved by Rouz\'e in \cite{Rouze:2019}.

\begin{proposition}[$\lambda$-Contractivity]\label{prop:contr}
If $\Ric(\cA, \nabla, \sigma) \geq \lambda$, then the $\lambda$-contractivity bound
\begin{align*}
	\cW(\cP_t^\dagger \rho_0, \cP_t^\dagger \rho_1) \leq
		e^{-\lambda t} \cW(\rho_0, \rho_1) 
\end{align*}
holds for all $\rho_0, \rho_1 \in \Dens$ and $t \geq 0$. 
\end{proposition}

\begin{proof}
This is a well-known consequence of the evolution variational inequality \eqref{eq:EVI}; see \cite[Proposition 3.1]{DS08}.
\end{proof}

Using the techniques developed in this paper, we can explicitly compute the  Ricci curvature for the depolarizing channel defined in Section \ref{sec:depolarizing-generator}. The result has been obtained independently by Rouz\'e in \cite{Rouze:2019}.

\begin{theorem}[Ricci bound for the depolarizing channel]\label{thm:Ricci-depolarizing}\footnote{As pointed out by Brannan, Gao, and Junge  (\texttt{arxiv:2007.06138v1}), the published curvature bound   $\Ric(\cA, \nabla, \tau) \geq \gamma$ is \emph{incorrect}. 
Here we state the published bound up to a factor $\frac12$ using a minor modification of the original proof.  
On the matrix algebra $\M_n(\C)$, it is possible to improve the curvature bound $\frac{\gamma}{2}$ to  $\frac{\gamma}{2}(1 + \frac1{n})$, using the scalar inequality $\partial_1 \Lambda(a,b) + \partial_2 \Lambda(a,b) \geq 1 \geq \frac1n\Lambda(a,b)$ for $0 < a, b \leq n$.  
}
Let $\gamma > 0$, and let $(\cA, \nabla, \tau)$ be a differential structure for the generator of the depolarizing channel given by $\cL A = \gamma (\tau[A]\one - A)$. Then $\Ric(\cA, \nabla, \tau) \geq \frac{\gamma}{2}$.
\end{theorem}

\begin{proof}
Since $\cL A = \gamma (\tau[A]\one - A)$ and $\partial_j \one = 0$, we have $\partial_j \cL A = - \gamma \partial_j A$, independently of the choice of the operators $\partial_j$. We will show that the result follows from this identity. 

First we note that
\begin{align}\label{eq:Pol1}
 - \tau[( \nabla \cL A)^* \hrho \# \nabla A] 
   = \gamma \tau[( \nabla A)^* \hrho \# \nabla A] 
   = \gamma
  \tau\big[(\nabla A)^*  \Lambda(\rho, \rho) \# (\nabla A) \big]    \ . 
\end{align}
Moreover, since $\partial_1\Lambda(a,b) = \int_0^1 (1-s) a^{-s} b^s \dd s$ we obtain (using the notation from \eqref{eq:geod-weak}),
\begin{align*}
\cN_{\rho,\cL \rho}^{(1),j}
  & = \gamma \delta\Lambda\Big( \widetreellr{\bullet}{\rho}{(\one -  \rho)}{\rho}{\rho} \Big)
  \\& = \gamma\sum_{k,m,p} \frac{\Lambda(\lambda_k, \lambda_p) - \Lambda(\lambda_m, \lambda_p)}{\lambda_k - \lambda_m}
   E_k(\one - \rho) E_m \ot E_p
  \\& = \gamma(\one - \rho) \sum_{k,p} \partial_1\Lambda(\lambda_k, \lambda_p)
   E_k \ot E_p
  \\& = \gamma \big( (\one - \rho)  \ot \one \big)\partial_1\Lambda(\rho, \rho)
\end{align*}
Similarly, we have $\cN_{\rho,\cL \rho}^{(2),j} =  \gamma
 \big( \one \ot (\one - \rho) \big)\partial_2\Lambda(\rho, \rho)$. Using the scalar identity $a\partial_1 \Lambda(a,b) + b\partial_2 \Lambda(a,b) = \Lambda(a,b)$, it follows that 
\begin{align*}
\cN_{\rho,\cL \rho}^{(1),j} + \cN_{\rho,\cL \rho}^{(2),j} = \gamma 
(\partial_1\Lambda + \partial_2\Lambda - \Lambda)(\rho, \rho) \ .
\end{align*}
Therefore, for $\eta =1 , 2$, we obtain using Lemma \ref{lem:Schur-bound},
\begin{equation}\begin{aligned}\label{eq:Pol2}
 \tau\big[(\nabla A)^* 
    \cN_{\rho,\cL^\dagger \rho}^{(\eta)} \# (\nabla A) \big] 
 & = \frac12 \tau\big[(\nabla A)^* 
    \big(\cN_{\rho,\cL^\dagger \rho}^{(1)} + \cN_{\rho,\cL^\dagger \rho}^{(2)}\big) \# (\nabla A) \big] 
\\& = \frac{\gamma}{2} 
  \tau\big[(\nabla A)^* (\partial_1 \Lambda + \partial_2 \Lambda - \Lambda)(\rho, \rho) \# (\nabla A) \big]  \ .
\end{aligned}\end{equation}
Combining \eqref{eq:Pol1} and \eqref{eq:Pol2}, and using that $\partial_1 \Lambda(a,b),  \partial_2 \Lambda(a,b) \geq 0$,
it follows from \eqref{eq:Hess2} that
\begin{align*}
 \Hess_\cK \Ent(\rho)[A, A] 
& =  - \tau[( \nabla \cL A)^* \hrho \# \nabla A] 
 	 +  \tau\big[(\nabla A)^* 
    \cN_{\rho,\cL^\dagger \rho}^{(\eta)} \# (\nabla A) \big] 
\\& = \frac{\gamma}{2} 
  \tau\big[(\nabla A)^* (\Lambda + \partial_1 \Lambda + \partial_2 \Lambda)(\rho, \rho) \# (\nabla A) \big]     
\\& \geq \frac{\gamma}{2} 
  \tau\big[(\nabla A)^* \Lambda(\rho, \rho) \# (\nabla A) \big]     
\\&  = \frac{\gamma}{2}  \tau[(\nabla A)^* \hrho \# \nabla A] 
 = \frac{\gamma}{2}   \ip{\cK_\rho A, A}_{L^2(\tau)} \ ,
\end{align*}
which proves the result. 
\end{proof}

\newpage 

Since the spectral gap of $\cL$ equals $\gamma$, it follows from the results in Section \ref{sec:functional} that the obtained constant is optimal.

\subsection{Geodesic convexity via intertwining}

In this subsection we provide a useful technique for proving Ricci curvature bounds, which has the advantage that it does not require an explicit computation of the Hessian of the entropy. Instead, it relies on the following intertwining property between the gradient and the quantum Markov semigroup.

\begin{definition}[Intertwining property]\label{def:intertwining}
For $\lambda \in \R$, we say that a collection of linear operators $(\vec{\cP_t})_{t \geq 0}$ on $\cB$ is $\lambda$-intertwining for the quantum Markov semigroup $(\cP_t)_{t \geq 0}$, if the following conditions hold:
\begin{enumerate}
\item \label{it:it} For all $A \in \cA$ and $t \geq 0$, we have $\nabla \cP_t A = \vec{\cP_t} \nabla A$;
\item \label{it:it-ineq} For all $\rho \in \Dens_+$, $\bbB = (B_j) \in \cB$ and $t \geq 0$, we have	
\begin{align}
	\label{eq:intertwining}
\sA\big(\rho, \vec{\cP_t^\dagger} \bbB \big)
	     \leq e^{-2\lambda t }  
    	 \sA\big(\rho, (\cP_t^\dagger B_j)_j \big) \ .
\end{align}
\end{enumerate}
\end{definition}

By duality, the intertwining relation \eqref{it:it} implies the identity
\begin{align}\label{eq:intertwining-dual}
\cP_t \dive (\bbA) = \dive(\vec{\cP_t^\dagger} \bbA)\ , \qquad \text{for } \bbA \in \cB.
\end{align}
The following lemma allows us to check the $\lambda$-intertwining property in several examples of interest.

\begin{lemma}\label{lem:intertwining}
Let $\lambda \in \R$, and suppose that $\partial_j \cL A = (\cL - \lambda ) \partial_j A$ for all $A \in \cA$. Then the semigroup $(\vec{\cP_t})_t$ defined by $(\vec{\cP_t} \bbB)_j = e^{-\lambda t} \cP_t B_j$ is $\lambda$-intertwining for the quantum Markov semigroup $(\cP_t)_{t \geq 0}$.
\end{lemma}

\begin{proof}
By spectral theory, the stated condition on the generator is equivalent to the semigroup property $\partial_j \cP_t A = e^{-\lambda t} \cP_t \partial_j A$ for all $t \geq 0$. Thus, the semigroup $(\vec{\cP_t})_t$ satisfies \eqref{it:it} in Definition \ref{def:intertwining}. Since $(\vec{\cP_t^\dagger} \bbB)_j = e^{-\lambda t} \cP_t^\dagger B_j$, condition \eqref{it:it-ineq} follows as well.
\end{proof}

\begin{theorem}[Lower Ricci bound via intertwining]\label{thm:intertwining-convexity}
Let $(\cA, \nabla, \sigma)$ be a differential structure, and let $\lambda \in \R$. If there exists a collection of linear operators $(\vec{\cP_t})_{t \geq 0}$ on $\cB$ that is $\lambda$-intertwining for the associated QMS $(\cP_t)_{t \geq 0}$, then $\Ric(\cA, \nabla, \sigma) \geq \lambda$.
\end{theorem}

\begin{proof}[Proof of Theorem \ref{thm:intertwining-convexity}]
The proof is a variation on an argument by Dolbeault, Nazaret and Savar\'e \cite{DNS09}.

Fix $\bar\rho, \nu \in \Dens$, and let $(\rho_s, \bbB_s)_{s \in [0,1]}$ be a solution to the continuity equation
\begin{align*}
 \partial_s \rho_s + \dive \bbB_s = 0\ , \qquad \rho_0 = \nu\ , \quad \rho_1 = \bar\rho\ ,
\end{align*}
that minimizes the action functional \eqref{eq:BB-form}. This implies that $(\rho_s)_s$ is a constant speed geodesic, and
\begin{align}\label{eq:all-s}
	 \sA(\rho_s, \bbB_{s}) = \cW(\nu, \bar\rho)^2
\end{align}
for all $s \in [0,1]$.
We define 
$ \rho_s^t := \cP_{st}^\dagger \rho_s$, so that $ \partial_s \rho_s^t 
   = \cP_{st}^\dagger (\partial_s \rho_s) - t \cL^\dagger \cP_{st}^\dagger \rho_s$.
Using this identity, we obtain
\begin{align*}
 \partial_s \rho_s^t 
  & = \cP_{st}^\dagger (\partial_s \rho_s) - t \cL^\dagger \cP_{st}^\dagger \rho_s
\\& = - \cP_{st}^\dagger (\dive \bbB_s) - t \cL^\dagger \rho_s^t
\\& = - \dive(\vec{\cP_{st}^\dagger} \bbB_s) - t \cL^\dagger \rho_s^t\ .
\end{align*}
Write $\widetilde\nabla = (\widetilde\partial_j)_j$, where $\widetilde\partial_j = e^{-\omega_j/2}V_j \err_j(\rho)  - e^{\omega_j/2}\ell_j(\rho) V_j$. It then follows from Lemma \ref{lem:chain-log-general} and Theorem \ref{thm:grad-flow-ent} that $\cL^\dagger = \dive \widetilde\nabla$. Hence, we infer that the curve $(\rho_s^t)_{s \in [0,1]}$ satisfies the continuity equation $\partial_s \rho_s^t + \dive \bbB_s^t = 0$, where 
\begin{align*}
 \bbB_s^t = \vec{\cP_{st}^\dagger}\bbB_s - t \widetilde\nabla \rho_s^t\ .
\end{align*}
Using the bilinearity of $\cA(\rho_s^t, \cdot, \cdot)$, we obtain
\begin{equation}
	\begin{aligned}
\label{eq:cW-bound}
\cW(\nu, \cP_t^\dagger \bar\rho)^2 
  & \leq \int_0^1 \sA(\rho_s^t, \bbB_s^t) \; ds
 \\& = \int_0^1 \sA(\rho_s^t, \vec{\cP_{st}^\dagger}\bbB_s) 
  		- 2t \sA(\rho_s^t, \bbB_s^t, \widetilde\nabla \rho_s^t)
  		- t^2 \sA(\rho_s^t, \widetilde\nabla \rho_s^t) \dd s\ .
	\end{aligned}
\end{equation}
Using \eqref{eq:intertwining} and Corollary \ref{cor:contr-CPTC} we infer that
\begin{align*}
	 \sA(\rho_s^t, \vec{\cP_{st}^\dagger}\bbB_s)
& \leq
	e^{-2\lambda st} 
	 \sA\big(\cP_{st}^\dagger \rho_s, (\cP_{st}^\dagger B_{j,s})_j\big)
  \leq 
e^{-2\lambda st} 
	 \sA\big(\rho_s, \bbB_s\big)
\end{align*}
hence \eqref{eq:all-s} yields
\begin{align*}
	\int_0^1 
	 \sA(\rho_s^t, \vec{\cP_{st}^\dagger}\bbB_s)
	\dd s 
	\leq \frac{1 - e^{-2\lambda t}}{2\lambda t} \cW(\nu, \bar\rho)^2 \ ,
\end{align*}
A direct computation using Lemma \ref{lem:chain-log-general} shows that 
\begin{align*}
	\partial_s \Ent_\sigma(\rho_s^t)
	&  = \tau[ (\log \rho_s^t -  \log \sigma) \partial_s \rho_s^t ]
	   = - \tau[(\log \rho_s^t -  \log \sigma) \dive \bbB_s^t ]
	\\&  = \tau[ (\nabla (\log \rho_s^t -  \log \sigma))^* \bbB_s^t ]
  = \tau[ (\lrho_s^t \# \widetilde\nabla \rho_s^t)^* \bbB_s^t ]
  =  \cA(\rho_s^t; \bbB_s^t, \widetilde\nabla \rho_s^t) \ . 
\end{align*}
Estimating the final term in \eqref{eq:cW-bound} by $0$, we infer that 
\begin{align*}
	\frac{1}{2t} \Big(\cW(\nu, \cP_t^\dagger \bar\rho)^2 -  \cW(\nu, \bar\rho)^2\Big)
	\leq \frac1{2t}\Big( \frac{1 - e^{-2\lambda t}}{2\lambda t} - 1\Big) \cW(\nu, \bar\rho)^2 - \int_0^1 \partial_s \Ent_\sigma(\rho_s^t) \dd s\ .
\end{align*} 
Since $t \mapsto \Ent_\sigma(\rho_s^t)$ is continuous, we observe that the right-hand side converges as $t \downarrow 0$.
Letting $t \downarrow 0$ we infer that 
\begin{align*}
	\frac{1}{2}\ddtr\bigg|_{t = 0}\ \cW(\nu, \cP_t \bar\rho)^2 
	\leq \frac{\lambda}{2} \cW(\nu, \bar\rho)^2
	  + \Ent_\sigma(\bar\rho) - \Ent_\sigma(\nu)\ ,
\end{align*} 
which proves the evolutional variational inequality from Theorem \ref{thm:Ric-equiv} for $t = 0$. By the semigroup property, the inequality holds for all $t \geq 0$, hence the result follows.
\end{proof}

\begin{remark}\label{rem:tensorisation}
As pointed out by an anonymous referee, the condition from Lemma \ref{lem:intertwining} is preserved under taking tensor products of quantum Markov semigroups. 
Therefore, Theorem \ref{thm:intertwining-convexity} yields a lower Ricci curvature bound for tensor product semigroups of this type.
It is an interesting open question whether such a tensorisation property holds for arbitrary quantum Markov semigroups, as is known to be true in the Markov chain setting \cite{EM12}.
\end{remark}

We finish the section with the example of the Fermionic Ornstein-Uhlenbeck equation from Section \ref{sec:fermion}, which was already discussed in \cite{2016-Carlen-Maas}. For the convenience of the reader we provide the details.

\begin{proposition}[Intertwining for fermions]\label{prop:fermi-intertwining}
In the fermionic setting, we have the commutation relations
$
	[\partial_j, \cL] = - \partial_j 
$
for $j = 1, \ldots, n$.
Consequently, the intertwining property holds with $\lambda = 1$.
\end{proposition}

\begin{proof}
	We use the well-known fact that the differential operator $\partial_j$ is the \emph{annihilation operator}: it maps the $k$-particle space $\cH^k$ into the $(k-1)$-particle space $\cH^{k-1}$ for any $0 \leq k \leq n$ (with the convention that $\cH^{-1} = \{0\})$. On the other hand, $-\cL$ is the \emph{number operator}, which satisfies $\cL A = - k A$ for all $A \in \cH^k$. Hence, for $A \in \cH^{k}$, we have $\partial_j \cL A = - k \partial_j A$, whereas $\cL\partial_j A = - (k-1)\partial_j A$. This yields the desired commutation relation $[\partial_j, \cL] = - \partial_j$ on $\cH^{k}$, which extends to $\Cln$ by linearity.
The result thus follows from Lemma \ref{lem:intertwining}.
\end{proof}

We immediately obtain the following result.

\begin{corollary}
The differential structure for the fermionic Ornstein-Uhlenbeck equation in Section \ref{sec:fermion} satisfies $\Ric(\Cln, \nabla, \tau) \geq 1$ in any dimension $n \geq 1$. 
\end{corollary}

It follows from the results in the following section that the constant $1$ is optimal.

\section{Functional inequalities}\label{sec:functional}

One of the advantages of the framework of this paper is that it allows one to prove a sequence of implications between several useful functional inequalities. Throughout this section we assume that $(\cP_t)_t$ is ergodic.

Recall that 
\begin{align*}
 \Ent_\sigma(\rho) := \Tr[\rho ( \log \rho - \log \sigma) ]\ ,\qquad
  \cI_\sigma(\rho) := -\Tr[(\log \rho - \log \sigma)\cL^\dagger\rho ]
\ ,
\end{align*}
and note that $\ddt \Ent_\sigma(\cP_t^\dagger \rho) = - \cI_\sigma(\cP_t^\dagger \rho)$ for $\rho \in \Dens_+$.
The quantity $\cI_\sigma$ is a quantum version of the Fisher information (or entropy production) relative to $\sigma$; we refer to \cite{Petz-Ghinea:2011} for an introduction to several notions of Fisher information in the quantum setting.

The  gradient flow structure from Theorem \ref{thm:grad-flow-ent} implies that $\cL^\dagger \rho = \dive (\hrho \#\nabla (\log \rho - \log \sigma) )$, which yields $\cI_\sigma(\rho) = \|\nabla(\log \rho - \log \sigma)\|_\rho^2$. Recall that for  $\rho \in \Dens $ and $A \in \cA$ we denote the associated Bogolioubov--Kubo--Mori scalar product and norm by
\begin{align*}
\ip{A, B}_{L^2_{\rm BKM}(\rho)} =  \int_0^1 \tau \big[A^* \rho^{1-s} B \rho^s \big] \dd s 
 \ , \qquad 
 \|A\|_{L^2_{\rm BKM}(\rho)}
   = \sqrt{\ip{A, A}_{ BKM} } \ .
\end{align*}

The results presented in this section have been obtained in the classical discrete setting of finite Markov chains in \cite{EM12}, and in the setting of Lindblad operators in \cite{RD17}. Here we state and prove the results in the more general framework that includes arbitrary differential structures $(\cA, \nabla, \sigma)$. 
The proofs closely follow the original arguments by Otto and Villani \cite{OV00}, which were adapted in \cite{EM12,RD17}. In our finite-dimensional setting, most of the results follow directly from Riemannian considerations, though some additional care is needed due to the degeneracy of the metric at the boundary $\Dens \setminus \Dens_+$.

\begin{definition}\label{def:MLSI}
  A differential structure $(\cA, \nabla, \sigma)$ satisfies 
\begin{enumerate}
\item  a \emph{modified logarithmic Sobolev inequality}  with constant $\lambda>0$ if for all $\rho\in\PX$,
  \begin{align}
  \tag{MLSI($\lambda$)}
  \label{eq:MLSI}
    \Ent_\sigma(\rho) \leq \frac{1}{2\lambda}\cI_\sigma(\rho)\ .
  \end{align}

\item an \emph{H$\cW$I inequality}  with
  constant $\kappa\in\R$ if for all
  $\rho\in\PX$,
  \begin{align}\label{eq:HWI}
  \tag{{H$\cW$I}($\kappa$)}
    \Ent_\sigma(\rho) \leq \cW(\rho,\sigma)\sqrt{\cI_\sigma(\rho)}-\frac{\kappa}{2}\cW(\rho,\sigma)^2\ .
  \end{align}

\item a \emph{modified Talagrand inequality}
   with constant $\lambda > 0$ if for all $\rho\in\Dens$,
  \begin{align}\label{eq:T}
  \tag{{T}$_{\cW}$($\lambda$)}
    \cW(\rho,\sigma) \leq \sqrt{\frac{2}{\lambda}\Ent_\sigma(\rho)}\ .
  \end{align}

\item a \emph{$T_1$-transport inequality}
   with constant $\lambda > 0$ if for all $\rho\in\Dens$,
  \begin{align}\label{eq:T-one}
  \tag{{T}$_{1}$($\lambda$)}
    W_1(\rho,\sigma) \leq \sqrt{\frac{2}{\lambda}\Ent_\sigma(\rho)}\ .
  \end{align}

\item a \emph{Poincar\'e inequality} (or \emph{spectral gap inequality}) with constant $\lambda > 0$ 
  if for all $A \in \cA_h$ with $\tau[\int_0^1 \sigma^{1-s} A \sigma^s \dd s] = 0$,
  \begin{align}\label{eq:P}
  \tag{{P}($\lambda$)}
 \|A\|_{L^2_{\rm BKM}(\sigma)}^2
 \leq \frac{1}{\lambda}\| \nabla A\|_\sigma^2  \ .
  \end{align}
\end{enumerate}  
\end{definition}

It is well known and an easy consequence of Gronwall's inequality, that \ref{eq:MLSI} is equivalent to the exponential decay of the entropy with rate $2\lambda$:
\begin{align}\label{eq:ent-rate}
  \Ent_\sigma(\cP_t^\dagger \rho) \leq e^{-2\lambda t} \Ent_\sigma(\rho)\ .
\end{align} 
There are other approaches to some of these inequalites and variants of them; see, e.g., 
 \cite{Bardet:2017,Bardet-Rouze:2018,CL1,KT12,PWPR06}. 

Recall that for an absolutely continuous curve $(\rho_t)_t \in (\Dens, \cW)$, its \emph{metric derivative} 
\begin{equation*}
\abs{\rho_t'} := \lim_{h\to0}\frac{\cW(\rho_{t+h},\rho_t)}{\abs{h}}
\end{equation*}
exists for a.e. $t\in[0,T]$; see \cite[Theorem 1.1.2]{AGS08}.

\begin{proposition}\label{prop:HW-derivative}
  Let $\rho,\nu\in\Dens_+$. For all $t \geq 0$ we have
\begin{align}
  \label{eq:W-derivative}
  \ddtr\cW(\cP_t^\dagger \rho,\nu)  \leq \sqrt{\cI_\sigma(\cP_t^\dagger\rho)}\;.
\end{align}
In particular, the metric derivative of the heat flow with respect to $\cW$ satisfies $\abs{(\cP_t^\dagger\rho)'} \leq\sqrt{\cI_\sigma(\cP_t^\dagger\rho)}$.
\end{proposition}

\begin{proof}
  Set $\rho_t:=\cP_t^\dagger\rho$. 
  Using the triangle inequality for $\cW$  we obtain
  \begin{align*}
    \ddtr\cW(\rho_t,\nu) 
     &= 
    \limsup_{s\downarrow 0}\frac{1}{s}\big(\cW(\rho_{t+s},\nu)-\cW(\rho_t,\nu)\big)
     \\ & 
     \leq \limsup_{s\downarrow 0} \frac{1}{s}\cW(\rho_t,\rho_{t+s})\ .
  \end{align*}
  In view of the gradient flow identity $\partial_t \rho = \dive (\hrho \#\nabla (\log \rho - \log \sigma) )$, the definition of $\cW$ yields  
\begin{align*}
    \limsup_{s\downarrow 0}\frac{1}{s} \cW(\rho_t,\rho_{t+s}) 
    & \leq
    \limsup_{s\downarrow 0}\frac{1}{s}\int_t^{t+s} 
    	\| \nabla( \log\rho_r - \log \sigma ) \|_{\rho_r}   \dd r \\
    &= \limsup_{s\downarrow 0}\frac{1}{s}\int\limits_t^{t+s}\sqrt{\cI_\sigma(\rho_r)}\dd r \\
    &= \sqrt{\cI_\sigma(\rho_t)}\ .
\end{align*}
  The last equality follows from the  continuity of $r\mapsto \sqrt{\cI_\sigma(\rho_r)}$.
\end{proof}

The following result is a non-commutative analogue of a well-known result by Otto and Villani \cite{OV00}.

\begin{theorem}\label{thm:Ric2HWI}
  Assume that $\Ric(\cA,\nabla,\sigma)\geq\kappa$ for some $\kappa\in\R$. Then
  $\HcWI(\kappa)$ holds as well.
\end{theorem}
\begin{proof}
Fix $\rho\in\Dens$. 
If $\cI_\sigma(\rho)=+\infty$ there is nothing to prove, so we will assume without loss of generality that $\rho \in \Dens_+$.
Set $\rho_t :=\cP_t^\dagger\rho$. From Theorem \ref{thm:Ric-equiv} and the lower bound on the Ricci curvature we know that the curve $(\rho_t)$ satisfies EVI($\kappa$), i.e., equation \eqref{eq:EVI}. 
Choosing $\nu=\sigma$ and $t = 0$ in the EVI($\kappa$) yields
  \begin{align*}
    \Ent_\sigma(\rho) \leq -\frac12 \ddtrz \cW(\rho_t,\sigma)^2 -\frac{\kappa}{2}\cW(\rho,\sigma)^2\ .
  \end{align*}
It remains to show that
  \begin{align*}
    -\frac12 \ddtrz\cW(\rho_t,\sigma)^2~\leq~\cW(\rho,\sigma) \sqrt{\cI_\sigma(\rho)}\ .
  \end{align*}
To see this, we use the triangle inequality to estimate
  \begin{align*}
    -\frac12 \ddtrz \cW(\rho_t,\sigma)^2 
    &= \liminf_{t\downarrow 0}\frac{1}{2t}
    	\left(\cW(\rho,\sigma)^2-\cW(\rho_{t},\sigma)^2\right)
    \\& \leq \limsup_{t\downarrow 0}\frac{1}{2t}
    		\left(\cW(\rho,\rho_t)^2 
			+ 2\cW(\rho,\rho_t)\cdot \cW(\rho,\sigma)\right)\ ,
  \end{align*}
  Using Proposition \ref{prop:HW-derivative} with $\nu = \rho$ and $t=0$ we see that
  the second term on the right-hand side is bounded by $\cW(\rho,\sigma) \sqrt{\cI_\sigma(\rho)}$, while the first term vanishes.
\end{proof}

The following result is now a simple consequence.

\begin{theorem}[Quantum Bakry--\'{E}mery Theorem]\label{thm:Ric2LSI}
Suppose that $\Ric(\cA,\sigma,\nabla)\geq\lambda$ for some $\lambda>0$. Then the modified logarithmic Sobolev inequality  $\mLSI(\lambda)$ holds.
\end{theorem}
\begin{proof}
Take $\rho \in \Dens_+$. It follows from Theorem \ref{thm:Ric2HWI} that $(\cA,\sigma,\nabla)$ satisfies $\HcWI(\lambda)$. Using this inequality followed by Young's inequality we obtain
\begin{align*}
 \Ent_\sigma(\rho) \leq \cW(\rho,\sigma) \sqrt{\cI_\sigma(\rho)} - \frac{\lambda}{2} \cW(\rho,\sigma)^2  \leq \frac{1}{2\lambda}\cI_\sigma(\rho) \ ,
\end{align*} 
which is $\mLSI(\lambda)$.
\end{proof}

\begin{theorem}[Quantum Otto--Villani Theorem]\label{thm:LSI2T}
Suppose that the differential structure $(\cA,\nabla,\sigma)$ satisfies $\mLSI(\lambda)$ for some $\lambda>0$. 
Then the Talagrand inequality $\mTal(\lambda)$ holds as well.
\end{theorem}
\begin{proof}
It suffices to prove $\mTal(\lambda)$ for $\rho \in \Dens_+$, since the inequality for general $\rho \in \Dens$ can then be obtained by approximation.

Fix $\rho \in \Dens_+$ and set $\rho_t = \cP_t^\dagger \rho$. As  $t\to\infty$, we use \eqref{eq:ent-rate} to infer that
  \begin{align}\label{eq:asymptotic}
    \Ent_\sigma(\rho_t)\to 0 \quad\mbox{ and }\quad \cW(\rho,\rho_t)\to
    \cW(\rho,\sigma)\;.
  \end{align}  
Define $F:\R_+\to\R_+$ by
  \begin{align*}
    F(t) := \cW(\rho,\rho_t) + \sqrt{\frac{2}{\lambda}\Ent_\sigma(\rho_t)}\;.
  \end{align*}
We have $F(0) = \sqrt{\frac{2}{\lambda}\Ent_\sigma(\rho)}$ and $F(t)\to\cW(\rho,\sigma)$ as $t\to\infty$ by \eqref{eq:asymptotic}. Hence it is sufficient to show that $\ddtr F(t) \leq 0$ for all $t \geq 0$. If $\rho_t\neq \sigma$, we use Proposition
  \ref{prop:HW-derivative} and the identity $\ddt \Ent_\sigma(\rho_t) = - \cI_\sigma(\rho_t)$ to obtain
  \begin{align*}
    \ddtr F(t) 
    	\leq \sqrt{\cI_\sigma(\rho_t)} 
		- \frac{\cI_\sigma(\rho_t)}{\sqrt{2\lambda\Ent_\sigma(\rho_t)}} 
		 \leq 0\ ,
  \end{align*}
where the last inequality follows from $\mLSI(\lambda)$. 
If $\rho_t = \sigma$, then the same inequality holds, since this implies that $\rho_r = \sigma$ for all $r\geq t$.  
\end{proof}

It is known that the modified logarithmic Sobolev inequality implies a Poincar\'e inequality by a linearization argument. 
The following result shows that Poincar\'e inequality is in fact implied by the Talagrand inequality, which is weaker than the MLSI in view of the previous theorem.
The BKM metric in the left-hand side of \ref{eq:P} appears since it also appears in the second order expansion of the relative entropy of $\Ent_\sigma(\rho)$ around $\rho = \sigma$; see \eqref{eq:Ent-expansion}.

\begin{proposition}\label{prop:T2P}
 Assume that the triple $(\cA,\sigma,\nabla)$ satisfies \ref{eq:T} for some $\lambda>0$. Then the Poincar\'e inequality \ref{eq:P} and the  $T_1$-transport inequality \ref{eq:T-one} hold as well. 
Moreover, $\Ric(\cA,\sigma,\nabla) \geq \lambda$ implies $\Poinc(\lambda)$.
\end{proposition}

\begin{proof}
The fact that \ref{eq:T} implies the $T_1$-inequality is an immediate consequence of Proposition \ref{prop:Wass-1-comparison}. 

Suppose that \ref{eq:T} holds and let us show $\Poinc(\lambda)$.
Fix $\nu \in \cA_0$ and set $\rho^\eps: = \sigma + \eps \nu$. 
Then $\rho^\eps \in \Dens_+$ for sufficiently small $\eps>0$.
For such $\eps > 0$, let $(\rho_t^\eps, \bbB_t^\eps)_t$ be an action minimizing curve connecting $\rho_0^\eps = \rho^\eps$ and $\rho_1^\eps = \sigma$. Thus we have $\partial_t \rho_t^\eps + \dive(\hrho_t^\eps \# \bbB_t^\eps) = 0$ and $\int_0^1 \tau[(\bbB_t^\eps)^*\hrho_t^\eps \# \bbB_t^\eps] \dd t = \cW(\rho^\eps, \sigma)^2$. 

Write $A = \int_0^\infty (x + \sigma)^{-1} \nu (x + \sigma)^{-1} \dd x$ so that $\nu = \int_0^1 \sigma^{1-s} A \sigma^s \dd s$.
Using the continuity equation we obtain
\begin{align*}
   \|A\|_{L^2_{\rm BKM}(\sigma)}^2
    &  = \frac{1}{\eps} \tau[A^* (\rho^\eps - \sigma) ]
= \frac{1}{\eps}
        		\tau[A^* \dive(\hrho_t^\eps \# \bbB_t^\eps) ]   
		= -\frac{1}{\eps} \int_0^1 
    	 \tau[ (\nabla A)^* \hrho_t^\eps \# \bbB_t^\eps ] \dd t\ .
\end{align*}
The Cauchy-Schwarz inequality yields
\begin{align*}
 \|A\|_{L^2_{\rm BKM}(\sigma)}^2
    & \leq \frac{1}{\eps}
    \bigg( \int_0^1 
    	 \| \nabla A \|_{\rho_t^\eps}^2 \dd t \bigg)^{1/2}
        \bigg( \int_0^1 
    	 \| \bbB_t^\eps\|_{\rho_t^\eps}^2  \dd t \bigg)^{1/2}
 \\ &= \frac{1}{\eps} \bigg( \int_0^1 
    	 \| \nabla A \|_{\rho_t^\eps}^2 \dd t \bigg)^{1/2}
  \cW(\rho^\eps,\sigma) \ ,
\end{align*}
since $(\rho_t^\eps)_t$ is a $\cW$-geodesic.
Using $\mTal(\lambda)$ we obtain
  \begin{align*}
 \limsup_{\eps \to 0} \frac{\cW(\rho^\eps,\sigma)}{\eps}
   & \leq  \limsup_{\eps \to 0} \frac{1}{\eps}\sqrt{\frac{2}{\lambda}\Ent_\sigma(\rho^\eps)} \leq \frac{1}{\sqrt{\lambda}} \|A\|_{L^2_{\rm BKM}(\sigma)}\ ,
  \end{align*}
since $\Ent_\sigma(\rho^\eps) = \frac12 \eps^2 \|A\|_{L^2_{\rm BKM}(\sigma)}^2 + o(\eps^2)$ by \eqref{eq:Ent-first} and \eqref{eq:Ent-expansion}.
It remains to show that, as $\eps \to 0$,
\begin{align*}
 \int_0^1 \| \nabla A \|_{\rho_t^\eps}^2 \dd t \to  \| \nabla A \|_{\sigma}^2 \ .
 \end{align*}
To see this, note that $\tau[|\rho^\eps - \sigma|] \to 0$, hence $\cW(\rho^\eps, \sigma) \to 0$. Since $\cW(\rho_t^\eps, \sigma) = (1-t) \cW(\rho^\eps, \sigma)$, it follows that $\cW(\rho_t^\eps, \sigma) \to 0$ as $\eps \to 0$ for all $t \in [0,1]$, which implies that $\| \nabla A \|_{\rho_t^\eps}^2 \to \| \nabla A \|_{\sigma}^2$ for all $t \in [0,1]$. The result now follows using dominated convergence, since
$\| \nabla A \|_{\rho_t^\eps}^2 \leq \|\nabla A\|_\cB$ by Lemma \ref{lem:Lip-bound}.

The final assertion of the proposition follows by combining this result with Theorem
  \ref{thm:Ric2LSI} and Theorem \ref{thm:LSI2T}.
\end{proof}

\subsection*{Acknowledgement} {\small E.C. gratefully acknowledges support through NSF grant DMS-174625. J.M. gratefully acknowledges support by the European Research Council (ERC) under the European Union's Horizon 2020 research and innovation programme (grant agreement No 716117), and by the Austrian Science Fund (FWF), Project SFB F65.
We are grateful to the anonymous referees for carefully reading the original manuscript and making useful comments.}

\bibliographystyle{plain}

\bibliography{CarlenMaas.bib}

 \end{document}